\documentclass[leqno,11pt,twoside]{amsart}
\usepackage[body={6.5in, 8.5in},left=1.3in,right=1.3in]{geometry}

\usepackage{times}
\usepackage{todonotes}

\usepackage[all]{xy} 
\usepackage{amsmath, amssymb, amsfonts, latexsym, mdwlist, amsthm}
\usepackage{subfig}
\usepackage{graphicx}
\usepackage{wrapfig}

\usepackage{url}

\usepackage[bookmarks, colorlinks, breaklinks, pdftitle={MMP for Moduli of Sheaves on K3s via Wall-Crossing},
pdfauthor={Arend Bayer, Emanuele Macri}]{hyperref}
\hypersetup{linkcolor=blue,citecolor=blue,filecolor=black,urlcolor=blue}

\usepackage{pgf, tikz}
\usetikzlibrary{arrows,chains,shapes.geometric,%
    decorations.pathreplacing,decorations.pathmorphing,shapes,%
    matrix,shapes.symbols,patterns}

\tikzset{
>=stealth',
  punktchain/.style={
    rectangle,
    rounded corners,
    draw=black, thick,
    minimum height=3em,
    text centered,
    on chain},
  line/.style={draw, thick, <-},
  element/.style={
    tape,
    top color=white,
    bottom color=blue!50!black!60!,
    minimum width=8em,
    draw=blue!40!black!90, very thick,
    text width=10em,
    minimum height=3.5em,
    text centered,
    on chain},
  every join/.style={->, thick,shorten >=1pt},
  decoration={brace},
  tuborg/.style={decorate},
  tubnode/.style={midway, right=2pt},
}

\usepackage{paralist}
\setdefaultenum{(a)}{(i)}{}{}
\usepackage{enumitem} 

\def\C{\ensuremath{\mathbb{C}}}

\def\P{\ensuremath{\mathbb{P}}}
\def\Q{\ensuremath{\mathbb{Q}}}
\def\R{\ensuremath{\mathbb{R}}}
\def\Z{\ensuremath{\mathbb{Z}}}

\def\alg{\mathrm{alg}}
\def\Aut{\mathop{\mathrm{Aut}}\nolimits}

\def\Bigc{\mathop{\mathrm{Big}}}
\def\ch{\mathop{\mathrm{ch}}\nolimits}

\def\Coh{\mathop{\mathrm{Coh}}\nolimits}

\def\dim{\mathop{\mathrm{dim}}\nolimits}

\def\Eff{\mathop{\mathrm{Eff}}}

\def\ext{\mathop{\mathrm{ext}}\nolimits}
\def\Ext{\mathop{\mathrm{Ext}}\nolimits}

\def\GL{\mathop{\mathrm{GL}}}

\def\Hom{\mathop{\mathrm{Hom}}\nolimits}

\def\RlHom{\mathop{\mathbf{R}\mathcal Hom}\nolimits}
\def\RHom{\mathop{\mathbf{R}\mathrm{Hom}}\nolimits}

\def\Imm{\mathop{\mathrm{Im}}\nolimits}

\def\Ker{\mathop{\mathrm{Ker}}\nolimits}
\def\Cok{\mathop{\mathrm{Cok}}\nolimits}

\def\Mov{\mathop{\mathrm{Mov}}\nolimits}

\def\min{\mathop{\mathrm{min}}\nolimits}

\def\Nef{\mathop{\mathrm{Nef}}\nolimits}
\def\NS{\mathop{\mathrm{NS}}\nolimits}

\def\Pos{\mathop{\mathrm{Pos}}}

\def\rk{\mathop{\mathrm{rk}}}

\def\ST{\mathop{\mathrm{ST}}\nolimits}

\def\Sym{\mathop{\mathrm{Sym}}}

\def\tr{\mathop{\mathrm{tr}}\nolimits}
\def\td{\mathop{\mathrm{td}}\nolimits}

\def\MG13{\ensuremath{{\mathcal M}_{\Gamma_1(3)}}}
\def\tildeMG13{\ensuremath{\widetilde{\mathcal M}_{\Gamma_1(3)}}}
\def\Stab{\mathop{\mathrm{Stab}}\nolimits}
\def\into{\ensuremath{\hookrightarrow}}
\def\onto{\ensuremath{\twoheadrightarrow}}

\def\blank{\underline{\hphantom{A}}}

\def\Db{\mathrm{D}^{b}}
\def\Dqc{\mathrm{D}_{\mathrm{qc}}}

\newcommand\TFILTB[3]{%
\xymatrix@=1pc{
{0 = {#1}_0} \ar[rr]&&
{{#1}_1} \ar[rr]\ar[ld] &&
{{#1}_2} \ar[r]\ar[ld] &
{\cdots} \ar[r] & { {#1}_{#3-1}} \ar[rr] &&
{{#1}_{#3} = {#1}} \ar[ld]
\\
& *{{#2}_1} \ar@{.>}[ul] &&
{{#2}_2} \ar@{.>}[ul] & &&&
{{#2}_{{#3}}} \ar@{.>}[ul]
}}

\def\abs#1{\left\lvert#1\right\rvert}

\newcommand\stv[2]{\left\{#1\,\colon\,#2\right\}}

\makeatletter
\newtheorem*{rep@theorem}{\rep@title}
\newcommand{\newreptheorem}[2]{%
\newenvironment{rep#1}[1]{%
 \def\rep@title{#2 \ref{##1}}%
 \begin{rep@theorem}}%
 {\end{rep@theorem}}}
\makeatother

\newtheorem{Thm}{Theorem}[section]
\newreptheorem{Thm}{Theorem}
\newtheorem{Prop}[Thm]{Proposition}
\newtheorem{PropDef}[Thm]{Proposition and Definition}
\newtheorem{Lem}[Thm]{Lemma}
\newtheorem{Cor}[Thm]{Corollary}
\newreptheorem{Cor}{Corollary}
\newtheorem{Con}[Thm]{Conjecture}

\newreptheorem{Con}{Conjecture}

\newtheorem{thm-int}{Theorem}

\theoremstyle{definition}
\newtheorem{Def-s}[Thm]{Definition}
\newtheorem{Def}[Thm]{Definition}
\newtheorem{Rem}[Thm]{Remark}

\newtheorem{Ex}[Thm]{Example}

\def\C{\ensuremath{\mathbb{C}}}

\def\P{\ensuremath{\mathbb{P}}}
\def\Q{\ensuremath{\mathbb{Q}}}
\def\R{\ensuremath{\mathbb{R}}}
\def\Z{\ensuremath{\mathbb{Z}}}

\def\AA{\ensuremath{\mathcal A}}

\def\CC{\ensuremath{\mathcal C}}
\def\DD{\ensuremath{\mathcal D}}
\def\EE{\ensuremath{\mathcal E}}
\def\FF{\ensuremath{\mathcal F}}

\def\HH{\ensuremath{\mathcal H}}

\def\LL{\ensuremath{\mathcal L}}

\def\OO{\ensuremath{\mathcal O}}
\def\PP{\ensuremath{\mathcal P}}

\def\TT{\ensuremath{\mathcal T}}

\def\WW{\ensuremath{\mathcal W}}

\def\ZZ{\ensuremath{\mathcal Z}}

\def\aa{\ensuremath{\mathbf a}}
\def\bb{\ensuremath{\mathbf b}}
\def\ss{\ensuremath{\mathbf s}}
\def\tt{\ensuremath{\mathbf t}}
\def\uu{\ensuremath{\mathbf u}}
\def\vv{\ensuremath{\mathbf v}}
\def\ww{\ensuremath{\mathbf w}}

\def\MMM{\mathfrak M}
\def\PPP{\mathfrak P}

\def\SSS{\mathfrak S}
\def\UUU{\mathfrak U}
\def\VVV{\mathfrak V}

\newcommand{\ignore}[1]{}

\def\Halg{H^*_{\alg}(X, \Z)}
\def\Halgalpha{H^*_{\alg}(X,\alpha, \Z)}

\begin{document}

\title[MMP for moduli of sheaves on K3s via wall-crossing]{MMP for moduli 
of sheaves on K3s via wall-crossing:\\ nef and movable cones, Lagrangian fibrations}

\author{Arend Bayer}
\address{School of Mathematics,
The University of Edinburgh,
James Clerk Maxwell Building,
The King's Buildings, Mayfield Road, Edinburgh, Scotland EH9 3JZ,
United Kingdom}
\email{arend.bayer@ed.ac.uk}
\urladdr{http://www.maths.ed.ac.uk/~abayer/}

\author{Emanuele Macr\`i}
\address{Department of Mathematics, The Ohio State University, 231 W 18th Avenue, Columbus, OH 43210-1174, USA}
\email{macri.6@math.osu.edu}
\urladdr{http://www.math.osu.edu/~macri.6/}



\begin{abstract}
We use wall-crossing with respect to Bridgeland stability conditions to systematically study
the birational geometry of a moduli space $M$ of stable sheaves on a K3 surface $X$:
\begin{enumerate}
\item We describe the nef cone, the movable cone, and the effective cone of
$M$ in terms of the Mukai lattice of $X$.
\item We establish a long-standing conjecture that predicts the existence of a birational Lagrangian
fibration on $M$ whenever $M$ admits an integral divisor class $D$ of square zero (with respect to
the Beauville-Bogomolov form).
\end{enumerate}
These results are proved using a natural map from the space of Bridgeland stability conditions
$\Stab(X)$ to the cone $\Mov(X)$ of movable divisors on $M$; this map relates wall-crossing in
$\Stab(X)$ to birational transformations of $M$. In particular, every minimal model of $M$ appears
as a moduli space of Bridgeland-stable objects on $X$.
\end{abstract}


\maketitle

\vspace{-15pt}
\setcounter{tocdepth}{1}
\tableofcontents

\vspace{-15pt}

A moduli space of Bridgeland stable objects automatically comes equipped with a numerically positive
determinant line bundle, depending only on the stability condition \cite{BM:projectivity}.
This provides a direct link between wall-crossing for stability conditions and birational
transformations of the moduli space.  In this paper, we exploit this link to systematically
study the birational geometry of moduli spaces of Gieseker stable sheaves on K3 surfaces.

\section{Introduction}\label{sec:intro}

\subsection*{Overview}
Let $X$ be a projective K3 surface $X$, and $\vv$ a primitive algebraic class in the Mukai lattice
with self-intersection with respect to the Mukai pairing $\vv^2 > 0$. For a generic
polarization $H$, the moduli space $M_H(\vv)$
of $H$-Gieseker stable sheaves is a projective holomorphic symplectic manifold (hyperk\"ahler
variety) deformation equivalent to Hilbert schemes of points on K3 surfaces.
The cone theorem and the minimal model program (MMP) induce a locally polyhedral chamber
decomposition of the movable cone of $M_H(\vv)$ (see \cite{HassettTschinkel:MovingCone}):
\begin{itemize*}
\item chambers correspond one-to-one to smooth $K$-trivial
birational models $\widetilde M \dashrightarrow M_H(\vv)$ of the moduli space, as the minimal model of the
pair $(M_H(\vv), D)$ for any $D$ in the corresponding chamber, and 
\item walls correspond to
extremal Mori contractions, as the canonical model of $(M_H(\vv), D)$.
\end{itemize*}
It is a very interesting question to  understand this chamber decomposition for general
hyperk\"ahler varieties \cite{HassettTschinkel:RationalCurves, HassettTschinkel:ExtremalRays,
HassettTschinkel:MovingCone}.  It has arguably become even more important in light of
Verbitsky's recent proof \cite{Verbitsky:torelli} of a global Torelli statement: 
two hyperk\"ahler varieties $X_1, X_2$ are isomorphic if and only if there exists an
isomorphism of integral Hodge structures $H^2(X_1) \to H^2(X_2)$ that is
induced by parallel transport in a family, and that maps the nef cone of $X_1$ to
the nef cone of $X_2$ (see also \cite{Huybrechts:torrelliafterVerbitsky, Eyal:survey}).

In addition, following the recent success \cite{BCHM} of MMP for the log-general case,
there has been enormous interest to relate MMPs for moduli spaces
to the underlying moduli problem; we refer \cite{Maksym-David:survey} for a survey of the case of the
moduli space $\overline{M}_{g, n}$ of stable curves, known as the Hassett-Keel program. Ideally, one
would like a moduli interpretation for every chamber of the base locus decomposition of the movable
or effective cone.

On the other hand, in \cite{Bridgeland:K3} Bridgeland described a connected component
$\Stab^\dag(X)$ of the
space of stability conditions on the derived category of $X$. He showed that
$M_H(\vv)$ can be recovered as the moduli space $M_\sigma(\vv)$ of $\sigma$-stable objects for
$\sigma \in \Stab^\dag(X)$ near the ``large-volume limit''. The manifold
$\Stab^\dag(X)$ admits a chamber decomposition, depending on $\vv$, such that
\begin{itemize*}
\item for a chamber $\CC$, the moduli space $M_\sigma(\vv) =:M_\CC(\vv)$ is independent of the choice of
$\sigma \in \CC$, and
\item walls consist of stability conditions with strictly semistable objects of class $\vv$.
\end{itemize*}

The main result of our article, Theorem \ref{thm:MAP}, relates these two
pictures directly. It shows that any MMP for the Gieseker moduli space (with movable boundary) can be
induced by wall-crossing for Bridgeland stability conditions, and so any minimal model has an
interpretation as a moduli space of Bridgeland-stable objects for some chamber.
In Theorem \ref{thm:nefcone}, we deduce the chamber decomposition of the movable
cone of $M_H(\vv)$ in terms of the Mukai lattice of $X$ from a description of the chamber
decomposition of $\Stab^\dag(X)$, given by Theorem \ref{thm:walls}.

We also obtain the proof of a long-standing conjecture: the existence of a
birational Lagrangian fibration $M_H(v) \dashrightarrow \P^n$ is equivalent to the existence of an integral
divisor class $D$ of square zero with respect to the Beauville-Bogomolov form, see Theorem
\ref{thm:SYZ}.  We use birationality of wall-crossing and a Fourier-Mukai transform to reduce the conjecture to
the well-known case of a moduli space of torsion sheaves, studied in
\cite{Beauville:ACIS}.  Further applications are mentioned below.

\subsection*{Birationality of wall-crossing and the map to the movable cone}
Let $\sigma, \tau \in \Stab^\dag(X)$ be two stability conditions, and assume that they are
\emph{generic} with respect to $\vv$. By \cite[Theorem 1.3]{BM:projectivity}, the moduli spaces
$M_{\sigma}(\vv)$ and $M_{\tau}(\vv)$ of stable objects $\EE \in \Db(X)$ with Mukai vector $\vv(\EE)
= \vv$ exist as smooth projective varieties. Choosing a path from $\sigma$ to $\tau$ in $\Stab^\dag(X)$
relates them by a series of wall-crossings. Based on a detailed analysis of all
possible wall-crossings, we prove:

\begin{Thm} \label{thm:birational-WC}
Let $\sigma,\tau$ be generic stability conditions with respect to $\vv$.
\begin{enumerate}
\item The two moduli spaces $M_\sigma(\vv)$ and $M_\tau(\vv)$ of Bridgeland-stable objects are
birational to each other.
\item \label{enum:birationalautoequivalence}
More precisely, there is a birational map induced by a derived (anti-)autoequivalence
$\Phi$ of $\Db(X)$ in the following
sense: there exists a common open subset $U \subset M_\sigma(\vv)$, $U \subset M_\tau(\vv)$, 
with complements of codimension at least two, such that
for any $u \in U$, the corresponding objects $\EE_u \in M_\sigma(\vv)$ and
$\FF_u \in M_\tau(\vv)$ are related via
$\FF_u = \Phi(\EE_u)$.
\end{enumerate}
\end{Thm}
An anti-autoequivalence is an equivalence from the opposite category $\Db(X)^{\mathrm{op}}$ to
$\Db(X)$, for example given by the local dualizing functor $\RlHom(\blank, \OO_X)$.

As a consequence, we can canonically identify the N\'eron-Severi groups of $M_{\sigma}(\vv)$
and $M_\tau(\vv)$.
Now consider the chamber decomposition of $\Stab^\dag(X)$
with respect to $\vv$ as above, and let $\CC$ be a chamber.  The main result of \cite{BM:projectivity} gives a natural map
\begin{equation} \label{eq:ellCC}
\ell_\CC \colon \CC \to \NS\left(M_{\CC}(\vv)\right)
\end{equation}
to the N\'eron-Severi group of the moduli space, whose image is contained in the ample cone of
$M_{\CC}(\vv)$. More technically stated, our main result describes the global behavior of this map:

\begin{Thm} \label{thm:MAP}
Fix a base point $\sigma \in \Stab^\dag(X)$. 
\begin{enumerate}
\item \label{enum:piecewise}
Under the identification of the N\'eron-Severi groups induced by the birational maps
of Theorem \ref{thm:birational-WC}, the maps $\ell_\CC$ of \eqref{eq:ellCC} glue to a piece-wise
analytic continuous map
\begin{equation} \label{eq:ell}
\ell \colon \Stab^\dag(X) \to \NS \left(M_\sigma(\vv)\right).
\end{equation}
\item \label{enum:imagemovable}
The image of $\ell$ is the intersection of the movable cone with the big cone of 
$M_\sigma(\vv)$.
\item \label{enum:allMMP}
The map $\ell$ is compatible, in the sense that for any generic $\sigma' \in \Stab^\dag(X)$, the 
moduli space $M_{\sigma'}(\vv)$ is the birational model corresponding to $\ell(\sigma')$. 
In particular, every smooth $K$-trivial birational model of $M_{\sigma}(\vv)$ appears as a moduli space $M_\CC(\vv)$ of Bridgeland stable
objects for some chamber $\CC \subset \Stab^\dag(X)$.
\item \label{enum:AmpleCone} For a chamber $\CC \subset \Stab^\dag(X)$, we have $\ell(\CC)=\mathrm{Amp}(M_\CC(\vv))$.
\end{enumerate}
\end{Thm}
The image $\ell(\tau)$ of a stability condition $\tau$ is determined by its central charge; see
Theorem \ref{thm:ellandgroupaction} for a precise statement.

Claims \eqref{enum:imagemovable} and \eqref{enum:allMMP} are the precise version of our claim above
that MMP can be run via wall-crossing: any minimal model can be reached after wall-crossing as a
moduli space of stable objects. Extremal contractions arising as canonical models are given as 
coarse moduli spaces for stability conditions on a wall.

\subsection*{Wall-crossing transformation}
Our second main result is Theorem \ref{thm:walls}. It determines
the location of walls in $\Stab^\dag(X)$, and for each wall $\WW$ it describes the associated
birational modification of the moduli space precisely. These descriptions are given purely
in terms of the algebraic Mukai lattice $H^*_\alg(X, \Z)$ of $X$:

To each wall $\WW$ we associate a rank two lattice $\HH_\WW \subset H^*_\alg(X, \Z)$,
consisting of Mukai vectors whose central charges align for stability conditions on $\WW$. Theorem
\ref{thm:walls} determines the birational wall-crossing behavior of $\WW$ completely in terms of the
pair $(\vv, \HH_\WW)$.  Rather than setting up the necessary notation here, we invite the 
reader to jump directly to Section \ref{sec:hyperbolic} for the full statement.
 
The proof of Theorem \ref{thm:walls} takes up Sections \ref{sec:hyperbolic} to \ref{sec:flopping},
and can be considered the heart of this paper. The ingredients in the proof include
Harder-Narasimhan filtrations in families, a priori constraints on the geometry of birational
contractions of hyperk\"ahler varieties, and the essential fact that every
moduli space of stable objects on a K3 surface has expected dimension.

\subsection*{Fourier-Mukai transforms and birational moduli spaces}
The following result is a consequence of 
Mukai-Orlov's Derived Torelli Theorem for K3 surfaces, a crucial Hodge-theoretic result by
Markman, and Theorem \ref{thm:birational-WC}. It
completes Mukai's program, started in \cite{Mukai:duality-Picard,Mukai:Fourier-moduli}, to understand birational maps between
moduli spaces of sheaves via Fourier-Mukai transforms. Following Mukai, consider 
$H^*(X,\Z)$ equipped with its weight two Hodge structure, polarized by the \emph{Mukai pairing}.
We write $\vv^{\perp, \tr} \subset H^*(X, \Z)$ for the orthogonal complement of $\vv$. By a result
of Yoshioka \cite{Yoshioka:Abelian}, $\vv^{\perp, \tr}$ and $H^2(M_H(\vv), \Z)$ are isomorphic as
Hodge structures; the Mukai pairing on $H^*(X, \Z)$ gets identified with the 
\emph{Beauville-Bogomolov} pairing on
$H^2(M_H(\vv),\Z)$.

\begin{Cor}\footnote{We will prove this and the following results more generally for moduli spaces of
Bridgeland-stable complexes.}\label{cor:Mukaifull}
Let $X$ and $X'$ be smooth projective K3 surfaces.
Let $\vv\in H^*_{\alg}(X,\Z)$ and $\vv'\in H^*_{\alg}(X',\Z)$ be primitive Mukai vectors.
Let $H$ (resp., $H'$) be a generic polarization with respect to $\vv$ (resp., $\vv'$).
Then the following statements are equivalent:
\begin{enumerate}
\item $M_H(\vv)$ is birational to $M_{H'}(\vv')$. \label{enum:birationalmoduli}
\item The embedding $\vv^{\perp, \tr} \subset H^*(X, \Z)$ of integral weight-two Hodge structures is 
isomorphic to the embedding $\vv'^{\perp, \tr} \subset H^*(X', \Z)$. \label{enum:Markmancrit}
\item There is an
(anti-)equivalence $\Phi$ from $\Db(X)$ to $\Db(X')$ with $\Phi_*(\vv)=\vv'$.
\label{enum:equivnumerical}
\item There is an (anti-)equivalence $\Psi$ from $\Db(X)$ to $\Db(X')$ with $\Psi_*(\vv) = \vv'$ that 
maps a generic object $E \in M_H(\vv)$ to an object $\Psi(E) \in M_{H'}(\vv')$.
\label{enum:equivbirational}
\end{enumerate}
\end{Cor}

The equivalence $\eqref{enum:birationalmoduli} \Leftrightarrow \eqref{enum:Markmancrit}$ is 
a special case of \cite[Corollary 9.9]{Eyal:survey}, which is based on Markman's description of the monodromy group 
and Verbitsky's global Torelli theorem.  We will only need the implication
$\eqref{enum:birationalmoduli} \Rightarrow \eqref{enum:Markmancrit}$, which is part of
earlier work by Markman: \cite[Theorem 1.10 and Theorem 1.14]{Eyal:integral} (when combined with the
fundamental result \cite[Corollary 2.7]{Huybrechts:Kaehlercone} that birational hyperk\"ahler
varieties have isomorphic cohomology).

By \cite{Toda:K3}, stability is an open property in families; thus $\Psi$ as in
\eqref{enum:equivbirational} directly induces a birational map
$M_H(\vv) \dashrightarrow M_{H'}(\vv')$; in particular, $\eqref{enum:equivbirational} \Rightarrow
\eqref{enum:birationalmoduli}$.
We will prove at the end of Section \ref{sec:MainThms} that 
derived Torelli for K3 surfaces \cite{Orlov:representability} gives
$\eqref{enum:Markmancrit} \Rightarrow \eqref{enum:equivnumerical}$, and
that Theorem \ref{thm:birational-WC}
provides the missing implication $\eqref{enum:equivnumerical} \Rightarrow \eqref{enum:equivbirational}$.
Thus, in the case of moduli spaces of sheaves, we obtain a proof of Markman's version
\cite[Corollary 9.9]{Eyal:survey} of global Torelli independent of \cite{Verbitsky:torelli}.

\subsection*{Cones of curves and divisors}
As an application, we can use Theorems \ref{thm:MAP} and \ref{thm:walls} to determine the cones
of effective, movable, and nef divisors (and thus dually the Mori cone of curves) of 
the moduli space $M_H(\vv)$ of $H$-Gieseker stable sheaves completely in terms of the algebraic Mukai lattice of $X$;
as an example we will state here our description of the nef cone.

Recall that we assume $\vv$ primitive and $H$ generic; in particular, $M_H(\vv)$ is smooth.
Restricting the Hodge isomorphism of \cite{Yoshioka:Abelian} mentioned previously to the algebraic
part, we get an isometry 
$\theta \colon \vv^\perp \to \NS(M_H(\vv))$ of lattices, where $\vv^\perp$ denotes the orthogonal complement
of $\vv$ inside the algebraic Mukai lattice $\Halg$. (Equivalently,
$\vv^{\perp} \subset \vv^{\perp, \tr}$ is the sublattice of $(1,1)$-classes with respect to the
induced Hodge structure on $\vv^{\perp, \tr}$.)
Let $\Pos(M_H(\vv))$ denote the cone of 
strictly positive classes $D$ with respect to the Beauville-Bogomolov pairing, satisfying $(D, D) > 0$
and $(A, D) > 0$ for a fixed ample class $A \in \NS(M_H(\vv))$.
We let $\overline{\Pos}(M_H(\vv))$ denote its closure, and by abuse of language we call it the \emph{positive cone}.

\begin{repThm}{thm:nefcone}
Consider the chamber decomposition of the closed positive cone $\overline{\Pos}(M_H(\vv))$ whose walls are given
by linear subspaces of the form
\[
\theta(\vv^\perp \cap \aa^\perp),
\]
for all $\aa \in \Halg$ satisfying $\aa^2 \ge -2$
and $0 \le (\vv, \aa) \le \frac{\vv^2}2$. Then the nef cone of $M_H(\vv)$ is one of the chambers
of this chamber decomposition.

In other words, given an ample class $A \in \NS(M_H(\vv))$, a class $D \in \overline{\Pos}(M_H(\vv))$ is 
nef if and only if
$(D, \theta(\pm\aa)) \ge 0$ for all classes $\aa$ as above and a choice of sign such that 
$(A, \theta(\pm\aa)) > 0$.
\end{repThm}

We obtain similar descriptions of the movable and effective cone, see Section \ref{sec:cones}.
The intersection of the movable cone with the strictly positive cone has been described 
by Markman for any hyperk\"ahler variety \cite[Lemma 6.22]{Eyal:survey};
the pseudo-effective cone can also easily be deduced from his results.
Our method gives an alternative wall-crossing proof, and in
addition a description of the boundary, based the proof of the Lagrangian fibration
conjecture discussed below.

However, there was no known description of the nef cone except for specific examples, even in
the case of the Hilbert scheme of points.
A general conjecture by Hassett and Tschinkel, \cite[Thesis 1.1]{HassettTschinkel:ExtremalRays},
suggested that the nef cone (or dually, its Mori cone) of a hyperk\"ahler variety $M$
depends only on the lattice of algebraic cycles in $H_2(M,\Z)$. In small dimension, their conjecture
has been verified in \cite{HassettTschinkel:RationalCurves, HassettTschinkel:MovingCone, HassettTschinkel:ExtremalRays, HHT:ProjectiveSpaces,BakkerJorza:LagrangianK34}.
The original conjecture turned out to be incorrect, already for Hilbert schemes (see \cite[Remark 10.4]{BM:projectivity} and \cite[Remark 8.10]{KnutsenCiliberto}).
However, Theorem \ref{thm:nefcone} is in fact very closely related to the Hassett-Tschinkel Conjecture: we will explain this precisely in Section \ref{sec:cones}, in particular Proposition \ref{prop:RelationHT} and Remark \ref{rmk:RelationHT}.
In Section \ref{sec:examples}, we give many explicit examples of nef and movable cones.

Using deformation techniques, Theorem \ref{thm:nefcone} and Proposition \ref{prop:RelationHT} 
have now been extended to all hyperk\"ahler varieties of the same deformation
type, see \cite{Mori-cones, Mongardi:note}.

\subsection*{Existence of Lagrangian fibrations}
The geometry of a hyperk\"ahler variety $M$ is particularly rigid. For example,
Matsushita proved in \cite{Matsushita:Addendum} that any map $f \colon M \to Y$ with connected
fibers and $\dim(Y) < \dim(M)$ is a Lagrangian fibration; further, Hwang proved in
\cite{Hwang:fibrationsbase} that if $Y$ is smooth, it must be isomorphic to a projective space.

It becomes a natural question to ask when such a fibration exists, or when 
it exists birationally.
According to a long-standing conjecture, this can be detected purely in terms of the quadratic
Beauville-Bogomolov form on the N\'eron-Severi group of $M$:

\begin{Con}[Tyurin-Bogomolov-Hassett-Tschinkel-Huybrechts-Sawon]\label{conj:SYZ}
Let $M$ be a compact hyperk\"ahler manifold of dimension $2m$, and let $q$ denote its
Beauville-Bogomolov form.
\begin{enumerate}[label={(\alph*)}] 
\item \label{enum:birational}
There exists an integral divisor class $D$ with $q(D) = 0$ if and only if
there exists a birational hyperk\"ahler manifold $M'$ admitting a Lagrangian fibration.
\item \label{enum:nef}
If in addition, $M$ admits a \emph{nef} integral primitive divisor class $D$ with $q(D) = 0$, then
there exists a Lagrangian fibration $f \colon M \to \P^m$ induced by the complete linear system of $D$.
\end{enumerate}
\end{Con}

In the literature, it was first suggested by Hassett-Tschinkel in
\cite{HassettTschinkel:RationalCurves} for symplectic
fourfolds, and, independently, by Huybrechts \cite{GrossHuybrechtsJoyce} and Sawon
\cite{Sawon:AbelianFibred} in general; see \cite{Verbitsky:HyperkaehlerSYZ} for more remarks on the
history of the Conjecture.


\begin{Thm}\label{thm:SYZ}
Let $X$ be a smooth projective K3 surface.
Let $\vv\in H_{\mathrm{alg}}^*(X,\Z)$ be a primitive Mukai vector with $\vv^2>0$ and let
$H$ be a generic polarization with respect to $\vv$. 
Then Conjecture \ref{conj:SYZ} holds for the moduli space $M_H(\vv)$ of $H$-Gieseker stable sheaves.
\end{Thm}

The basic idea of our proof is the following: as we recalled above, the
N\'eron-Severi group of $M_H(\vv)$, along with its Beauville-Bogomolov form, is isomorphic to the
orthogonal
complement $\vv^\perp \subset \Halg$ of $\vv$ in the algebraic Mukai lattice of $X$, along with the restriction
of the Mukai pairing. The existence of an integral divisor $D = c_1(L)$ with $q(D) = 0$ is thus
equivalent to the existence of an isotropic class $\ww \in \vv^\perp$: a class with $(\ww, \ww) = 0$
and $(\vv, \ww) = 0$. The moduli space $Y = M_H(\ww)$ is a smooth K3 surface, and the associated Fourier-Mukai
transform $\Phi$ sends sheaves of class $\vv$ on $X$ to complexes of rank 0 on $Y$. While these
complexes on $Y$ are typically not sheaves---not even for a generic object in $M_H(\vv)$---,
we can arrange them to be Bridgeland-stable complexes with respect to a Bridgeland-stability
condition $\tau$ on $\Db(Y)$. We then deform $\tau$ along a path with endpoint $\tau'$, such that
$\tau'$-stable complexes of class $\Phi_*(\vv)$ are Gieseker stable sheaves, necessarily of rank zero.
In other words, the Bridgeland-moduli space $M_{\tau'}(\Phi_*(\vv))$ is a moduli space of sheaves
$\FF$ with support $\abs{\FF}$ on a curve of fixed degree.  The map $\FF \mapsto
\abs{\FF}$ defines a map from $M_{\tau'}(\Phi_*(\vv))$ to the linear system of the associated curve;
this map is a Lagrangian fibration, known as the \emph{Beauville integral system}. On the other hand, birationality of wall-crossing shows that 
$M_{\tau}(\Phi_*(\vv)) = M_H(\vv)$ is birational to $M_{\tau'}(\Phi_*(\vv))$.

The idea to use a Fourier-Mukai transform to prove Conjecture \ref{conj:SYZ} was used previously by
Markushevich \cite{Markushevich:Lagrangian} and Sawon \cite{Sawon:LagrangianFibrations} for a specific 
family of Hilbert schemes on K3 surfaces of Picard rank one. Under their assumptions,
the Fourier-Mukai transform of an ideal sheaf is a stable torsion sheaf;
birationality of wall-crossing makes such a claim unnecessary.

\begin{Rem}
By \cite{EyalSukhendu:Density}, Hilbert schemes of $n$ points on projective K3 surfaces are dense in the moduli space of hyperk\"ahler varieties of $K3^{[n]}$-type.

Conjecture \ref{conj:SYZ} has been proved independently by Markman
\cite{Eyal:LagrangianFibration} for a very general hyperk\"ahler variety $M$ of $K3^{[n]}$-type;
more specifically, under
the assumption that $H^{2,0}(M)\oplus H^{0,2}(M)$ does not contain any integral class.
His proof is completely different from ours, based on Verbitsky's
Torelli Theorem, and a way to associate a K3 surface (purely lattice theoretically) to such
hyperk\"ahler manifolds with a square-zero divisor class.

These results have been extended by 
Matsushita to any variety of $K3^{[n]}$-type \cite{Matsushita:isotropic}.
\end{Rem}

\subsection*{Geometry of flopping contractions}
As mentioned previously, every extremal contraction of $M_H(\vv)$ is induced by a wall 
in the space of Bridgeland stability conditions. In Section 
\ref{sec:floppingeometry}, we explain how basic geometric properties of 
flopping contractions are also determined via the associated lattice-theoretic wall-crossing data; 
this adds geometric content to Theorem \ref{thm:walls}.
We obtain examples where the exceptional locus
has either arbitrarily many connected components, or arbitrarily many irreducible components all
intersecting in one point.

\subsection*{Strange Duality}
In Section \ref{sec:SD} we apply Theorem \ref{thm:SYZ} to study Le
Potier's Strange Duality, in the case where one of the two classes involved has square zero. 
We give sufficient criteria for strange duality to hold, which are determined by wall-crossing, and
which are necessary in examples.

\subsection*{Generality}
In the introduction, we have stated most results for Gieseker moduli spaces $M_H(\vv)$.
In fact, we will work throughout more generally with moduli spaces $M_\sigma(\vv)$ of Bridgeland stable objects
on a K3 surface $(X, \alpha)$ with a Brauer twist $\alpha$, and all results will be proved 
in that generality.

\subsection*{Relation to previous work on wall-crossing}
Various authors have previously studied examples of the relation between wall-crossing and the
birational geometry of the moduli space induced by the chamber decomposition of its cone of movable
divisors: the first examples (for moduli of torsion sheaves on $K$-trivial surfaces) were studied in
\cite{Aaron-Daniele}, and moduli on abelian surfaces were considered (in varying generality) in
\cite{MaciociaMeachan, Maciocia:walls, Minamide-Yanagida-Yoshioka:wall-crossing, MYY2,
YY:abeliansurfaces, Yoshioka:cones}.

Several of our results have analogues for abelian surfaces that have been obtained
previously by Yoshioka, or by Minamide, Yanagida, and Yoshioka:
the birationality of wall-crossing has been established in
\cite[Theorem 4.3.1]{Minamide-Yanagida-Yoshioka:wall-crossing};
the ample cone of the moduli spaces is described in \cite[Section 4.3]{MYY2}; 
statements related to Theorem \ref{thm:MAP} can be found in \cite{Yoshioka:cones}; 
an analogue of Corollary \ref{cor:Mukaifull} is contained in \cite[Theorem
0.1]{Yoshioka:FM_abelian_surfaces}; and Conjecture \ref{conj:SYZ} is proved in
\cite[Proposition 3.4 and Corollary 3.5]{Yoshioka:FM_abelian_surfaces} with the same basic
approach.

The crucial difference between abelian surfaces and K3 surfaces is the existence of spherical objects
on the latter. They are responsible for the existence of \emph{totally semistable walls} (walls for
which there are no strictly stable objects) that are harder to control; in particular, these can
correspond to any possible type of birational transformation (isomorphism, divisorial contraction,
flop). The spherical classes are the main reason our wall-crossing analysis in Sections
\ref{sec:hyperbolic}---\ref{sec:flopping} is fairly involved.

A somewhat different behavior was established in \cite{ABCH:MMP} in many cases for the Hilbert
scheme of points on $\P^2$ (extended to torsion-free sheaves in \cite{Huizenga:P2, AaronAndStudents:P2}, and to
Hirzebruch surfaces in \cite{Aaron-Izzet:pointsonsurfaces}): the authors show that the chamber
decomposition in the space of stability conditions corresponds to the base locus decomposition of
the \emph{effective} cone. In particular, while the map $\ell_\CC$ of equation \eqref{eq:ellCC}
exists similarly in their situation, it will behave differently across walls corresponding to a
divisorial contraction: in our case, the map ``bounces back'' into the ample cone, while in their
case, it will extend across the wall.

\subsection*{Acknowledgments}
Conversations with Ralf Schiffler dissuaded us from pursuing a failed approach to the birationality
of wall-crossing, and we had extremely useful discussions with Daniel Huybrechts. Tom Bridgeland
pointed us towards Corollary \ref{cor:Mukaifull}, and Dragos Oprea towards the results in Section
\ref{sec:SD}.  We also received helpful comments from Daniel Greb, Antony Maciocia, Alina Marian,
Eyal Markman, Dimitri Markushevich, Daisuke Matsushita, Ciaran Meachan, Misha Verbitsky, K\=ota
Yoshioka, Ziyu Zhang, and we would like to thank all of them. We would also like to thank the
referee very much for an extremely careful reading of the paper, and for many useful suggestions.

The authors were visiting the Max-Planck-Institut Bonn, respectively the
Hausdorff Center for Mathematics in Bonn, while working on this paper, and would like to thank
both institutes for their hospitality and stimulating atmosphere.

A.~B.~ is partially supported by NSF grant DMS-1101377.
E.~M.~ is partially supported by NSF grant DMS-1001482/DMS-1160466, Hausdorff Center for
Mathematics, Bonn, and by SFB/TR 45.

\subsection*{Notation and Convention}

For an abelian group $G$ and a field $k(=\Q,\R,\C)$, we denote by $G_k$ the $k$-vector space $G\otimes k$.

Throughout the paper, $X$ will be a smooth projective K3 surface over the complex numbers. We refer
to Section \ref{sec:ReviewK3s} for all notations specific to K3 surfaces.

We will abuse notation and usually denote all derived functors as if they were underived. We write the
dualizing functor as $(\blank)^\vee = \RlHom(\blank, \OO_X)$.

The skyscraper sheaf at a point $x\in X$ is denoted by $k(x)$.
For a complex number $z\in\mathbb{C}$, we denote its real and imaginary part by $\Re z$ and $\Im z$, respectively.

By \emph{simple object} in an abelian category we will denote an object that has no non-trivial subobjects.

Recall that an object $S$ in a K3 category is spherical if $\Hom^\bullet(S, S) = \C \oplus \C[-2]$.
We denote the associated spherical twist at $S$ by $\ST_S(\blank)$; it is defined
\cite{Mukai:BundlesK3, Seidel-Thomas:braid} by the exact triangle 
\[
\Hom^\bullet(S, E) \otimes S \to E \to \ST_S(E).
\]

We will write \emph{stable} (in italics) whenever we are considering strictly stable objects in a
context allowing strictly semistable objects: for a non-generic stability condition, or  for
objects with non-primitive Mukai vector.

\section{Review: derived categories of K3 surfaces, stability conditions, moduli spaces}\label{sec:ReviewK3s}

In this section, we give a review of stability conditions K3 surfaces, and their moduli spaces of
stable complexes.  The main references are \cite{Bridgeland:Stab, Bridgeland:K3, Toda:K3, Yoshioka:Abelian,
BM:projectivity}.

\subsection*{Bridgeland stability conditions}

Let $\DD$ be a triangulated category.

\begin{Def}\label{def:slicing}
A slicing $\PP$ of the category $\DD$ is a collection of full extension-closed subcategories $\PP(\phi)$ for $\phi \in \R$ with the following properties:
\begin{enumerate}
\item $\PP(\phi + 1) = \PP(\phi)[1]$.
\item If $\phi_1 > \phi_2$, then $\Hom(\PP(\phi_1), \PP(\phi_2)) = 0$.
\item For any $E \in \DD$, there exists a collection of real numbers $\phi_1 > \phi_2 > \dots >
\phi_n$ and a sequence of triangles
\begin{equation} \label{eq:HN-filt}
 \TFILTB E A n
\end{equation}
with $A_i \in \PP(\phi_i)$.
\end{enumerate}
\end{Def}

The collection of exact triangles in \eqref{eq:HN-filt} is called the \emph{Harder-Narasimhan (HN) filtration}
of $E$.
Each subcategory $\PP(\phi)$ is extension-closed and abelian.
Its nonzero objects are called semistable of phase $\phi$, and its simple objects are called stable.

We will write $\phi_{\min}(E) := \phi_n$ and $\phi_{\max}(E) := \phi_1$.
By $\PP(\phi -1, \phi]$ we denote the full subcategory of objects with $\phi_{\min}(E) > \phi -1$ and $\phi_{\max}(E) \le \phi$.
This is the heart of a bounded t-structure $(\DD^{\le 0}, \DD^{\ge 0})$ given by 
\[ \DD^{\le 0} = \PP(>\phi-1) = \{E \in \DD\colon \phi_{\min} > \phi-1\}
\quad \text{and} \quad
\DD^{\ge 0} = \PP(\le\phi) = \{E \in \DD\colon \phi_{\max} \le \phi\}.\]

Let us fix a lattice of finite rank $\Lambda$ and a surjective map $\vv\colon K(\DD)\onto\Lambda$.

\begin{Def}[{\cite{Bridgeland:Stab, Kontsevich-Soibelman:stability}}]
\label{def:Bridgeland}
A \emph{Bridgeland stability condition} on $\DD$ is a pair $(Z,\PP)$, where
\begin{itemize}
\item $Z\colon\Lambda\to\C$ is a group homomorphism, and
\item $\PP$ is a slicing of $Z$,
\end{itemize}
satisfying the following compatibilities:
\begin{enumerate}
\item $\frac{1}{\pi}\mathrm{arg}Z(\vv(E))=\phi$, for all non-zero $E\in\PP(\phi)$;
\item given a norm $\|\blank\|$ on $\Lambda_\R$, there exists a constant $C>0$ such that
\[
|Z(\vv(E))| \geq C \| \vv(E) \|,
\]
for all $E$ that are semistable with respect to $\PP$.
\end{enumerate}
\end{Def}

We will write $Z(E)$ instead of $Z(\vv(E))$ from now on.

A stability condition is called \emph{algebraic} if $\Imm(Z)\subset\Q\oplus\Q\sqrt{-1}$.

The main theorem in \cite{Bridgeland:Stab} shows that the set $\Stab(\DD)$ of stability conditions
on $\DD$ is a complex manifold; its dimension equals the rank of $\Lambda$.

\begin{Rem}[{\cite[Lemma 8.2]{Bridgeland:Stab}}]
\label{rmk:GroupAction}
There are two group actions on $\Stab(\DD)$.
The group $\Aut(\DD)$ of autoequivalences acts on the left by
$\Phi_*(Z,\PP)=(Z\circ\Phi_*^{-1},\Phi(\PP))$, where $\Phi \in \Aut(\DD)$ and $\Phi_*$ also denotes the push-forward on the
K-group.    The universal cover $\widetilde{\GL}^+_2(\R)$ of 
matrices in $\GL_2(\R)$ with positive determinant acts on the right, lifting the action of
$\GL_2(\R)$ on $\Hom(K(\DD), \C) = \Hom(K(\DD), \R^2)$.
\end{Rem}

\subsection*{Twisted K3 surfaces}

Let $X$ be a smooth K3 surface.
The (cohomological) \emph{Brauer group} $\mathrm{Br}(X)$ is the torsion part of the cohomology group $H^2(X,\OO_X^*)$ in the analytic topology.

\begin{Def}\label{def:twisted}
Let $\alpha\in\mathrm{Br}(X)$.
The pair $(X,\alpha)$ is called a \emph{twisted K3 surface}.
\end{Def}

Since $H^3(X,\Z)=0$, there exists a \emph{B-field lift} $\beta_0\in H^2(X,\Q)$ such that $\alpha = e^{\beta_0}$.
We will always tacitly fix both such B-field lift and a \v{C}ech representative $\alpha_{ijk}\in\Gamma(U_i\cap U_j\cap U_k,\mathcal{O}^*_X)$ on an open analytic cover $\{U_i\}$ in $X$;
see \cite[Section 1]{HuybrechtsStellari:Twisted} for a discussion about these issues.

\begin{Def}\label{def:TwistedSheaf}
An \emph{$\alpha$-twisted coherent sheaf}
$F$ consists of a collection $(\{F_i\},\{\varphi_{ij}\})$, where $F_i$ is a coherent sheaf on $U_i$
and $\varphi_{ij}\colon F_j|_{U_i\cap U_j}\to F_i|_{U_i\cap U_j}$ is an isomorphism, such that: 
\[
\varphi_{ii}=\mathrm{id};\quad
\varphi_{ji}=\varphi_{ij}^{-1};\quad
\varphi_{ij}\circ\varphi_{jk}\circ\varphi_{ki}=\alpha_{ijk}\cdot\mathrm{id}.
\]
\end{Def}

We denote by $\Coh(X,\alpha)$ the category of $\alpha$-twisted coherent sheaves on $X$, and by $\Db(X,\alpha)$ its bounded derived category.
We refer to \cite{Caldararu:Thesis, HuybrechtsStellari:Twisted, Yoshioka:TwistedStability, Lieblich:Twisted} for basic facts about twisted sheaves on K3 surfaces.

In \cite[Section 1]{HuybrechtsStellari:Twisted}, the authors define a twisted Chern character by
\[
\ch\colon K(\Db(X,\alpha)) \to H^*(X,\Q), \quad \ch(\blank) = e^{\beta_0} \cdot
\ch^{\mathrm{top}}(\blank),
\]
where $\ch^{\mathrm{top}}$ is the topological Chern character.
By \cite[Proposition 1.2]{HuybrechtsStellari:Twisted}, we have
\[
\ch(\blank) \in \left[ e^{\beta_0} \cdot \left( H^0(X,\Q)\oplus\mathrm{NS}(X)_\Q\oplus H^4(X,\Q) \right)\right] \cap H^*(X,\Z).
\]

\begin{Rem}\label{rmk:TwistedHodgeStructure}
Let $H^*(X,\alpha,\Z):= H^*(X,\Z)$.
In \cite{HuybrechtsStellari:Twisted}, the authors define a weight-2 Hodge structure on the whole cohomology $H^*(X,\alpha,\Z)$
with 
\[
H^{2,0}(X, \alpha, \C) := e^{\beta_0} \cdot H^{2,0}(X, \C).
\]
We denote by
\[
\Halgalpha := H^{1,1}(X,\alpha,\C) \cap H^*(X,\Z) 
\]
its $(1,1)$-integral part.
It coincides with the image of the twisted Chern character.
When $\alpha=1$, this reduces to the familiar definition
$H^*_{\alg}(X,\Z) = H^0(X,\Z)\oplus\mathrm{NS}(X)\oplus H^4(X,\Z).$
\end{Rem}

\subsection*{The algebraic Mukai lattice}

Let $(X,\alpha)$ be twisted K3 surface.

\begin{Def}\label{def:MukaiLattice}
\begin{enumerate*}
\item We denote by $\vv \colon K(\Db(X,\alpha)) \to H^*_{\alg}(X, \alpha, \Z)$ the \emph{Mukai vector}
\[
\vv(E) := \ch(E) \sqrt{\td(X)}.
\]
\item The \emph{Mukai pairing} $(\blank, \blank)$ is defined on $H^*_{\alg}(X, \alpha, \Z)$ by
\[
((r,c,s),(r',c',s')) := cc' - rs'-sr' \in \Z.
\]
It is an even pairing of signature $(2,\rho(X))$, satisfying $-(\vv(E), \vv(F)) = \chi(E, F) =
\sum_i (-1)^i \ext^i(E, F)$ for all $E,F\in \Db(X,\alpha)$.
\item The \emph{algebraic Mukai lattice} is defined to be the pair $\left(H^*_{\alg}(X, \alpha, \Z), (\blank,\blank)\right)$.
\end{enumerate*}
\end{Def}
 
Recall that an embedding $i\colon V\to L$ of a lattice $V$ into a lattice $L$ is \emph{primitive} if $L/i(V)$ is a free abelian group.
In particular, we call a non-zero vector $\vv\in H^*_{\alg}(X, \alpha, \Z)$ \emph{primitive} if it is not divisible in $H^*_{\alg}(X, \alpha, \Z)$.
Throughout the paper $\vv$ will often denote a primitive class with $\vv^2 > 0$.

Given a Mukai vector $\vv\in H^*_{\alg}(X, \alpha, \Z)$, we denote its orthogonal complement by
$\vv^\perp$.

\subsection*{Stability conditions on K3 surfaces}

Let $(X,\alpha)$ be a twisted K3 surface. We remind the reader that this includes
fixing a B-field lift $\beta_0$ of the Brauer class $\alpha$.

\begin{Def}
A (full, numerical) \emph{stability condition} on $(X,\alpha)$ is a Bridgeland stability condition
on $\Db(X,\alpha)$, whose lattice $\Lambda$ is given by the Mukai lattice $\Halgalpha$.
\end{Def}

In \cite{Bridgeland:K3}, Bridgeland describes a connected component of the space of numerical stability
conditions on $X$. These results have been extended to twisted K3 surfaces in \cite{HMS:generic_K3s}.
In the following, we briefly summarize the main results.

Let $\beta, \omega \in \NS(X)_\R$ be two real divisor classes, with $\omega$ being ample.
For $E\in \Db(X,\alpha)$, define
\[
Z_{\omega,\beta}(E) := \left(e^{i\omega + \beta + \beta_0}, \vv(E)\right).
\]
In \cite[Lemma 6.1]{Bridgeland:K3} 
Bridgeland constructs a heart $\AA_{\omega, \beta}$ by tilting
at a torsion pair (see \cite[Section 3.1]{HMS:generic_K3s} for the case $\alpha \neq 1$).
Its objects are two-term complexes $E^{-1} \xrightarrow{d} E^0$ with the property:
\begin{itemize}
\item $\Ker d$ is a torsion-free $\alpha$-twisted sheaf such that, for every non-zero subsheaf $E'\subset \Ker d$, we have $\Im Z_{\omega,\beta}(E')\leq0$;
\item the torsion-free part of $\Cok d$ is such that, for every non-zero torsion free quotient $\Cok d \onto E''$, we have $\Im Z_{\omega,\beta}(E'')>0$. 
\end{itemize}

\begin{Thm}[{\cite[Sections 10, 11]{Bridgeland:K3}, \cite[Proposition 3.8]{HMS:generic_K3s}}] \label{thm:BridgelandK3geometric}
Let $\sigma = (Z, \PP)$ be a stability condition such that all
skyscraper sheaves $k(x)$ of points are $\sigma$-stable. Then
there are real divisor classes $\omega, \beta \in \NS(X)_\R$ with $\omega$ ample, such that,
up to the $\widetilde{\GL}^+_2(\R)$-action, $\sigma$ is equal to the stability condition
$\sigma_{\omega, \beta}$ 
determined by $\PP((0, 1]) = \AA_{\omega, \beta}$ and $Z=Z_{\omega,\beta}$.
\end{Thm}

We will call such stability conditions \emph{geometric}, and write 
$U(X, \alpha) \subset \Stab(X, \alpha)$ for the 
the open subset of geometric stability conditions.

Using the Mukai pairing, we identify any central charge
$Z \in \Hom(\Halgalpha, \C)$ with a vector $\Omega_Z$ in $\Halgalpha \otimes \C$ such that
\[
Z (\blank) = \left( \Omega_Z,\blank \right).
\]
The vector $\Omega_Z$ belongs to the domain $\PP_0^+(X,\alpha)$, which we now describe.
Let
\[
\PP(X,\alpha) \subset \Halgalpha \otimes \C
\]
be the set of vectors $\Omega$ such that
$\Im \Omega, \Re \Omega$ span a positive definite 2-plane in $\Halgalpha \otimes \R$. 
The subset $\PP_0(X,\alpha)$ is the set of vectors not orthogonal to any spherical class:
\[ \PP_0(X,\alpha) =
\stv{\Omega \in \PP(X,\alpha)}{ (\Omega, \ss) \neq 0,\,
				\text{for all $\ss \in \Halgalpha \text{ with } \ss^2 = -2$}}.
\]
Finally, $\PP_0(X,\alpha)$ has two connected components, corresponding to the orientation induced on the 
plane spanned by $\Im \Omega, \Re \Omega$; we let
$\PP_0^+(X,\alpha)$ be the component containing vectors of the form $e^{i\omega + \beta + \beta_0}$ for 
$\omega$ ample.

\begin{Thm}[{\cite[Section 8]{Bridgeland:K3}, \cite[Proposition 3.10]{HMS:generic_K3s}}] \label{thm:Bridgeland_coveringmap}
Let $\Stab^\dag(X,\alpha)$ be the connected component of the space of stability conditions
containing geometric stability conditions $U(X,\alpha)$. 
Let $\ZZ \colon \Stab^\dag(X,\alpha) \to \Halgalpha \otimes \C$ be the map sending
a stability condition $(Z, \PP)$ to $\Omega_Z$, where $Z(\blank) = (\Omega_Z, \blank)$.

Then $\ZZ$ is a covering map of $\PP_0^+(X,\alpha)$.
\end{Thm}

We will need the following observation:

\begin{Prop} \label{prop:dualstability}
The stability conditions $\sigma_{\omega, \beta}$ on $U(X, \alpha)$ and $\sigma_{\omega, -\beta}$ 
on $U(X, \alpha^{-1})$ are dual to each
other in the following sense: An object $E \in \Db(X, \alpha)$ is $\sigma_{\omega, \beta}$-(semi)stable 
of phase $\phi$
if and only if its shifted derived dual $E^\vee[2] \in \Db(X, \alpha^{-1})$ is $\sigma_{\omega,-\beta}$-(semi)stable of phase $-\phi$.
\end{Prop}

\begin{proof}
By \cite[Propositions 3.3.1 \& 4.2]{large-volume}, this follows as in \cite[Proposition 4.3.6]{BMT:3folds-BG}.
\end{proof}

\subsection*{Derived Torelli}

Any positive definite 4-plane in $H^*(X, \alpha, \R)$ comes equipped with a canonical orientation,
induced by the K\"ahler cone. A Hodge-isometry $\phi \colon H^*(X, \alpha, \Z) \to H^*(X', \alpha',
\Z)$ is called orientation-preserving if it is compatible with this orientation data. 
\begin{Thm}[Mukai-Orlov] \label{thm:derivedtorelli}
Given an orientation-preserving Hodge isometry $\phi$ between the Mukai lattice of twisted K3
surfaces $(X, \alpha)$ and $(X', \alpha')$, there exists a derived equivalence
$\Phi \colon \Db(X, \alpha) \to \Db(X', \alpha')$ with $\Phi_* = \phi$.
Moreover, $\Phi$ may be chosen such that it sends the distinguished component $\Stab^\dag(X,\alpha)$ to
$\Stab^\dag(X',\alpha')$. 
\end{Thm}
\begin{proof}
The case $\alpha=1$ follows from Orlov's representability result \cite{Orlov:representability}
(based on \cite{Mukai:BundlesK3}), see \cite{HLOY:autoequivalences, Ploog:thesis, HMS:Orientation}. The
twisted case was treated in \cite{HuybrechtsStellari:CaldararuConj}.
The second statement follows identically to the case $X =X'$ treated in
\cite[Proposition 7.9]{Hartmann:cusps}; see also \cite{Huybrechts:derived-abelian}.
\end{proof}

\subsection*{Walls}

One of the main properties of the space of Bridgeland stability conditions is that it admits a well-behaved wall and chamber structure.
This is due to Bridgeland and Toda (the precise statement is \cite[Proposition 2.3]{BM:projectivity}).

Let $(X,\alpha)$ be a twisted K3 surface and let $\vv\in H^*_{\alg}(X,\alpha,\Z)$ be a Mukai vector.
Then there exists a locally finite set of \emph{walls} (real codimension one submanifolds with boundary) 
in $\Stab^\dag(X,\alpha)$, depending only on $\vv$, with the following properties:
\begin{enumerate}
\item When $\sigma$ varies within a chamber, the sets of $\sigma$-semistable and
$\sigma$-stable objects of class $\vv$ does not change.
\item When $\sigma$ lies on a single wall $\WW \subset \Stab^\dag(X,\alpha)$,
then there is a $\sigma$-semistable object
that is unstable in one of the adjacent chambers, and semistable in the other adjacent chamber.
\item When we restrict to an intersection of finitely many walls $\WW_1, \dots, \WW_k$, we
obtain a wall-and-chamber decomposition on $\WW_1 \cap \dots \cap \WW_k$ with the same properties,
where the walls are given by the intersections $\WW \cap \WW_1 \cap \dots \cap \WW_k$ for any
of the walls $\WW \subset \Stab^\dag(X,\alpha)$ with respect to $\vv$.
\end{enumerate}

Moreover, if $\vv$ is primitive, then $\sigma$ lies on a wall if and only if there exists a
strictly $\sigma$-semistable object of class $\vv$.
The Jordan-H\"older filtration of $\sigma$-semistable objects
does not change when $\sigma$ varies within a chamber.

\begin{Def}\label{def:generic}
Let $\vv\in  H^*_{\alg}(X,\alpha,\Z)$.
A stability condition is called \emph{generic} with respect to $\vv$ if it does not lie on a wall.
\end{Def}

\begin{Rem} \label{rem:Giesekerchamber}
Given a polarization $H$ that is generic with respect to $\vv$, there is always a Gieseker
chamber $\CC$: for $\sigma \in \CC$, the moduli space
$M_{\sigma}(\vv)$ of Bridgeland stable objects is exactly the moduli space
of $H$-Gieseker stable sheaves; see \cite[Proposition 14.2]{Bridgeland:K3}.
\end{Rem}

\subsection*{Moduli spaces and projectivity}

Let $(X,\alpha)$ be a twisted K3 surface and let $\vv\in H^*_{\alg}(X,\alpha,\Z)$.
Given $\sigma=(Z,\PP)\in\Stab^\dagger(X,\alpha)$ and $\phi\in\R$ such that $Z(\vv)\in\R_{>0}\cdot
e^{\pi \phi \sqrt{-1}}$, let $ \MMM_{\sigma}(\vv,\phi)$ and $\MMM_{\sigma}^{st}(\vv,\phi)$
be the moduli stack of $\sigma$-semistable and $\sigma$-stable objects with phase $\phi$ and Mukai
vector $\vv$, respectively.
We will omit $\phi$ from the notation from now on.

If $\sigma\in\Stab^\dagger(X,\alpha)$ is generic with respect to $\vv$, then $\MMM_{\sigma}(\vv)$
has a coarse moduli space $M_{\sigma}(\vv)$ of $\sigma$-semistable objects with Mukai vector $\vv$
(\cite[Theorem 1.3(a)]{BM:projectivity}, which generalizes \cite[Theorem 0.0.2]{MYY2}).
It is a normal projective irreducible variety with $\Q$-factorial singularities.
If $\vv$ is primitive, then $M_{\sigma}(\vv)=M^{st}_{\sigma}(\vv)$ is a smooth projective hyperk\"ahler manifold (see Section \ref{sec:ReviewHK}).

By results of Yoshioka and Toda, there is a very precise criterion for non-emptiness of a moduli
space, and it always has expected dimension:

\begin{Thm}\label{thm:nonempty}
Let $\vv = m\vv_0 \in H^*_\alg(X,\alpha,\Z)$ be a vector with $\vv_0$ primitive and $m>0$,
and let $\sigma\in\Stab^\dagger(X,\alpha)$ be a generic stability condition with respect
to $\vv$.
\begin{enumerate}
\item \label{enum:nonempty}
The coarse moduli space $M_{\sigma}(\vv)$ is non-empty if and only if $\vv_0^2 \geq -2$.
\item \label{enum:dimandsquare}
Either $\dim M_\sigma(\vv) = \vv^2 + 2$ and $M^{st}_\sigma(\vv)\neq\emptyset$, or $m > 1$ and $\vv_0^2 \le 0$.
\end{enumerate}
\end{Thm}

In other words, when $\vv^2 \neq 0$ and the dimension of the moduli space is positive,
then it is given by $\dim M_\sigma(\vv) = \vv^2 + 2$.

\begin{proof}
This is well-known: we provide a proof for completeness.
First of all, claim \eqref{enum:nonempty} follows from results of Yoshioka and Toda
(see \cite[Theorem 6.8]{BM:projectivity}).
Since $\sigma$ is generic with respect to $\vv$, we know that $M_{\sigma}(\vv)$ exists as a projective variety, parameterizing S-equivalence classes of semistable objects.
Moreover, if $E\in M_\sigma(\vv)$, and we let $F\into E$ be such that $\phi_{\sigma}(F)=\phi_{\sigma}(E)$, then $\vv(F)=m' \vv_0$, for some $m'>0$.
Hence, the locus of strictly semistable objects in $M_\sigma(\vv)$ coincides with the image of the natural map
\[
\mathrm{SSL}\colon \coprod_{m_1+m_2=m} M_\sigma(m_1\vv_0)\times M_{\sigma}(m_2\vv_0) \longrightarrow M_\sigma(\vv),
\quad
\mathrm{SSL}\bigl((E_1, E_2)\bigr) = E_1\oplus E_2.
\]

If we assume $\vv_0^2>0$ (and so $\geq2$), then we can proceed by induction on $m$.
For $m=1$, $M_\sigma^{st}(\vv_0)=M_{\sigma}(\vv_0)$ and the conclusion follows from the Riemann-Roch Theorem and \cite{Mukai:BundlesK3}.
If $m>1$, then we deduce from the inductive assumption that the image of the map $\mathrm{SSL}$ has dimension equal to the maximum of $(m_1^2+m_2^2)\vv_0^2+4$, for $m_1+m_2=m$.

We claim that we can construct a semistable object $E$ with vector $\vv$ which is also a Schur
object, i.e. $\Hom(E, E) = \C$.
Indeed, again by the inductive assumption, we can consider a $\sigma$-\emph{stable} object $F_{m-1}$ with vector $(m-1)\vv_0$.
Let $F\in M_\sigma(\vv_0)$.
Then, again by the Riemann-Roch Theorem, $\Ext^1(F,F_{m-1})\neq0$.
We can take any non-trivial extension
\[
0\to F_{m-1} \to F_m \to F \to 0.
\]
Since both $F_{m-1}$ and $F$ are Schur objects, and they have no morphism between each other, $F_m$
is also a Schur object.

Again by the Riemann-Roch Theorem and \cite{Mukai:Symplectic}, we deduce that the dimension of $M_\sigma(\vv)$ is equal to $\mathrm{ext}^1(F_m,F_m)=m^2 \vv_0^2+2$.
Since, for all $m_1,m_2>0$ with $m_1+m_2=m$, we have
\[
(m_1^2+m_2^2)\vv_0^2 + 4 < m^2\vv_0^2 +2,
\]
this shows that $M^{st}_\sigma(\vv)\neq\emptyset$ as claimed.

For the case $\vv_0^2 \le 0$, see \cite[Lemma 7.1 and Lemma 7.2]{BM:projectivity}.
\end{proof}
Let us also point out that the proof shows a stronger statement:
\begin{Lem} \label{lem:nonprimitive}
Let $\vv = m \vv_0$ with $\vv_0^2 > 0$, and $\sigma \in \Stab^\dag(X,\alpha)$, not necessarily generic with
respect to $\vv$. If there exist $\sigma$-\emph{stable} objects of class $\vv_0$, then the same
holds for $\vv$.
\end{Lem}
\begin{proof}
Let $F'$ be a generic deformation of $F_m$, and assume that it is strictly semistable; let
$E \into F'$ be a semistable subobject of the same phase. The above proof shows
the Mukai vector $\vv(E)$ cannot be a multiple of $\vv_0$. 
Using the universal closedness of moduli spaces of semistable objects, it follows as in 
\cite[Theorem 3.20]{Toda:K3} that $F_m$ also has a semistable subobject with Mukai vector
equal to $\vv(E)$. This is not possible by construction.
\end{proof}

\subsection*{Line bundles on moduli spaces}

In this section we recall the main result of \cite{BM:projectivity}. It shows that every
moduli space of Bridgeland-stable objects comes equipped with a numerically positive line bundle,
naturally associated to the stability condition.

Let $(X,\alpha)$ be a twisted K3 surface.
Let $S$ be a proper algebraic space of finite type over $\C$, let $\sigma = (Z, \PP)
\in\Stab^\dag(X,\alpha)$, and let $\EE\in\Db(S\times (X,\alpha))$ be a family of $\sigma$-semistable objects
of class $\vv$ and phase $\phi$: for all closed points $s\in S$, $\EE_s\in\PP(\phi)$ with
$\vv(\EE_s)=\vv$. We write $\Phi_\EE \colon \Db(S) \to \Db(X,\alpha)$ for the Fourier-Mukai transform
associated to $\EE$.

We construct a class $\ell_{\sigma}\in\NS(S)_\R$ on $S$ as follows: To every curve $C\subset S$, we associate
\[
 C\mapsto \ell_{\sigma}.C := \Im \left(-\frac{Z(\vv(\Phi_{\EE}(\OO_C)))}{Z(\vv)}\right).
\]
This defines a numerical Cartier divisor class on $S$, see \cite[Section
4]{BM:projectivity}.

\begin{Rem} \label{rmk:comparison}
The classical construction of determinant line bundles (see \cite[Section 8.1]{HL:Moduli}) induces,
up to duality, the so-called \emph{Mukai morphism}
$\theta_{\EE} \colon \vv^\perp \to \NS(S)$. It can also be defined by
\begin{equation} \label{eq:definetheta}
\theta_{\EE}(\ww).C:= \bigl(\ww, \vv(\Phi_\EE(\OO_C))\bigr).
\end{equation}

If we assume $Z(\vv)=-1$, and write $Z(\blank)=(\Omega_Z, \blank)$ as above, we can also write
\begin{equation} \label{eq:ellandtheta}
\ell_\sigma = \theta_{\EE}(\Im \Omega_Z).
\end{equation}
\end{Rem}

\begin{Thm}[{\cite[Theorem 4.1 \& Remark 4.6]{BM:projectivity}}] \label{thm:ampleness}
The main properties of $\ell_{\sigma}$ are:
\begin{enumerate}
\item $\ell_{\sigma}$ is a nef divisor class on $S$. Additionally, for a curve $C\subset S$, we have $\ell_\sigma.C = 0$ if and only
if, for two general closed points $c, c' \in C$, the corresponding objects $\EE_c, \EE_{c'}\in\Db(X,\alpha)$ are S-equivalent.
\item For any Mukai vector $\vv\in H^*_\alg(X,\alpha,\Z)$ and a stability condition
$\sigma \in \Stab^\dag(X, \alpha)$ that is generic, 
$\ell_{\sigma}$ induces an ample divisor class on the coarse moduli space $M_{\sigma}(\vv)$.
\end{enumerate}
\end{Thm}

For any chamber $\CC\subset\Stab^\dagger(X,\alpha)$, we thus get a map
\begin{equation} \label{eq:elltoAmp}
\ell_\CC \colon \CC \to \mathrm{Amp}(M_{\CC}(\vv)),
\end{equation}
where we used the notation $M_{\CC}(\vv)$ to denote the coarse moduli space $M_{\sigma}(\vv)$, 
independent of the choice $\sigma\in\CC$.
The main goal of this paper is to understand the global behavior of this map.

We recall one more result from \cite{BM:projectivity}, which will be crucial for our wall-crossing
analysis.
Let $\vv\in H^*_{\alg}(X,\alpha,\Z)$ be a \emph{primitive} vector with $\vv^2\geq -2$.
Let $\WW$ be a wall for $\vv$ and let $\sigma_0\in \WW$ be a generic stability condition on the
wall, namely it does not belong to any other wall.
We denote by $\sigma_+$ and $\sigma_-$ two generic stability conditions nearby $\WW$ in opposite chambers.
Then all $\sigma_{\pm}$-semistable objects are also $\sigma_0$-semistable.
Hence, $\ell_{\sigma_0}$ induces two nef divisors $\ell_{\sigma_0,+}$ and $\ell_{\sigma_0,-}$ on $M_{\sigma_+}(\vv)$ and $M_{\sigma_-}(\vv)$ respectively.

\begin{Thm}[{\cite[Theorem 1.4(a)]{BM:projectivity}}]\label{thm:contraction}
The divisors $\ell_{\sigma_0,\pm}$ are big and nef on $M_{\sigma_{\pm}}(\vv)$.
In particular, they are semi-ample, and induce birational contractions
\[
\pi^{\pm}\colon M_{\sigma_{\pm}}(\vv)\to \overline{M}_{\pm},
\]
where $\overline{M}_{\pm}$ are normal irreducible projective varieties.
The curves contracted by $\pi^{\pm}$ are precisely the curves of objects that are
S-equivalent with respect to $\sigma_0$.
\end{Thm}

\begin{Def}\label{def:TypeOfWalls}
We call a wall $\WW$:
\begin{enumerate}
\item a \emph{fake wall}, if there are no curves in $M_{\sigma_{\pm}}(\vv)$ of objects that are S-equivalent
to each other with respect to $\sigma_0$;
\item a \emph{totally semistable wall}, if $M^{st}_{\sigma_0}(\vv)=\emptyset$;
\item a \emph{flopping wall}, if we can identify $\overline{M}_+=\overline{M}_-$ and the induced
map $M_{\sigma_+}(\vv)\dashrightarrow M_{\sigma_-}(\vv)$ induces a flopping contraction;
\item a \emph{divisorial wall}, if the morphisms $\pi^{\pm}\colon M_{\sigma_{\pm}}(\vv)\to \overline{M}_{\pm}$ are both divisorial contractions.
\end{enumerate}
\end{Def}

By \cite[Theorem 1.4(b)]{BM:projectivity}, if $\WW$ is not a fake wall and $M^{st}_{\sigma_0}(\vv)\subset M_{\sigma_{\pm}}(\vv)$ has complement of codimension at least two, then $\WW$ is a flopping wall.
We will classify walls in Theorem \ref{thm:walls}.

\section{Review: basic facts about hyperk\"ahler varieties}\label{sec:ReviewHK}

In this section we give a short review on hyperk\"ahler manifolds.
The main references are \cite{Beauville:HK,GrossHuybrechtsJoyce,Eyal:survey}.

\begin{Def}\label{def:HK}
A \emph{projective hyperk\"ahler manifold} is a simply connected smooth projective variety $M$ such that $H^0(M,\Omega^2_M)$ is one-dimensional, spanned by an everywhere non-degenerate holomorphic $2$-form.
\end{Def}

The N\'eron-Severi group of a hyperk\"ahler manifold carries a natural bilinear form, called the \emph{Fujiki-Beauville-Bogomolov form}.
It is induced by a quadratic form on the whole second cohomology group $q:H^2(M,\Z)\to\Z$, which is
primitive of signature $(3,b_2(M)-3)$. It satisfies the Fujiki relation
\begin{equation}\label{eq:BBform}
\int_M \alpha^{2n} = F_M \cdot q(\alpha)^n,\qquad \alpha\in H^2(M,\Z),
\end{equation}
where $2n=\dim M$ and $F_M$ is the \emph{Fujiki constant}, which depends only on the deformation
type of $M$.
We will mostly use the notation $(\blank,\blank):=q(\blank,\blank)$ for the induced bilinear form on $\mathrm{NS}(M)$.

The Hodge structure $\left(H^2(M,\Z),q\right)$ behaves similarly to the case of a K3 surface.
For example, by \cite{Verbitsky:torelli}, there is a weak global Hodge theoretic Torelli theorem for
(deformation equivalent) hyperk\"ahler manifolds.

Moreover, some positivity properties of divisors on $M$ can be rephrased in terms of $q$.
We first recall a few basic definitions on cones of divisors.

\begin{Def}\label{def:MoveableBigPositiveEffDivisors}
An integral divisor $D\in\NS(M)$ is called
\begin{itemize}
\item \emph{big}, if its Iitaka dimension is maximal;
\item \emph{movable}, if its stable base-locus has codimension $\geq 2$;
\item \emph{strictly positive}, if $(D,D) > 0$ and $(D, A) > 0$ for a fixed ample class $A$ on $M$.
\end{itemize}
\end{Def}

The real (not necessarily closed) cone generated by big (resp., movable, strictly positive, effective) integral divisors will be denoted by $\Bigc(M)$ (resp., $\Mov(M)$, $\Pos(M)$, $\mathrm{Eff}(M)$).
We have the following inclusions:
\begin{align*}
&\Pos(M)\subset \Bigc(M)\subset\mathrm{Eff}(M)\\
&\Nef(M)\subset\overline{\Mov}(M)\subset  \overline{\Pos}(M)\subset\overline{\Bigc}(M)=\overline{\mathrm{Eff}}(M).
\end{align*}
The only non-trivial inclusion is $\Pos(M)\subset \Bigc(M)$, which follows from \cite[Corollary 3.10]{Huybrechts:compactHyperkaehlerbasic}.
Divisors in $\overline{\Pos}(M)$ are called \emph{positive}.

We say that an irreducible divisor $D \subset M$ is \emph{exceptional} if there is a birational map $\pi \colon M
\dashrightarrow M'$ contracting $D$. Using the Fujiki relations, one proves $D^2 < 0$ and $(D,
E) \ge 0$ for every movable divisor $E$ \cite[Section 1]{Huybrechts:compactHyperkaehlerbasic}.
We let $\rho_D$ be the reflection at $D$, i.e., the linear involution of $\NS(M)_\Q$ fixing $D^\perp$ and sending $D$ to $-D$.

\begin{Prop}[\cite{Eyal:prime-exceptional}] \label{prop:Weylmovablecone}
The reflection $\rho_D$ at an irreducible exceptional divisor is an integral involution of $\NS(M)$. 
Let $W_{\mathrm{Exc}}$ be the Weyl group generated by such reflections $\rho_D$.
The cone $\Mov(M) \cap \Pos(M)$ of big movable divisors is the fundamental chamber, for
the action of $W_{\mathrm{Exc}}$ on $\Pos(M)$, given by $(D, \blank) \ge 0$ for every exceptional divisor $D$.
\end{Prop} 
The difficult claim is the integrality of $\rho_D$; in our case, we could also deduce it from our
classification of divisorial contractions in Theorem \ref{thm:walls}.
As explained in \cite[Section 6]{Eyal:survey}, the remaining statements follow
from Zariski decomposition for divisors \cite{Boucksom:Zariskidecomposition} and standard results
about Weyl group actions on hyperbolic lattices.

\begin{Def}\label{def:LagrangianFibration}
Let $M$ be a projective hyperk\"ahler manifold of dimension $2n$.
A \emph{Lagrangian fibration} is a surjective morphism with connected fibers $h\colon M\to B$, where $B$
is a smooth projective variety, such that the generic fiber is Lagrangian with respect to the
symplectic form $\omega\in H^0(M,\Omega^2_M)$.
\end{Def}

By the Arnold-Liouville Theorem, any smooth fiber of a Lagrangian fibration is an abelian variety of dimension $n$.
Moreover:

\begin{Thm}[{\cite{Matsushita:Fibrations, Matsushita:Addendum} and \cite{Hwang:fibrationsbase}}]\label{thm:MatsushitaHwang}
Let $M$ be a projective hyperk\"ahler manifold of dimension $2n$.
Let $B$ be a smooth projective variety of dimension $0<\dim B < 2n$ and let $h\colon M\to B$ be a surjective morphism with connected fibers.
Then $h$ is a Lagrangian fibration, and $B\cong \P^n$.
\end{Thm}

This result explains the importance of Conjecture \ref{conj:SYZ}. In addition, 
the existence of a Lagrangian fibration is equivalent to the existence of a single Lagrangian
torus in $M$ (see \cite{GLR:Lagrangian1, HW:Lagrangian,Matsushita:BeauvilleConj}, based on previous results in
\cite{Amerik:BeauvilleConjDim4,GLR:Lagrangian2}).

The examples of hyperk\"ahler manifolds we will consider are moduli spaces of stable complexes,
as explained by the theorem below. It has been proven for moduli of sheaves
in \cite[Sections 7 \& 8]{Yoshioka:Abelian}, and generalized to Bridgeland stability
conditions in \cite[Theorem 6.10 \& Section 7]{BM:projectivity}:

\begin{Thm}[Huybrechts-O'Grady-Yoshioka]\label{thm:ModuliSpacesAreHK}
Let $(X,\alpha)$ be a twisted K3 surface and let $\vv\in H^*_{\alg}(X,\alpha,\Z)$ be a primitive vector with $\vv^2\geq -2$.
Let $\sigma\in\Stab^\dagger(X,\alpha)$ be a generic stability condition with respect to $\vv$.
Then:
\begin{enumerate}
\item $M_{\sigma}(\vv)$ is a projective hyperk\"ahler manifold, deformation-equivalent to the
Hilbert scheme of points on any K3 surface.
\item  The Mukai morphism induces an isomorphism
\begin{itemize}
\item $\theta_{\sigma, \vv}\colon \vv^\perp \xrightarrow{\sim} \mathrm{NS}(M_{\sigma}(\vv))$, if $\vv^2>0$;
\item $\theta_{\sigma, \vv}\colon \vv^\perp/\vv \xrightarrow{\sim} \mathrm{NS}(M_{\sigma}(\vv))$, if $\vv^2=0$.
\end{itemize}
Under this isomorphism, the quadratic Beauville-Bogomolov form for $\mathrm{NS}(M_{\sigma}(\vv))$
coincides with the quadratic form of the Mukai pairing on $(X,\alpha)$.
\end{enumerate}
\end{Thm}
Here $\theta_{\sigma, \vv}$ is the Mukai morphism as in Remark \ref{rmk:comparison}, induced
by a (quasi-)universal family. We will often drop $\sigma$ or $\vv$ from the notation.
It extends to an isomorphism of Hodge structures, identifying
the orthogonal complement $\vv^{\perp, \tr}$ inside the
whole cohomology $H^*(X,\alpha,\Z)$ (rather than its algebraic part) with $H^2(M_{\sigma}(\vv), \Z)$. 
The following result is Corollary 9.9 in \cite{Eyal:survey} for the untwisted case $\alpha = 1$; by 
deformation techniques, the result also holds in the twisted case:
\begin{Thm}[{\cite{Verbitsky:torelli}, \cite{Eyal:survey}}]
For $\vv$ primitive and $\vv^2>0$, the embedding $H^2(M_{\sigma}(\vv), \Z) \cong \vv^{\perp, \tr}\into H^*(X,\alpha, \Z)$ of integral Hodge
structures determines the birational class of $M_{\sigma}(\vv)$.
\end{Thm}
However, as indicated in the introduction, we only need the implication that 
birational moduli spaces have isomorphic extended Hodge structures.

We will also the need the following special case of a result by Namikawa and Wierzba:

\begin{Thm}[{\cite[Theorem 1.2 (ii)]{Wierzba:contractions} and \cite[Proposition
1.4]{Namikawa:deformation}}] \label{thm:NamikawaWierzba}
Let $M$ be a projective hyperk\"ahler manifold of dimension $2n$, and let $\overline{M}$
be a projective normal variety.  Let $\pi \colon M\to\overline{M}$ be a birational projective morphism.
We denote by $S_i$ the set of points $p\in \overline{M}$ such that $\dim \pi^{-1}(p)=i$.
Then $\dim S_i \leq 2n-2i$.

In particular, if $\pi$ contracts a divisor $D\subset M$, we must have $\dim \pi(D) = 2n-2$.
\end{Thm}

Consider a non-primitive vector $\vv$. As shown by O'Grady and Kaledin-Lehn-Sorger, the moduli space
$M_{\sigma}(\vv)$ can still be thought of as a singular hyperk\"ahler manifold, in the following
sense:

\begin{Def}\label{def:SingularHK}
A normal projective variety $M$ is said to have \emph{symplectic singularities} if
\begin{itemize}
\item the smooth part $M_\mathrm{reg}\subset M$ admits a symplectic 2-form $\omega$, such that
\item for any resolution $f\colon \widetilde M\to M$, the pull-back of $\omega$ to
$f^{-1}(M_\mathrm{reg})$ extends to a holomorphic form on $\widetilde M$.
\end{itemize}
\end{Def}

Given a hyperk\"ahler manifold $M$ and a dominant rational map $M\dashrightarrow \overline{M}$, where $\overline{M}$ is a normal projective variety with symplectic singularities, then
it follows from the definitions that  $\dim(M)=\dim(\overline{M})$. This explains the relevance of
the following theorem; our results in \cite{BM:projectivity} reduce it to the case of moduli of sheaves:
\begin{Thm}[{\cite{OGrady} and \cite{KLS:SingSymplecticModuliSpaces}}]
\label{thm:KLS}
Let $(X,\alpha)$ be a twisted K3 surface and let $\vv=m\vv_0\in H^*_{\alg}(X,\alpha,\Z)$ be a Mukai vector with $\vv_0$ primitive and $\vv_0^2\geq 2$.
Let $\sigma\in\Stab^\dagger(X,\alpha)$ be a generic stability condition with respect to $\vv$.
Then $M_{\sigma}(\vv)$ has symplectic singularities.
\end{Thm}

\section{Harder-Narasimhan filtrations in families}\label{sec:HNfamily}

In this section, we will show that results by Abramovich-Polishchuk and Toda imply the existence
of HN filtrations in families, see Theorem \ref{thm:HNfamily}.

The results we present will work as well in the twisted context; to simplify notation, we only state
the untwisted case.
Let $Y$ be a smooth projective variety over $\C$.
We will write $\Dqc(Y)$ for the unbounded derived of quasi-coherent sheaves. Pick a lattice
$\Lambda$ and $\vv$ for the bounded derived category $\Db(Y)$ as in Definition \ref{def:Bridgeland}, 
and let $\sigma$ be a Bridgeland stability on $\Db(Y)$.

\begin{Def} \label{def:openness}
We say $\sigma$ satisfies \emph{openness of stability} if the following condition holds: for any
scheme $S$ of finite type over $\C$, and for any $\EE \in \Db(S\times Y)$ such that its derived
restriction $\EE_s$ is a $\sigma$-semistable object of $\Db(Y)$ for some $s \in S$,
there exists an open neighborhood
$s \in U \subset S$ of $s$, such that $\EE_{s'}$ is $\sigma$-semistable for all $s' \in U$.
\end{Def}

\begin{Thm}[{\cite[Section 3]{Toda:K3}}] \label{thm:Todaopenness}
Openness of stability holds when $Y$ is a K3 surface and $\sigma$ is
a stability condition in the connected component $\Stab^\dag(Y)$.\footnote{In \cite[Section 3]{Toda:K3}, this Theorem is only stated for
families $\EE$ satisfying $\Ext^{<0}(\EE_s, \EE_s) = 0$ for all $s \in S$. However, Toda's proof
in Lemma 3.13 and Proposition 3.18 never uses that assumption.}
\end{Thm}

\begin{Thm} \label{thm:HNfamily}
Let $\sigma = (Z, \AA) \in \Stab(Y)$ be an algebraic stability condition satisfying openness of stability.
Assume we are given an irreducible variety $S$ over $\C$, and an object
$\EE \in \Db(S\times Y)$. Then there exists a system of maps
\begin{equation} \label{eq:HNfamily}
0 = \EE^0 \to \EE^1 \to \EE^2 \to \dots \to \EE^m = \EE
\end{equation}
in $\Db(S \times Y)$, and an open subset $U \subset S$ with the following property: 
for any $s \in U$, the derived restriction of the system of maps \eqref{eq:HNfamily}
\[
0 = \EE_s^0 \to \EE_s^1 \to \EE_s^2 \to \dots \to \EE_s^m = \EE_s
\]
is the HN filtration of $\EE_s$.
\end{Thm}

The proof is based on the notion of constant family of t-structures due to Abramovich and Polishchuk,
constructed in \cite{Abramovich-Polishchuk:t-structures} (in case $S$ is smooth) and
\cite{Polishchuk:families-of-t-structures} (in general).

Throughout the remainder of this section, we will assume that $\sigma$ and $S$ satisfy the
assumptions of Theorem \ref{thm:HNfamily}.
A t-structure is called \emph{close to Noetherian} if it can be obtained via tilting from a
t-structure whose heart is Noetherian.
For $\phi \in \R$, the category 
$\PP((\phi-1, \phi]) \subset \Db(Y)$ is the heart of a close to Noetherian bounded t-structure on $Y$ given by
$\DD^{\le 0} = \PP((\phi -1, +\infty))$ and $\DD^{\ge 0} = \PP((-\infty, \phi])$ (see the example
discussed at the end of \cite[Section 1]{Polishchuk:families-of-t-structures}).
In this situation, Abramovich and Polishchuk's work
induces a bounded t-structure $(\DD^{\le 0}_S, \DD^{\ge 0}_S)$ on $\Db(S \times Y)$; we
paraphrase their main results as follows:

\begin{Thm}[{\cite{Abramovich-Polishchuk:t-structures, Polishchuk:families-of-t-structures}}] \label{thm:APP}
Let $\AA$ be the heart of a close to Noetherian bounded t-structure $(\DD^{\le 0}, \DD^{\ge 0})$ on $\Db(Y)$.
Denote by $\AA_{qc} \subset \Dqc(Y)$ the closure of $\AA$ under infinite coproducts
in the derived category of quasi-coherent sheaves.
\begin{enumerate}
\item \label{enum:deftstruct}
For any scheme $S$ of finite type of $\C$ there is a close to Noetherian bounded t-structure
$(\DD_S^{\le 0}, \DD_S^{\ge 0})$
on $\Db(S\times Y)$, whose heart $\AA_S$ is characterized by
\[ \EE \in \AA_S \Leftrightarrow (p_Y)_* \left(\EE |_{Y \times U}\right) \in \AA_{qc}
\quad \text{for every open affine $U \subset S$}
\]
\item \label{enum:sheaf}
The above construction defines a sheaf of t-structures over $S$: when $S = \bigcup_i U_i$
is an open covering of $S$, then $\EE \in \AA_S$ if and only if
$\EE |_{Y \times U_i} \in \AA_{U_i}$ for every $i$. In particular, for $i \colon U \subset S$ open, the
restriction functor $i^*$ is t-exact.
\item \label{enum:pushforwardexact}
When $i \colon S' \subset S$ is a closed subscheme, then $i_*$ is t-exact, and $i^*$ is t-right exact.
\end{enumerate}
\end{Thm}

We briefly comment on the statements that are not explicitly mentioned in
\cite[Theorem 3.3.6]{Polishchuk:families-of-t-structures}: From part (i) of 
\cite[Theorem 3.3.6]{Polishchuk:families-of-t-structures}, it follows that the t-structure
constructed there on $\mathrm{D}(S \times Y)$ descends to a bounded t-structure on 
$\Db(S \times Y)$. To prove that the push-forward in claim \eqref{enum:pushforwardexact} is t-exact,
we first use the sheaf property to reduce to the case where $S$ is affine; in this case, the claim
follows by construction. By adjointness, it follows that $i^*$ is t-right exact.

For an algebraic stability condition $\sigma = (Z, \PP)$ on $\Db(Y)$ and a phase
$\phi \in \R$, we will from now on
denote its associated t-structure by $\PP(>\phi) = \DD^{\le -1}$, $\PP(\le \phi) = \DD^{\ge 0}$, and
the associated truncation functors by $\tau^{>\phi}, \tau^{\le \phi}$.
By \cite[Lemma 2.1.1]{Polishchuk:families-of-t-structures}, it induces a t-structure on
$\Dqc(Y)$, which we denote by $\PP_{qc}(>\phi), \PP_{qc}(\le \phi)$.
For the t-structure on $\Db(S \times Y)$ induced via Theorem \ref{thm:APP}, we will similarly write
$\PP_S(>\phi), \PP_S(\le \phi)$, and $\tau_S^{>\phi}, \tau_S^{\le \phi}$.

We start with a technical observation:
\begin{Lem} \label{lem:trivialbuttechnical}
The t-structures on $\Db(S \times Y)$ constructed via Theorem \ref{thm:APP} satisfy the following
compatibility relation:
\begin{equation} \label{eq:intersect}
\bigcap_{\epsilon > 0} \PP_S(\le \phi + \epsilon) = \PP_S(\le \phi).\end{equation}
\end{Lem}
\begin{proof}
Assume $\EE$ is in the intersection of the left-hand side of \eqref{eq:intersect}.
By the sheaf property, we may assume that $S$ is affine.
The assumption implies
$(p_Y)_* \EE \in \PP_{qc}(\le \phi + \epsilon)$ for all $\epsilon > 0$.

By \cite[Lemma 2.1.1]{Polishchuk:families-of-t-structures}, we can
describe $\PP_{qc}(\le \phi + \epsilon) \subset \Dqc(Y)$ as the right orthogonal complement of
$\PP(>\phi + \epsilon) \subset \Db(Y)$ inside $\Dqc(Y)$; thus we obtain
\begin{align*}
 \bigcap_{\epsilon > 0} \PP_{qc}(\le \phi + \epsilon)
= \bigcap_{\epsilon > 0} \bigl(\PP(> \phi + \epsilon) \bigr)^\perp
= \Bigl( \bigcup_{\epsilon > 0} \PP(>\phi + \epsilon) \Bigr)^\perp
= \Bigl( \PP(>\phi) \Bigr)^\perp
= \PP_{qc}(\le \phi).
\end{align*}
Hence $(p_Y)_* \EE \in \PP_{qc}(\le \phi)$, proving the lemma.
\end{proof}

We next observe that the truncation functors $\tau_S^{> \phi}, \tau_S^{\le \phi}$ induce
a slicing on $\Db(S\times Y)$. (See Definition \ref{def:slicing} for the notion of slicing on
a triangulated category.)

\begin{Lem} \label{lem:slicingS}
Assume that $\sigma = (Z, \PP)$ is an algebraic stability condition, and $\PP_S(> \phi),
\PP_S(\le \phi)$ are as defined above.
There is a slicing $\PP_S$ on $\Db(S\times Y)$ defined by
\[
\PP_S(\phi) = \PP_S(\le \phi) \cap \bigcap_{\epsilon > 0} \PP_S(> \phi - \epsilon).
\]
\end{Lem}
Note that $\PP_S(\phi)$ cannot be characterized by the analogue of Theorem
\ref{thm:APP}, part \eqref{enum:deftstruct}.
For example, consider the case where $Y$ is a curve and $(Z, \PP)$ the standard
stability condition corresponding to classical slope-stability in $\Coh Y$.
Then $\PP(1) \subset \Coh Y$ is the category of 
torsion sheaves, and $\PP_S(1) \subset \Coh S \times Y$ is the category of sheaves $\FF$ that are torsion
relative over $S$. However, for $U \subset S$ affine and a non-trivial family $\FF$, the
push-forward $(p_Y)_* \FF|_U$ is never a torsion sheaf.

\begin{proof}
By standard arguments, it is sufficient to construct a HN filtration for any object $\EE \in \AA_S := \PP_S(0, 1]$.
In particular, since $\sigma$ is algebraic, we can assume that both $\AA:=\PP(0,1]$ and $\AA_S$ are Noetherian.
For any $\phi \in (0, 1]$, we have $\PP_S(\phi, \phi + 1] \subset \langle \AA_S, \AA_S[1] \rangle$.
By \cite[Lemma 1.1.2]{Polishchuk:families-of-t-structures}, this induces a torsion pair
$(\TT_\phi, \FF_\phi)$ on $\AA_S$ with
\[
\TT_\phi = \AA_S \cap \PP_S(\phi, \phi + 1] \quad \text{and} \quad
\FF_\phi = \AA_S \cap \PP_S(\phi-1, \phi].\]
Let $T_\phi \into \EE \onto F_\phi$ be the induced short exact sequence in $\AA_S$. 
Assume $\phi < \phi'$; since $\FF_{\phi} \subset \FF_{\phi'}$, the surjection
$\EE \onto F_\phi$ factors via
$\EE \onto F_{\phi'} \onto F_\phi$. Since $\AA_S$ is Noetherian, the
set of induced quotients $\stv{F_\phi}{\phi \in (0, 1]}$ of $\EE$ must be finite.
In addition, if $F_{\phi} \cong F_{\phi'}$, we must also have
$F_{\phi''} \cong F_{\phi}$ for any $\phi'' \in (\phi, \phi')$. 

Thus, there exist real numbers $\phi_0 = 1 > \phi_1 > \phi_2 > \dots > \phi_l > \phi_{l+1} = 0$ such that
$F_\phi$ is constant for $\phi \in (\phi_{i+1}, \phi_i)$, but such that
$F_{\phi_i - \epsilon} \neq F_{\phi_i + \epsilon}$. Let us assume for simplicity that
$F_{\phi_1 + \epsilon} \cong \EE$; the other case is treated similarly by setting $F^1 = F_{\phi_1 +
\epsilon}$, and shifting all other indices by one.
For $i = 1, \dots, l$ we set
\begin{itemize}
\item $F^i := F_{\phi_i - \epsilon}$,
\item $\EE^i := \Ker (\EE \onto F^i)$, and
\item $A^i = \EE^i/\EE^{i-1}$.
\end{itemize}
We have $\EE^i \in \PP_S(> \phi_i - \epsilon)$ and
$\EE^{i-1} = \tau_S^{>\phi_i + \epsilon} \EE^i$ for all $\epsilon > 0$. Hence the quotient
$A^i$ satisfies, for all $\epsilon > 0$, 
\begin{itemize}
\item $A^i \in \PP_S(>\phi_i - \epsilon)$, 
\item $A^i \in \PP_S(\le \phi_i + \epsilon)$.
\end{itemize}
The latter implies $A^i \in \PP_S(\le \phi_i)$ by Lemma \ref{lem:trivialbuttechnical}. 
By definition, we obtain $A^i \in \PP_S(\phi_i)$.
Finally, we have $F^l \in \PP_S(0, 1] \cap \PP_S(\le \epsilon)$ for all $\epsilon > 0$. Using
Lemma \ref{lem:trivialbuttechnical} again, we obtain $F^l = 0$, and thus $\EE^l = \EE$.
Thus the $\EE^i$ induce a HN filtration as claimed.
\end{proof}

The following lemma is an immediate extension of
\cite[Proposition 3.5.3]{Abramovich-Polishchuk:t-structures}:

\begin{Lem}\label{lem:densestable}
Assume that $\EE \in \PP_S(\phi)$ for some $\phi \in \R$.
and that $\EE_s \neq 0$ for $s \in S$ generic.
Then there exists a dense subset $Z \subset S$, such that 
$\EE_s$ is semistable of phase $\phi$ for all $s \in Z$.
\end{Lem}
\begin{proof}
By \cite[Proposition 3.5.3]{Abramovich-Polishchuk:t-structures}, applied to the smooth locus
of $S$, there exists a dense
subset $Z \subset S$ such that $\EE_s \in \PP((\phi-1, \phi])$. Since
$\EE \in \PP_S(>\phi - \epsilon)$ for all $\epsilon > 0$, and since $i_s^*$ is t-right exact, we also have
$\EE_s \in \PP(>\phi - \epsilon)$ for all $\epsilon > 0$.
Considering the HN filtration of $\EE_s$, this shows that $\EE_s \in \PP(\phi)$ for all $s \in Z$.
\end{proof}

\begin{proof}[Proof of Theorem \ref{thm:HNfamily}]
The statement now follows easily from the above two lemmas. First of all, under
the assumption of openness of stability, the dense subset $Z$ of Lemma \ref{lem:densestable} may of
course be taken to be open.

Given any $\EE \in \Db(S\times Y)$, let 
\begin{equation} 
 0 = \EE^0 \to \EE^1 \to \dots \to \EE^m = \EE 
\end{equation}
be the HN filtration with respect to the slicing of Lemma \ref{lem:slicingS}, and
let $A^j$ be the HN filtration quotients fitting in the exact triangle
$\EE^{j-1} \to \EE^j \to A^j$. Let
$j_1, \dots, j_l$ be the indices for which the generic fiber $i_s^* A^j$ does not vanish,
and let $\phi_i$ be the phase of $A^{j_i}$. 
Then we claim that 
\begin{equation} \label{eq:EHNslicing}
0 = \EE^0 \to \EE^{j_1} \to \EE^{j_2} \to \dots \to \EE^m = \EE \end{equation}
has
the desired property. Indeed, there is an open subset $U$ such that for all $s \in U$,
the fibers
$A^{j_i}_s$ are semistable for all $i = 1, \dots, l$, and such that $A^j_s = 0$ for all 
$j \notin \{i_1, \dots, i_l\}$. Then, for each such $s$, the restriction
of the sequence of maps \eqref{eq:EHNslicing}
via $i_s^*$ induces a sequence of maps that satisfies all properties of a HN filtration.
\end{proof}

\section{The hyperbolic lattice associated to a wall}\label{sec:hyperbolic}

Our second main tool will be a rank two hyperbolic lattice associated to any wall.
Let $(X,\alpha)$ be a twisted K3 surface.
Fix a primitive vector $\vv \in H^*_\alg(X,\alpha,\Z)$ with $\vv^2 > 0$, and a wall $\WW$ of the chamber
decomposition with respect to $\vv$. 

\begin{Prop} \label{prop:HW}
To each such wall, let $\HH_\WW \subset H^*_\alg(X,\alpha,\Z)$ be the set of classes
\[ \ww \in \HH_\WW \quad \Leftrightarrow \quad \Im \frac{Z(\ww)}{Z(\vv)} = 0 \quad
\text{for all $\sigma = (Z, \PP) \in \WW$.} \]
Then $\HH_\WW$ has the following properties:
\begin{enumerate}
\item \label{enum:signature}
It is a primitive sublattice of rank two and of signature $(1, -1)$ (with respect to the restriction
of the Mukai form).
\item \label{enum:HNfactorsinHW}
Let $\sigma_+, \sigma_-$ be two sufficiently close and generic stability conditions on opposite sides
of the wall $\WW$, and consider any $\sigma_+$-stable object $E \in M_{\sigma_+}(\vv)$.
Then any HN filtration factor $A_i$ of $E$ with respect to $\sigma_-$ has
Mukai vector $\vv(A_i)$ contained in $\HH_\WW$.
\item \label{enum:semistablehasfactorsinHW}
If $\sigma_0$ is a generic stability condition on the wall $\WW$, the conclusion of the
previous claim also holds for any $\sigma_0$-semistable object $E$ of class $\vv$.
\item \label{enum:JHfactorsinHW}
Similarly, let $E$ be any object with $\vv(E) \in \HH_\WW$, and assume that it is
$\sigma_0$-stable for a generic stability condition $\sigma_0 \in \WW$.
Then every Jordan-H\"older factors of $E$ with respect to $\sigma_0$ will have Mukai vector contained
in $\HH_\WW$.
\end{enumerate}
\end{Prop}
The precise meaning of ``sufficiently close'' will become apparent in the proof.
\begin{proof}
The first two claims of \eqref{enum:signature} are evident. To verify the claim on the signature,
first note that by the assumption $\vv^2 > 0$, the lattice $\HH_\WW$ is either hyperbolic or
positive (semi-)definite. On the other hand,
consider a stability condition $\sigma = (Z, \AA)$ with $Z(\vv) = -1$.
Since $(\Im Z)^2 > 0$ by Theorem \ref{thm:Bridgeland_coveringmap}, since
$\HH_\WW$ is contained in the orthogonal complement of $\Im Z$, and since
the algebraic Mukai lattice has signature $(2, \rho(X))$, this leaves the hyperbolic
case as the only possibility.

In order to prove the remaining claims, consider an $\epsilon$-neighborhood $B_\epsilon(\tau)$
of a generic stability condition $\tau \in \WW$, with $0 < \epsilon \ll 1$.
Let $\SSS_{\vv}$ be the set of objects $E$ with $\vv(E) = \vv$, and that are semistable for
some stability condition in $B_\epsilon(\tau)$.
Let $\UUU_{\vv}$ be the set of classes $\uu \in H^*_\alg(X,\alpha,\Z)$ that can appear as Mukai
vectors of Jordan-H\"older factors of $E \in \SSS_{\vv}$, for any stability condition
$(Z', \AA') \in B_\epsilon(\tau)$.  As shown in the
proof of local finiteness of walls (see \cite[Proposition 9.3]{Bridgeland:K3} or \cite[Proposition
3.3]{localP2}), the set $\UUU_\vv$ is finite; indeed, such a class would have to satisfy
$\abs{Z'(\uu)} < \abs{Z'(\vv)}$. Hence, the union of all walls for all classes in $\UUU_\vv$ is still
locally finite. 

To prove claim \eqref{enum:HNfactorsinHW}, we may
assume that $\WW$ is the only wall separating $\sigma_+$ and $\sigma_-$, among all walls for classes
in $\UUU_{\vv}$. Let $\sigma_0 = (Z_0, \PP_0) \in \WW$ be a generic stability condition in the wall separating
the chambers of $\sigma_+, \sigma_-$.  It follows that $E$ and all $A_i$ are
$\sigma_0$-semistable of the same phase, i.e.
$\Im \frac{Z_0(\vv(A_i))}{Z_0(\vv)} = 0$.
Since this argument works for generic $\sigma_0$, we must have
$\vv(A_i) \in \HH_\WW$ by the definition of $\HH_\WW$.

Claim \eqref{enum:semistablehasfactorsinHW} follows from the same discussion, and
\eqref{enum:JHfactorsinHW} similarly by considering the set of all walls for the classes
$\UUU_{\vv(E)}$ instead of $\UUU_{\vv}$.
\end{proof}

Our main approach is to characterize which hyperbolic lattices $\HH \subset H^*_\alg(X,\alpha,\Z)$ correspond to a wall, and
to determine the type of wall purely in terms of $\HH$. We start by making the following definition:

\begin{Def} \label{def:potentialwall}
Let $\HH \subset H^*_\alg(X,\alpha,\Z)$ be a primitive rank two hyperbolic sublattice containing $\vv$. A
\emph{potential wall} $\WW$ associated to $\HH$ is a
connected component of the real codimension one submanifold of stability conditions
$\sigma = (Z, \PP)$ which satisfy the condition that $Z(\HH)$ is contained in a line. 
\end{Def}

\begin{Rem} \label{rem:HW}
The statements of Proposition \ref{prop:HW} are still valid when $\WW$ is a potential
wall as in the previous definition.
\end{Rem}

\begin{Def} \label{def:PW}
Given any hyperbolic lattice $\HH \subset H^*_\alg(X,\alpha,\Z)$ of rank two containing $\vv$, we denote
by $P_\HH \subset \HH \otimes \R$ the cone generated by integral classes $\uu \in \HH$ with
$\uu^2 \ge 0$ and $(\vv, \uu) > 0$. We call $P_\HH$ the \emph{positive cone} of $\HH$, and 
a class in $P_\HH \cap \HH$ is called a \emph{positive class}.
\end{Def}
The condition $(\vv, \uu) > 0$ just picks out one of the two components of the set of real
classes with $\uu^2 > 0$. Observe that $P_\HH$ can be an open or closed cone, depending on whether
the lattice contains integral classes $\ww$ that are isotropic: $\ww^2 = 0$.

\begin{Prop} \label{prop:CW}
Let $\WW$ be a potential wall associated to a hyperbolic rank two sublattice
$\HH \subset H^*_\alg(X,\alpha,\Z)$.
For any $\sigma = (Z, \PP) \in \WW$, let $C_\sigma \subset \HH \otimes \R$ be the cone generated by classes
$\uu \in \HH$ satisfying the two conditions
\[ \uu^2 \ge -2 \quad \text{and} \quad \Re \frac{Z(\uu)}{Z(\vv)} > 0.
\]
This cone does not depend on the choice of $\sigma \in \WW$, and it contains $P_\HH$.

If $\uu \in C_\sigma$, then there exists a semistable object of class $\uu$ for every
$\sigma' \in \WW$. If $\uu \notin C_\sigma$, then there does not exist a semistable object of class
$\uu$ for generic $\sigma' \in \WW$.
\end{Prop}
From here on, we will write $C_\WW$ instead of $C_\sigma$, and call it the cone of effective classes
in $\HH$. Given two different walls $\WW_1$, $\WW_2$, the corresponding effective cones
$C_{\WW_1}, C_{\WW_2}$ will only differ by spherical classes.

\begin{proof}
If $\uu^2 \ge -2$, then by Theorem \ref{thm:nonempty} there exists a $\sigma$-semistable object
of class $\uu$ for every $\sigma = (Z, \PP) \in \WW$. Hence $Z(\uu) \neq 0$, i.e, we cannot
simultaneously have $\uu \in \HH$ (which implies $\Im \frac{Z(\uu)}{Z(\vv)} = 0$) and
$\Re \frac{Z(\uu)}{Z(\vv)} = 0$. Therefore, the condition $\Re \frac{Z(\uu)}{Z(\vv)} > 0$ is invariant
under deforming a stability condition inside $\WW$, and $C_\sigma$ does not depend on
the choice of $\sigma \in \WW$.

Now assume for contradiction that $P_\HH$ is not contained in $C_\WW$. Since $\vv \in C_\WW$,
this is only possible if there is a real class
$\uu \in P_\HH$ with $\Re \frac{Z(\uu)}{Z(\vv)} = 0$; after deforming $\sigma \in \WW$ slightly, we may
assume $\uu$ to be integral. As above, this implies $Z(\uu) = 0$, in contradiction to the existence 
of a $\sigma$-semistable object of class $\uu$.

The statements about existence of semistable objects follow directly from
Theorem \ref{thm:nonempty}.
\end{proof}

\begin{Rem}\label{rmk:GenericOnTheWall}
Note that by construction, $C_\WW \subset \HH \otimes \R$ is strictly contained in a half-plane. In
particular, there are only finitely many classes in $C_\WW \cap \bigl(\vv - C_\WW\bigr) \cap \HH$ (in other
words, effective classes $\uu$ such that
$\vv - \uu$ is also effective).

We will use this observation throughout in order to freely make
genericity assumptions: a generic stability condition $\sigma_0 \in \WW$ will be a stability condition
that does not lie on any additional wall (other than $\WW$) for any of the above-mentioned classes.
Similarly, by stability conditions $\sigma_+, \sigma_-$ \emph{nearby $\sigma_0$} we will mean
stability conditions that lie in the two chambers adjacent to $\sigma_0$ for the wall-and-chamber decompositions
with respect to any of the classes in $C_\WW \cap \bigl(\vv - C_\WW\bigr) \cap \HH$.
\end{Rem}

The behavior of the potential wall $\WW$ is completely determined by $\HH$ and its
effective cone $C_\WW$:
\begin{Thm} \label{thm:walls}
Let $\HH \subset H^*_\alg(X,\alpha,\Z)$ be a primitive hyperbolic rank two sublattice containing $\vv$.
Let $\WW \subset \Stab^\dag(X,\alpha)$ be a potential wall associated to $\HH$ (see Definition \ref{def:potentialwall}).

The set $\WW$ is a totally semistable wall if and only if there exists either an isotropic class
$\ww \in \HH$ with $(\vv, \ww) = 1$, or an effective spherical class $\ss \in C_\WW \cap \HH$ with $(\ss, \vv) < 0$.
In addition:
\begin{enumerate}
\item \label{enum:niso-divisorial}
The set $\WW$ is a wall inducing a divisorial contraction if one of the following three conditions
hold:
\begin{description*}
\item[(Brill-Noether)] there exists a spherical class $\ss \in \HH$ with $(\ss, \vv) = 0$, or
\item[(Hilbert-Chow)] there exists an isotropic class $\ww \in \HH$
with $(\ww, \vv) = 1$, or
\item[(Li-Gieseker-Uhlenbeck)] there exists an isotropic class $\ww \in \HH$
with $(\ww, \vv) = 2$. 
\end{description*}
\item \label{enum:niso-flop}
Otherwise, if $\vv$ can be written as the sum $\vv = \aa + \bb$ of two positive\footnote{In the
sense of Definition \ref{def:PW}.} classes, or if there exists
a spherical class $\ss \in \HH$ with $0 < (\ss, \vv) \le \frac{\vv^2}2$,
then $\WW$ is a wall corresponding to a flopping contraction.
\item \label{enum:niso-fakeornothing}
In all other cases, $\WW$ is either a fake wall (if it is a totally semistable wall), or it is
not a wall.
\end{enumerate}
\end{Thm}

The Gieseker-Uhlenbeck morphism from the moduli space of Gieseker semistable sheaves to slope-semistable vector bundle was constructed in \cite{JunLi:Uhlenbeck}.
Many papers deal with birational transformations between moduli spaces of twisted Gieseker
semistable sheaves, induced by variations of the polarization.
In particular, we refer to \cite{Thaddeus:GIT-flips, DolgachevHu:Variation} for the general theory of variation of GIT quotients and \cite{EllingsrudGottsche:Variation, FriedmanQin:Variation, MatsukiWenthworth:TwistedVariation} for the case of sheaves on surfaces.
Theorem \ref{thm:walls} can be thought as a generalization and completion of these results in the case of K3 surfaces.

\subsection*{Proof outline}
The proof of the above theorem will be broken into four sections. We will distinguish
two cases, depending on whether $\HH$ contains isotropic classes:
\begin{Def} \label{def:isotropicwall}
We say that $\WW$ is an \emph{isotropic} wall if $\HH_\WW$ contains an isotropic class.
\end{Def}

In Section \ref{sec:noniso-totsemistable}, we analyze totally semistable non-isotropic walls,
and Section \ref{sec:divisorialcontraction} describes non-isotropic walls corresponding to divisorial
contractions. In Section
\ref{sec:iso}, we use a Fourier-Mukai transform to reduce the treatment of isotropic walls to
the well-known behavior of the Li-Gieseker-Uhlenbeck morphism from the Gieseker moduli space to the Uhlenbeck
space. For the remaining cases, Section \ref{sec:flopping} describes whether it is a 
flopping wall, a fake walls, or no wall at all.

To give an example of the strategy of our proof, consider a wall with
a divisor $D \subset
M_{\sigma_+}(\vv)$ of objects that become strictly semistable on the wall.  We
use the contraction morphism $\pi^+$ of Theorem \ref{thm:contraction}; Theorem
\ref{thm:NamikawaWierzba} implies $\dim \pi^+(D) \ge \dim D-1 = \vv^2$. Recall 
that $\pi^+$ contracts a curve if the associated objects
have the same Jordan-H\"older factors. Intuitively, this means that the sum of the
dimensions of the moduli spaces parameterizing the Jordan-H\"older factors is at least
$\vv^2$; a purely lattice-theoretic argument (using that moduli spaces
always have expected dimension) leads to a contradiction except
in the cases listed in the Theorem.  To make this
argument rigorous, we use the relative Harder-Narasimhan filtration with respect to $\sigma_-$ in the
family parameterized by $D$; it induces a rational map from $D$ to a product of moduli
spaces of $\sigma_-$-stable objects. 
The most technical part of our arguments deals with totally semistable walls induced by a spherical class. 
We use a sequence of spherical twists to reduce to the previous cases,
see Proposition \ref{prop:sphericaltotallysemistable}.

\section{Totally semistable non-isotropic walls}
\label{sec:noniso-totsemistable}

In this section, we will analyze \emph{totally semistable walls}; while some of our intermediate
results hold in general, we will focus on the case where $\HH$ does not contain an isotropic class.
The relevance of this follows from Theorem \ref{thm:nonempty}: in this case, 
if the dimension of a moduli space $M_{\sigma}(\mathbf{u})$ is positive, then it is given by $\mathbf{u}^2 + 2$.

We will first describe the possible configurations of effective spherical classes in $C_\WW$, and of
corresponding spherical objects with $\vv(S) \in \HH_\WW$.

We start with the following classical argument of Mukai (cfr.~\cite[Lemma 5.2]{Bridgeland:K3}):

\begin{Lem}[Mukai]\label{lem:Mukai}
Consider an exact sequence
$0\to A \to E \to B \to 0$
in the heart of a bounded t-structure $\AA \subset \Db(X, \alpha)$ with $\Hom(A,B)=0$.
Then
\begin{equation*}
\mathrm{ext}^1(E,E)\geq\mathrm{ext}^1(A,A) + \mathrm{ext}^1(B,B).
\end{equation*}
\end{Lem}

The following is a well-known consequence of Mukai's lemma (cfr.~\cite[Section 2]{HMS:generic_K3s}):

\begin{Lem} \label{lem:JHspherical}
Assume that $S$ is a $\sigma$-semistable object with $\Ext^1(S, S) = 0$.
Then any Jordan-H\"older filtration factor of $S$ is spherical.
\end{Lem}
\begin{proof}
Pick any stable subobject $T \subset S$ of the same phase.  Then there exists a short exact sequence
$ \widetilde T \into S \onto R $
with the following two properties:
\begin{enumerate}
\item The object $\widetilde T$ is an iterated extension of $T$.
\item $\Hom(T, R) = 0$.
\end{enumerate}
Indeed, this can easily be constructed inductively: we let $R_1 = S/T$. If $\Hom(T, S/T) = 0$,
the subobject $\widetilde T = T$ already has the desired properties. Otherwise, any
non-zero morphism $T \to R_1$ is necessarily injective; if we let $R_2$ be its quotient, then the
kernel of $S \onto R_2$ is a self-extension of $T$, and we can proceed inductively.

It follows that $\Hom(\widetilde T, R) = 0$, and we can apply Lemma \ref{lem:Mukai} to conclude
that $\Ext^1(\widetilde T, \widetilde T) = 0$. Hence
$(\vv(\widetilde T), \vv(\widetilde T)) < 0$, which also implies $(\vv(T), \vv(T)) < 0$. Thus $\vv(T)$ is
spherical, too.

The lemma follows by induction on the length of $S$.
\end{proof}

\begin{Prop} \label{prop:stablespherical}
Let $\WW$ be a potential wall associated to the primitive hyperbolic lattice $\HH$, and
let $\sigma_0 = (Z_0, \PP_0) \in \WW$ be a generic stability condition with
$Z_0(\HH) \subset \R$. Then $\HH$ and $\sigma_0$ satisfy
one of the following mutually exclusive conditions:
\begin{enumerate}
\item \label{case:0spherical} The lattice $\HH$ does not admit a spherical class.
\item \label{case:1spherical} The lattice $\HH$ admits, up to sign, a unique spherical class, and
there exists a unique $\sigma_0$-stable object $S \in \PP_0(1)$ with $\vv(S) \in \HH$.
\item \label{case:2spherical} The lattice $\HH$ admits infinitely many spherical classes, and there exist
exactly two $\sigma_0$-stable spherical objects $S, T \in \PP_0(1)$ with $\vv(S), \vv(T) \in \HH$.
In this case, $\HH$ is not isotropic.
\end{enumerate}
\end{Prop}

\begin{proof}
Given any spherical class, $\ss \in \HH$, then by Theorem \ref{thm:nonempty}, there exists a
$\sigma_0$-semistable object $S$ with $\vv(S) = \ss$ and $S \in \PP_0(1)$.  If $\HH$ admits a unique
spherical class, then by Proposition \ref{prop:HW} and Lemma \ref{lem:JHspherical}, $S$ must be
stable.

Hence it remains to consider the case where $\HH$ admits two linearly independent spherical classes.
If we consider the Jordan-H\"older filtrations of $\sigma_0$-semistable objects of the 
corresponding classes, and apply Proposition \ref{prop:HW} and Lemma \ref{lem:JHspherical},
we see that there must be two $\sigma_0$-stable objects $S, T$ whose Mukai vectors 
are linearly independent. 

Now assume that there are three stable spherical objects $S_1, S_2, S_3 \in \PP_0(1)$, and
let $\ss_i = \vv(S_i)$. Since they are stable of the same phase, we have 
$\Hom(S_i, S_j) = 0$ for all $i \neq j$, as well as $\Ext^k(S_i, S_j) = 0$ for $k < 0$.
Combined with Serre duality, this implies $(\ss_i, \ss_j) = \ext^1(S_i, S_j) \ge 0$.

However, a rank two lattice of signature $(1, -1)$ can never contain
three spherical classes $\ss_1, \ss_2, \ss_3$ with $(\ss_i, \ss_j) \ge 0$ for $i \neq j$. Indeed, we may assume
that $\ss_1, \ss_2$ are linearly independent. Let $m := (\ss_1, \ss_2) \ge 0$; since $\HH$ has signature
$(1, -1)$, we have $m \ge 3$. If we write $\ss_3 = x \ss_1 + y \ss_2$, we get the following implications:
\begin{eqnarray*}
(\ss_1, \ss_3 ) \ge 0 & \Rightarrow & y \ge \frac 2m x \\
(\ss_2, \ss_3 ) \ge 0 & \Rightarrow & y \le \frac m2 x \\
(\ss_3, \ss_3) = -2 & \Rightarrow & -2x^2 + 2m xy - 2y^2 < 0
\end{eqnarray*}
However, by solving the quadratic equation for $y$, it is immediate that
the term in the last inequality is positive in the range $\frac 2m x \le y \le \frac m2 x$ (see also
Figure \ref{fig:HWplanenohyperbola}).

Finally, if $\HH$ admits two linearly independent spherical class $\ss, \tt$, then the group
generated by the associated reflections $\rho_{\ss}, \rho_{\tt}$ is infinite; the orbit of $\ss$
consists of infinitely many spherical classes. Additionally, an isotropic class would be a rational
solution of $-2x^2 + 2m xy - 2y^2 = 0$, but the discriminant $m^2 - 4$ can never be a square when 
$m$ is an integer $m \ge 3$.
\end{proof}

\begin{wrapfigure}{r}{0.42\textwidth}
\begin{tikzpicture}[line cap=round,line join=round,>=triangle 45,x=1.0cm,y=1.0cm]
\clip(-2.9,-1.9) rectangle (3.2,2.2);
\draw[->,color=black] (-3.5,0) -- (3.2,0);
\foreach \x in {-3,-2,-1,1,2,3}
\draw[shift={(\x,0)},color=black] (0pt,2pt) -- (0pt,-2pt);
\draw[->,color=black] (0,-2.5) -- (0,2.2);
\foreach \y in {-2,-1,1,2}
\draw[shift={(0,\y)},color=black] (2pt,0pt) -- (-2pt,0pt);
\draw (0.5,0.9) node[anchor=north west] {$y = r_1 x$};
\draw (0.38,1.55) node[anchor=north west] {$ y = r_2 x $};
\draw [domain=-3.5:3.5] plot(\x,{(-0-2.23*\x)/-0.6});
\draw [domain=-3.5:3.5] plot(\x,{(-0--0.6*\x)/2.23});
\draw (0.67,2.1) node[anchor=north west] {$Q(x, y) > 0$};
\draw (-2.5,-0.72) node[anchor=north west] {$Q(x, y) > 0$};
\draw (-2.58,1.98) node[anchor=north west] {$Q(x, y) < 0$};
\draw (0.62,-1.35) node[anchor=north west] {$Q(x, y) < 0$};
\fill [color=black] (1,0) circle (1.5pt);
\draw[color=black] (1,-0.35) node {$S$};
\fill [color=black] (0,1) circle (1.5pt);
\draw[color=black] (-0.26,1) node {$T$};
\fill [color=black] (-1,0) circle (1.5pt);
\draw[color=black] (-0.99,0.29) node {$S[1]$};
\fill [color=black] (0,-1) circle (1.5pt);
\draw[color=black] (0.56,-1.07) node {$T[1]$};
\end{tikzpicture}
\caption{$\HH_\WW$, as oriented by $\sigma_+$}
\label{fig:HWplanenohyperbola}
\end{wrapfigure}

Whenever we are in case \eqref{case:2spherical}, we will will denote the two $\sigma_0$-stable
spherical objects by $S, T$.  We may assume that $S$ has smaller phase than $T$ with respect
to $\sigma_+$; conversely, $S$ has bigger phase than $T$ with respect to $\sigma_-$.
We will also write $\ss := \vv(S), \tt := \vv(T)$, and $m := (\ss, \tt) > 2$.
We identify $\R^2$ with $\HH_\WW \otimes \R$ by sending the standard basis
to $(\ss, \tt)$; under this identification, the ordering of phases in $\R^2$
will be consistent with the ordering induced by $\sigma_+$.
We denote by $Q(x, y) = -2x^2 + 2mxy - 2y^2$ the pull-back of the quadratic form
induced by the Mukai pairing on $\HH_\WW$. 
Let $r_1 < r_2$ be the two solutions of $-2r^2 + 2mr - 2 = 0$; they are both positive and irrational
(as $m^2 - 4$ cannot be a square for $m \ge 3$ integral). The positive cone $P_\HH$ is thus the cone
between the two lines $y = r_i x$, and the effective cone $C_\WW$ is the upper right quadrant
$x, y \ge 0$.

We will first prove that the condition for the existence of totally semistable walls given in
Theorem \ref{thm:walls} is necessary in the case of non-isotropic walls. We start with an easy
numerical observation:
\begin{Lem} \label{lem:numericaldimensions}
Given $l > 1$ positive classes $\aa_1, \dots, \aa_l \in P_\HH$ with $\aa_i^2 > 0$, set
$\aa = \aa_1 + \dots + \aa_l$. Then
\[
\sum_{i=1}^l \left(\aa_i^2 + 2\right) < \aa^2.
\]
\end{Lem}
\begin{proof}
Since the $\aa_i$ are integral classes, and $\HH_\WW$ is an even lattice, we have $\aa_i^2 \ge 2$.
If $\aa_i \neq \aa_j$, then $\aa_i, \aa_j$ span a lattice of signature $(1, -1)$, which gives
\[
(\aa_i, \aa_j) > \sqrt{\aa_i^2 \aa_j^2} \ge 2, \quad \text{and thus} \quad
\aa^2 > \sum_{i=1}^l \aa_i^2 + 2 l (l-1) \ge \sum_{i=1}^l \aa_i^2 + 2l. \]
\end{proof}

\begin{Lem} \label{lem:nottotallysemistable}
Assume that the potential wall $\WW$ associated to $\HH$ satisfies the following conditions:
\begin{enumerate}
\item \label{enum:noniso} The wall is non-isotropic.
\item \label{enum:nonegs} There does not exist an effective spherical class $\ss \in C_\WW$ with $(\ss, \vv) < 0$.
\end{enumerate}
Then $\WW$ cannot be a totally semistable wall.
\end{Lem}
In other words, there exists a $\sigma_0$-stable object of class $\vv$. Note that by Lemma
\ref{lem:nonprimitive}, this statement automatically holds in the case of non-primitive $\vv$ as well.

\subsubsection*{Proof}
We will consider two maps from the moduli space $M_{\sigma_+}(\vv)$. 
On the one hand, by Theorem \ref{thm:contraction}, the line bundle
$\ell_{\sigma_0}$ on $M_{\sigma_+}(\vv)$ induces a birational morphism
\[ \pi^+ \colon M_{\sigma_+}(\vv) \to \overline{M}. \]
The curves contracted by $\pi^+$ are exactly curves of S-equivalent objects.

For the second map, first assume for simplicity that $M_{\sigma_+}(\vv)$ is a fine moduli space,
and let $\EE$ be a universal family. 
Consider the relative HN filtration for $\EE$ with respect to $\sigma_-$
given by Theorem \ref{thm:HNfamily}. Let $\aa_1, \dots, \aa_m$ be the Mukai vectors of the semistable
HN filtration quotients of a generic fiber $\EE_m$ for $m \in M_{\sigma_+}(\vv)$; by assumption
\eqref{enum:noniso}, we have $\aa_i^2 \neq 0$. On the open subset
$U$ of the Theorem \ref{thm:HNfamily}, the filtration quotients
$\EE^i/\EE^{i-1}$ are flat families of $\sigma_-$-semistable objects of class $\aa_i$; thus
we get an induced rational map
\[
\mathrm{HN} \colon M_{\sigma_+}(\vv) \dashrightarrow M_{\sigma_-}(\aa_1) \times \dots \times
M_{\sigma_-}(\aa_m).
\]
Let $I \subset \{1, 2, \dots, m\}$ be the subset of indices $i$ with $\aa_i^2 > 0$, and let
$\aa = \sum_{i \in I} \aa_i$. 

Our first claim is $\aa^2 \le \vv^2$, with equality if and only if $\aa =\vv$, i.e., if there are no
classes with $\aa_i^2 < 0$: Let $\bb = \vv - \aa = \sum_{i \notin I} \aa_i$. 
If $\bb^2 \ge 0$, and so $\bb^2 \ge 2$, the claim follows trivially from
$(\aa, \bb) > 0$:
\begin{equation} \label{eq:a2lessthanv2-1}
\vv^2 = \aa^2 + 2(\aa, \bb) + \bb^2 \ge \aa^2 + 4.
\end{equation}
Otherwise, observe that by our assumption
$(\vv, \blank)$ is non-negative on all effective classes; in particular, $(\vv, \bb) \ge 0$.
Combined with $\bb^2 \le -2$ we obtain 
\begin{equation} \label{eq:a2lessthanv2-2}
 \aa^2 = \vv^2-2(\vv, \bb) + \bb^2 \le \vv^2 -2.
\end{equation}

Lemma \ref{lem:numericaldimensions} then implies
\begin{equation} \label{eq:dimcomparison}
\vv^2 + 2 \ge \aa^2 + 2 \ge \sum_{i \in I} \left(\aa_i^2 + 2\right),
\end{equation}
with equality if and only if $\abs{I} = 1$.
By Theorem \ref{thm:nonempty}, part \eqref{enum:dimandsquare},
this says that the target of the rational map $\mathrm{HN}$ has at most the dimension of the source:
\begin{equation}\label{eq:InequalityDimensionHNFactors}
\dim M_{\sigma_+}(\vv) \ge \sum_{i = 1}^m \dim M_{\sigma_-}(\aa_i).
\end{equation}

However, if $\mathrm{HN}(E_1) = \mathrm{HN}(E_2)$, then $E_1, E_2$ are S-equivalent: indeed, they
admit Jordan-H\"older filtrations that are refinements of their HN filtrations with
respect to $\sigma_-$, which have the same filtration quotients.

It follows that any curve contracted by $\mathrm{HN}$ is also
contracted by $\pi^+$; therefore
\[
\sum_{i = 1}^m \dim M_{\sigma_-}(\aa_i) \ge \dim \overline{M} = \dim M_{\sigma_+}(\vv)
\]
Hence we have equality in each step of the above inequalities, the relative HN
filtration is trivial, and the generic fiber $\EE_m$ is $\sigma_-$-stable. In other words, the
generic object of $M_{\sigma_+}(\vv)$ is also $\sigma_-$-stable, which proves the claim.

In case $M_{\sigma_+}(\vv)$ does not admit a universal family, we can construct $\mathrm{HN}$ by first
passing to an \'etale neighborhood $f \colon U \to M_{\sigma_+}(\vv)$ admitting a universal family; the induced
rational map from $U$ induced by the relative HN filtration will then factor via $f$.
\hfill$\Box$

We recall some theory of Pell's equation in the language of spherical reflections of the
hyperbolic lattice $\HH$:

\begin{PropDef} \label{prop:pellequation}
Let $G_\HH \subset \Aut \HH$ be the group generated by spherical reflections $\rho_{\ss}$ for
effective spherical classes $\ss \in C_\WW$. 
Given a positive
class $\vv \in P_\HH \cap \HH$, the $G_\HH$-orbit $G_\HH.\vv$ contains a unique class $\vv_0$ such that
$(\vv_0, \ss) \ge 0$ for all effective spherical classes $\ss \in C_\WW$. 

We call $\vv_0$ the \emph{minimal class} of the orbit $G_\HH.\vv$.
\end{PropDef}
Note that the notion of minimal class depends on the potential wall $\WW$, not just on the lattice
$\HH$.
\begin{proof}
Again, we only treat the case \eqref{case:2spherical} of Proposition \ref{prop:stablespherical}, the
other cases being trivial. It is sufficient to prove that $(\vv_0, \ss) \ge 0$ and
$(\vv_0, \tt) \ge 0$. Assume $(\vv, \ss) < 0$. Then 
$\rho_\ss(\vv) = \vv - \abs{(\vv, \ss)}\cdot \ss$ is still in the upper right quadrant, with
smaller $x$-coordinate than $\vv$, and with the same $y$-coordinate.
Similarly if $(\vv, \tt) < 0$. If we proceed inductively, the procedure has to terminate, thus
reaching $\vv_0$. 

The uniqueness follows from Proposition \ref{prop:orbitlist} below.
\end{proof}

Assume additionally that $\HH$ admits infinitely many spherical classes, so we are in case
\eqref{case:2spherical} of Proposition \ref{prop:stablespherical}. The hyperbola
$\vv^2 = -2$ intersects the upper right quadrant $x, y \ge 0$ in two branches, starting at $\ss$ and
$\tt$, respectively. Let $\ss_0 = \ss, \ss_{-1}, \ss_{-2}, \dots$ be the integral spherical classes on the
lower branch starting at $\ss$, and $\tt_1 = \tt, \tt_2, \tt_3, \dots$ be those on the upper branch
starting at $\tt$, see also Figure \ref{fig:orbitlist}. The $\ss_i$ can be defined recursively by 
$\ss_{-1} = \rho_\ss(\tt)$, and $\ss_{k-1} = \rho_{\ss_{k}}(\ss_{k+1})$ for $k \le -1$; similarly
for the $\tt_i$. 

\begin{Prop} \label{prop:orbitlist}
Given a minimal class $\vv_0$ of a $G_\HH$-orbit, define $\vv_i, i \in \Z$ via 
$\vv_{i} = \rho_{\tt_{i}}(\vv_{i-1})$ for $i > 0$, and $\vv_i = \rho_{\ss_{i+1}}(\vv_{i+1})$ for $i < 0$. 
Then the orbit $G.\vv_0$ is given by $\stv{\vv_i}{i \in \Z}$, where the latter are ordered according to
their slopes in $\R^2$.
\end{Prop}
Note that these classes may coincide pairwise, in case $\vv_0$ is orthogonal to $\ss$ or $\tt$.
\begin{proof}
The group $G_\HH$ is the free product $\Z_2 \star \Z_2$, generated by $\rho_\ss$ and $\rho_\tt$. 
It is straightforward to check that with $\vv_i$ defined as above, we have
\[
\vv_{-1} = \rho_\ss (\vv_0), \quad \vv_{-2} = \rho_\ss \rho_\tt (\vv_0), \quad
\vv_{-3} = \rho_\ss \rho_\tt \rho_\ss (\vv_0), \ldots,
\]
and similarly $\vv_1 = \rho_{\tt} (\vv_0)$ and so on. This list contains $g(\vv_0)$ for all $g \in \Z_2
\star \Z_2$. That the $\vv_i$ are ordered by slopes is best seen by drawing a picture; see also
Figure \ref{fig:orbitlist}.
\end{proof}

\begin{figure}
\definecolor{qqqqcc}{rgb}{0,0,0.8}
\definecolor{uququq}{rgb}{0.2509803922,0.2509803922,0.2509803922}
\begin{tikzpicture}[line cap=round,line join=round,>=triangle 45,x=1.0cm,y=1.0cm]
\draw[->,color=black] (-0.8,0) -- (8.5,0);
\foreach \x in {,1,2,3,4,5,6,7,8}
\draw[shift={(\x,0)},color=black] (0pt,2pt) -- (0pt,-2pt);
\draw[->,color=black] (0,-0.8) -- (0,4.5);
\foreach \y in {,1,2,3,4}
\draw[shift={(0,\y)},color=black] (2pt,0pt) -- (-2pt,0pt);
\clip(-0.8,-0.8) rectangle (8.5,4.5);

\draw [samples=50,domain=0:0.99,rotate around={135:(0,0)},xshift=0cm,yshift=0cm] plot
({0.5773502692*(1+\x^2)/(1-\x^2)},{1*2*\x/(1-\x^2)});
\draw [samples=50,domain=0:0.99,rotate around={135:(0,0)},xshift=0cm,yshift=0cm] plot
({0.5773502692*(1+\x^2)/(1-\x^2)},{-2*\x/(1-\x^2)});

\draw [samples=50,domain=0:0.99,rotate around={135:(0,0)},xshift=0cm,yshift=0cm] plot
({0.5773502692*(-1-\x^2)/(1-\x^2)},{2*\x/(1-\x^2)});
\draw [samples=50,domain=0:0.99,rotate around={135:(0,0)},xshift=0cm,yshift=0cm] plot
({0.5773502692*(-1-\x^2)/(1-\x^2)},{(-2)*\x/(1-\x^2)});


\draw [samples=50,domain=0:0.99,rotate around={-135:(0,0)},xshift=0cm,yshift=0cm] plot
({1.4142135624*(-1-\x^2)/(1-\x^2)},{0.8164965809*(-2)*\x/(1-\x^2)});
\draw [samples=50,domain=0:0.99,rotate around={-135:(0,0)},xshift=0cm,yshift=0cm] plot
({1.4142135624*(-1-\x^2)/(1-\x^2)},{0.8164965809*2*\x/(1-\x^2)});

\draw (5.5,1.3) node[anchor=north west] {$ Q(x,y) = -2 $};
\draw [dash pattern=on 2pt off 2pt,color=qqqqcc] (7.4604462269,2.0766418445)--
(4.3852498909,1.3078427605)-- (0.8461211513,1.3078427605)-- (0.8461211513,2.0766418445)--
(4.3852498909,16.2331568031);
\begin{scriptsize}
\fill [color=black] (1,0) circle (1.5pt);
\draw[color=black] (1.,-0.3516959583) node {$\mathbf s_0 = \mathbf s$};
\fill [color=black] (0,1) circle (1.5pt);
\draw[color=black] (0.46394913,0.855968031) node {$\mathbf t_1 = \mathbf t$};
\fill [color=uququq] (0.8461211513,2.0766418445) circle (1.5pt);
\draw[color=uququq] (1.6381894244,2.127953402) node {$\mathbf v_1 = \rho_{\mathbf t} (\mathbf v)$};
\fill [color=black] (0.8461211513,1.3078427605) circle (1.5pt);
\draw[color=black] (1.1013621151,1.4386989547) node {$\mathbf v_0$};
\fill [color=uququq] (4.3852498909,1.3078427605) circle (1.5pt);
\draw[color=uququq] (4.3852174029,1.6460115384) node {$\mathbf v_{-1} = \rho_{\mathbf s}(\mathbf
v)$};
\fill [color=uququq] (4,1) circle (1.5pt);
\draw[color=uququq] (4,0.7327083255) node {$\mathbf s_{-1}$};
\fill [color=uququq] (1,4) circle (1.5pt);
\draw[color=uququq] (0.7,3.9999740551) node {$\mathbf t_2$};
\fill [color=uququq] (7.4604462269,2.0766418445) circle (1.5pt);
\draw[color=uququq] (7.2,2.5114733858) node {$\mathbf v_{-2} = \rho_{\mathbf
s_{-1}}(\mathbf v_{-1})$};
\fill [color=uququq] (0,0) circle (1.5pt);
\end{scriptsize}
\end{tikzpicture}
\caption{The orbit of $\vv_0$}
\label{fig:orbitlist}
\end{figure}

For $i > 0$, let $T_i^+ \in \PP_0(1)$ be the unique $\sigma_+$-stable object with $\vv(T_i^+) =
\tt_i$; similarly for $S_i^+$ with $\vv(S_i^+) = \ss_i$ for $i \le 0$. We also write $T_i^-$ and
$S_i^-$ for the corresponding $\sigma_-$-stable objects. 

\begin{Prop} \label{prop:sphericaltotallysemistable}
Let $\WW$ be a potential wall, and assume there is an effective spherical class $\tilde \ss \in C_\WW$
with $(\vv, \tilde \ss) < 0$. Then $\WW$ is a totally semistable wall.

Additionally, let $\vv_0$ be the minimal class in the orbit $G_\HH.\vv$, and write
$\vv = \vv_l$ as in Proposition \ref{prop:orbitlist}. 
If $\phi^+(\vv) > \phi^+(\vv_0)$, then
\[
\ST_{T_l^+} \circ \ST_{T_{l-1}^+} \circ \dots \circ \ST_{T_1^+} (E_0)
\]
is $\sigma_+$-stable of class $\vv$,
for every $\sigma_0$-stable object $E_0$ of class $\vv_0$.

Similarly, if $\phi^+(\vv) < \phi^+(\vv_0)$, then 
\[
\ST_{S_{-l+1}^+}^{-1} \circ \ST_{S_{-l+2}^+}^{-1} \circ \dots \circ \ST_{S_0^+}^{-1} (E_0)
\]
is $\sigma_+$-stable of class $\vv$ for every $\sigma_0$-stable object of class $\vv_0$.

The analogous statement holds for $\sigma_-$.
\end{Prop}

Note that when we are in case \eqref{case:1spherical} of Proposition \ref{prop:stablespherical}, the
above sequence of stable spherical objects will consist of just one object.

Before the proof, we recall the following statement (see {\cite[Lemma 5.9]{localP2}}):
\begin{Lem}
\label{lem:stableextension}
Assume that $A, B$ are simple objects in an abelian category.
If $E$ is an extension of the form
\[
A \into E \onto B^{\oplus r}
\]
with $\Hom(B, E) = 0$, then any quotient of $E$ is of the form $B^{\oplus r'}$.
Similarly, given an extension
\[
A^{\oplus r} \into E \onto B
\]
with $\Hom(E, A) = 0$, then any subobject of $E$ is of the form $A^{\oplus r'}$. 
\end{Lem}

\begin{proof}
We consider the former case, i.e., an extension $A \into E \onto B^{\oplus r}$; the latter case
follows by dual arguments.
Let $E \onto N$ be any quotient of $E$.
Since $A$ is a simple object, the composition
$\psi \colon A \into E \onto N$ is either injective, or zero.

If $\psi = 0$, then $N$ is a quotient of $B^{\oplus r}$, and the claim
follows. 
If $\psi$ is injective, let $M$ be the kernel of $E \onto N$. Then $M \cap A = 0$, and so
$M$ is a subobject of $B^{\oplus r}$. Since $B$ is a simple object,
$M$ is of the form $B^{\oplus r'}$ for some $r' < r$; since $\Hom(B, E) = 0$,
this is a contradiction.
\end{proof}

\subsubsection*{Proof of Proposition \ref{prop:sphericaltotallysemistable}}
Continuing with the convention of Proposition \ref{prop:stablespherical}, we use the
$\widetilde \GL_2^+(\R)$-action to assume 
$Z_0(\HH) \subset \R$, and $Z_0(\vv) \in \R_{<0}$. 

Consider the first claim.
By assumption, we may find an effective spherical class $\tilde{\ss}$ such that $(\vv,\tilde{\ss})<0$.
Pick a $\sigma_0$-semistable object $S$ with $\vv(S)=\tilde{\ss}$.
By considering its Jordan-H\"older filtration, and using Lemma \ref{lem:JHspherical},
we may find a $\sigma_0$-\emph{stable} spherical object $\widetilde S$ with $(\vv, \vv(\widetilde S)) < 0$.
Assume, for a contradiction, that $\WW$ is not a totally semistable wall.
Then there exists a $\sigma_0$-\emph{stable} object $E$ of class $\vv$.
By stability, since $E$ and $\widetilde{S}$ have the same phase, we have $\Hom(\widetilde S, E) = \Hom(E, \widetilde S) = 0$; hence
$(\vv, \vv(\widetilde S)) = \ext^1(\widetilde S, E) \ge 0$, a contradiction.

To prove the construction of $\sigma_+$-stable objects, let us assume that we are in the case
of infinitely many spherical classes. Let us also assume that $\phi^+(\vv) > \phi^+(\vv_0)$, the other
case is analogous; in the notation of Proposition \ref{prop:orbitlist}, this means
 $\vv = \vv_l$ for some $l > 0$.
We define $E_i$ inductively by 
\[
E_i = \ST_{T_i^+}(E_{i-1}).
\]
By the compatibility of the spherical twist $\ST_{T}$, for $T$ a spherical object, with the reflection
$\rho_{\vv(T)}$ and Proposition \ref{prop:orbitlist}, we have $\vv(E_i) = \vv_i$.
Lemma \ref{lem:stableextension} shows that $E_1$ is $\sigma_+$-stable; however, for the
following induction steps, we cannot simply use Lemma \ref{lem:stableextension} again, as neither $E_i$
nor $T_i^+$ are simple objects in $\PP_0(1)$.

\begin{wrapfigure}{r}{0.39\textwidth}
\vspace{-1em}
\begin{tikzpicture}[line cap=round,line join=round,>=triangle 45,x=1.0cm,y=1.0cm]
\clip(-1.2,-2.8) rectangle (4.3,4.61);
\draw[->,color=black] (-3,0) -- (4.3,0);
\foreach \x in {-2,-1,1,2,3,4,5}
\draw[shift={(\x,0)},color=black] (0pt,2pt) -- (0pt,-2pt);
\draw[->,color=black] (0,-2.8) -- (0,4.6);
\foreach \y in {-2,-1,1,2,3,4}
\draw[shift={(0,\y)},color=black] (2pt,0pt) -- (-2pt,0pt);
\draw [samples=50,domain=-0.99:0.99,rotate around={135:(0,0)},xshift=0cm,yshift=0cm] plot ({0.58*(1+(\x)^2)/(1-(\x)^2)},{1*2*(\x)/(1-(\x)^2)});
\draw [samples=50,domain=-0.99:0.99,rotate around={135:(0,0)},xshift=0cm,yshift=0cm] plot ({0.58*(-1-(\x)^2)/(1-(\x)^2)},{1*(-2)*(\x)/(1-(\x)^2)});
\draw [samples=50,domain=-0.99:0.99,rotate around={-135:(0,0)},xshift=0cm,yshift=0cm] plot ({2*(-1-(\x)^2)/(1-(\x)^2)},{1.15*(-2)*(\x)/(1-(\x)^2)});
\draw [dash pattern=on 2pt off 2pt,domain=-5.6:11.61] plot(\x,{(-0--4*\x)/1});
\draw (0.01,3) node[anchor=north west] {$\mathcal{T}_1$};
\draw [shift={(0.27,2.23)},dotted]  plot[domain=-0.18:2.96,variable=\t]({1*0.28*cos(\t r)+0*0.28*sin(\t r)},{0*0.28*cos(\t r)+1*0.28*sin(\t r)});
\draw [dash pattern=on 2pt off 2pt,domain=-5.6:11.61] plot(\x,{(-0--4.61*\x)/1.34});
\draw [shift={(0.44,3.18)},dotted]  plot[domain=-0.33:2.82,variable=\t]({1*0.47*cos(\t r)+0*0.47*sin(\t r)},{0*0.47*cos(\t r)+1*0.47*sin(\t r)});
\draw (0.12,4.16) node[anchor=north west] {$\mathcal{T}_2$};
\draw [shift={(-0.01,-0.03)},dotted]  plot[domain=-1.55:1.325,variable=\t]({1*2.2*cos(\t r)+0*2.2*sin(\t r)},{0*2.2*cos(\t r)+1*2.2*sin(\t r)});
\draw (1.51,-0.3) node[anchor=north west] {$\mathcal{A}_1$};
\draw [shift={(0.01,0.02)},dotted]  plot[domain=-1.80:1.29,variable=\t]({1*2.63*cos(\t r)+0*2.63*sin(\t r)},{0*2.63*cos(\t r)+1*2.63*sin(\t r)});
\draw (2.28,-0.67) node[anchor=north west] {$\mathcal{A}_2$};
\draw [shift={(0.01,0.02)},dotted] plot[domain=0:1.57,variable=\t]({1.5*cos(\t r)},{1.5*sin(\t r)});
\draw (0.6,1.2) node[anchor=north west] {$\mathcal{A}_0$};
\begin{scriptsize}
\fill [color=black] (1,0) circle (1.5pt);
\draw[color=black] (1.1,-0.25) node {$S_1$};
\fill [color=black] (0,1) circle (1.5pt);
\draw[color=black] (-0.3,1.1) node {$T_1$};
\fill [color=black] (0,-1) circle (1.5pt);
\draw[color=black] (0.59,-1.07) node {$T_1[-1]$};
\fill [color=black] (4,1) circle (1.5pt);
\draw[color=black] (4.1,0.75) node {$S_0$};
\fill [color=black] (1,4) circle (1.5pt);
\draw[color=black] (0.8,4.2) node {$T_2$};
\end{scriptsize}
\end{tikzpicture}
\caption{The categories $\mathcal{A}_i$}
\vspace{-2em}
\label{fig:tilt}
\end{wrapfigure}

Instead, we will need a slightly stronger induction statement.
Using Proposition \ref{prop:HW}, in particular part \eqref{enum:HNfactorsinHW}, we can define
a torsion pair $(\TT_i, \FF_i)$ in $\AA_0 := \PP_0(1)$ as follows:
we let $\TT_i$ be the extension closure of all $\sigma_+$-stable objects
$F \in \AA_0$ with $\phi^+(F) > \phi^+(T_{i+1})$; by Theorem \ref{thm:nonempty},
since the Mukai vectors of stable objects have self-intersection $\geq-2$ and all objects $F$ as before have self-intersection $<0$, we deduce that
$\TT_i$ is the extension-closure $\TT_i = \langle T_1^+, \dots, T_i^+ \rangle$.
Then let $\AA_i = \langle \FF_i, \TT_i[-1] \rangle$ (see Figure \ref{fig:tilt}).
We can also describe $\AA_{i+1}$ inductively as the tilt of $\AA_i$ at the torsion pair $(\TT, \FF)$
with $\TT = \langle T_{i+1}^+ \rangle$ and $\FF = \langle T_{i+1}^+ \rangle^\perp$.
\begin{description*}
\item[Induction claim] 
We have $E_i \in \FF_i$, and both $E_i$ and $T_{i+1}^+$ are simple objects of $\AA_i$.
\end{description*}
By construction of the torsion pair $(\TT_i, \FF_i)$, this also shows that $E_i$ is 
$\sigma_+$-stable.
Indeed, the fact that $E_i$ is in $\FF_i$ shows that $Hom(F,E_i)=0$, for all $\sigma_+$-stable objects $F$ with $\phi^+(F) > \phi^+(T_{i+1})$.
Also, the fact that it is simple in $\AA_i$ shows that $Hom(F,E_i)=0$, also for all $\sigma_+$-stable objects $F\neq E_i$ with $\phi^+(E_i)\leq \phi^+(F)\leq \phi^+(T_{i+1})$.
By definition, this means that $E_i$ is $\sigma_+$-stable.

The case $i = 0$ follows by the assumption that $E_0$ is $\sigma_0$-\emph{stable}.
To prove the induction step, we first consider $T_{i+1}^+$.
By stability, we have $T_{i+1}^+ \in \TT_i^\perp = \FF_i$.
Using stability again, we also see that any non-trivial
quotient of $T_{i+1}^+$ is contained in $\TT_i$, so $T_{i+1}^+$ is a simple object of $\FF_i$. 
Since $T_{i+1}^+$ is stable of maximal slope in $\FF_i$, there also cannot be a short exact sequence
as in \eqref{eq:sesFFT} below. Therefore, Lemma \ref{lem:simpleintilt} shows that $T_{i+1}^+$ is a simple
object of $\AA_i$.

Since $E_i$ (by induction assumption) is also a simple object in
$\AA_i$, this shows $\Hom(E_i, T_{i+1}^+) = \Hom(T_{i+1}^+, E_i) = 0$. So
$\RHom(T_{i+1}^+, E_i) = \Ext^1(T_{i+1}^+, E_i)[-1]$, and
$E_{i+1}  = \ST_{T_{i+1}^+}(E_i)$ fits into a short exact sequence
\[
0 \to E_i \into E_{i+1} \onto T_{i+1}^+ \otimes \Ext^1(T_{i+1}^+, E_i) \to 0.
\]
In particular, $E_{i+1}$ is also an object of $\AA_i$. 
Note that
\[
\RHom(T_{i+1}^+, E_{i+1}) = \RHom(\ST_{T_{i+1}^+}^{-1}(T_{i+1}^+), \ST_{T_{i+1}^+}^{-1}(E_{i+1}))
= \RHom(T_{i+1}^+[1], E_i) 
\]
is concentrated in degree -2; this shows both that
$E_{i+1} \in (T_{i+1}^+)^\perp \subset \AA_i$, and that there are no extensions
$E_{i+1} \into F' \onto T_{i+1}^{\oplus k}$. Applying Lemma \ref{lem:simpleintilt} via the inductive
description of $\AA_{i+1}$ as a tilt of $\AA_i$, this proves the induction claim.
\hfill$\Box$

\begin{Lem} \label{lem:simpleintilt}
Let $(\TT, \FF)$ be a torsion pair in an abelian category $\AA$, and let $F \in \FF$ be an object
that is simple in the quasi-abelian category $\FF$, and that admits no non-trivial short exact sequences
\begin{equation} \label{eq:sesFFT}
0 \to F \into F' \onto T \to 0
\end{equation}
with $F' \in \FF$ and $T \in \TT$.
Then $F$ is a simple object in the tilted category $\AA^\sharp = \langle \FF, \TT[-1] \rangle$.
\end{Lem}
\begin{proof}
Consider a short exact sequence $A \into F \onto B$ in $\AA^{\sharp}$. The long exact cohomology
sequence with respect to $\AA$ is
\[
0 \to \HH^0_\AA(A) \into F \to F' \onto \HH^1_\AA(A) \to 0
\]
with $\HH^0_\AA(A) \in \FF, F' \in \FF$ and
$\HH^1_\AA(A) \in \TT$. Since $F$ is a simple object in $\FF$, we must have
$\HH^0_\AA(A) = 0$. Thus we get a short exact sequence as in \eqref{eq:sesFFT}, a contradiction.
\end{proof}

\section{Divisorial contractions in the non-isotropic case}
\label{sec:divisorialcontraction}

In this section we examine Theorem \ref{thm:walls} in the case of divisorial contractions when the lattice $\HH_\WW$ does not contain isotropic classes.
The goal is to prove the following proposition.

\begin{Prop} \label{prop:sphericaldivisorialwall}
Assume that the potential wall $\WW$ is non-isotropic. Then $\WW$ is a divisorial wall if and
only if there exists a spherical class $\tilde \ss \in \HH_\WW$ with $(\tilde \ss, \vv) = 0$.
If we choose $\tilde \ss$ to be effective, then the class of the contracted divisor $D$
is given by $D \equiv \theta(\tilde \ss)$.

If $\widetilde S$ is a stable spherical object of class $\vv(\widetilde S) = \tilde \ss$, then
$D$ can be described as a Brill-Noether divisor of $\widetilde S$: it is given 
either by the condition $\Hom(\widetilde S, \blank) \neq 0$, or by
$\Hom(\blank, \widetilde S) \neq 0$.
\end{Prop}
One can use more general results of Markman in \cite{Eyal:prime-exceptional} to show that a 
divisorial contraction implies the existence of an orthogonal spherical class in the non-isotropic
case. We will instead give a categorical proof in our situation.

We first treat the case in which there exists a $\sigma_0$-\emph{stable} object of class $\vv$:
\begin{Lem} \label{lem:notdivisorial}
Assume that $\HH$ is non-isotropic, and that $\WW$ is a potential wall associated to $\HH$. 
If $\vv$ is a minimal class of a $G_\HH$-orbit, and if there is no spherical class
$\tilde \ss \in \HH_\WW$ with $(\tilde \ss, \vv) = 0$, then
the set of $\sigma_0$-\emph{stable} objects in $M_{\sigma_+}(\vv)$ has complement of codimension
at least two.
\end{Lem}
In particular, $\WW$ cannot induce a divisorial contraction.

\begin{proof}
The argument is similar to Lemma \ref{lem:nottotallysemistable}; additionally, it uses
Namikawa's and Wierzba's characterization of divisorial contractions recalled in 
Theorem \ref{thm:NamikawaWierzba}.

For contradiction, assume that there is an irreducible divisor $D \subset M_{\sigma_+}(\vv)$ of objects
that are strictly semistable with respect to $\sigma_0$.
Let $\pi^+ \colon M_{\sigma_+}(\vv) \to \overline{M}$ be the morphism induced by
$\ell_{\sigma_0}$; it is either an isomorphism or a divisorial contraction.
The divisor $D$ may or may not be contracted by $\pi^+$; 
by Theorem \ref{thm:NamikawaWierzba}, we have 
$\dim \pi^+(D) \ge \dim D - 1 = \dim M_{\sigma_+}(\vv) - 2 = \vv^2$ in either case. 

On the other hand, consider the restriction of 
the universal family $\EE$ on $M_{\sigma_+}(\vv)$ to the divisor $D$, 
and its relative HN filtration with respect to $\sigma_-$. As before, this induces
a rational map
\[
\mathrm{HN}_D \colon D \dashrightarrow M_{\sigma_-}(\aa_1) \times \dots \times M_{\sigma_-}(\aa_l).
\]
Again, let $I \subset \{1, \dots, l\}$ be the subset of indices $i$ with $\aa_i^2 > 0$, and
$\aa = \sum_{i \in I} \aa_i$. The arguments leading to inequalities \eqref{eq:a2lessthanv2-1}
and \eqref{eq:a2lessthanv2-2} still apply, and show $\aa^2 \le \vv^2$.

If $I \neq \{1, \dots, l\}$, there exists a class $\aa_j$ appearing in the HN filtration
of the form
$\aa_j = m \tilde \ss$, $\tilde \ss^2 = -2$. Under the assumptions, we now have the \emph{strict}
inequality $(\tilde \ss, \vv) > 0$; thus, in equation \eqref{eq:a2lessthanv2-2}, we also have $(\vv, \bb) > 0$, and so
$\aa^2 \le \vv^2 - 4$ in all cases.

Otherwise, if $I = \{1, \dots, l\}$, we have $\abs{I} > 1$, and we can apply Lemma
\ref{lem:numericaldimensions}; in either case we obtain
\[
\sum_{i=1}^l \dim M_{\sigma_-}(\aa_i) = \sum_{i \in I} (\aa_i^2 + 2) < \vv^2 = \dim \pi^+(D).
\]
As before, this is a contradiction to the observation that any curve contracted by
$\mathrm{HN}_D$ is also contracted by $\pi^+$.
\end{proof}

The case of totally semistable walls can be reduced to the previous one:
\begin{Cor} 			\label{cor:notdivisorial}
Assume that $\HH$ is non-isotropic, and that there does not exist a spherical class
$\tilde \ss \in \HH$ with $(\tilde \ss, \vv) = 0$. Then a potential wall associated to $\HH$ cannot
induce a divisorial contraction.
\end{Cor}
In fact, we will later see that all potential walls associated to $\HH$ are mapped to the same wall in
the movable cone of the moduli space; thus they have to exhibit identical birational behavior.

\begin{proof}
As before, consider the minimal class $\vv_0$ of the orbit $G_\HH.\vv$, in the sense of
Definition \ref{prop:pellequation}.
By Lemma \ref{lem:notdivisorial}, there is an open subset
$U \subset M_{\sigma_+}(\vv_0)$ of objects that are $\sigma_0$-\emph{stable} that
has complement of codimension at least two.

Let $\Phi$ be the composition of spherical twists given
by Proposition \ref{prop:sphericaltotallysemistable}, such that
$\Phi(E_0)$ is $\sigma_+$-stable of class $\vv$ for every $[E_0] \in U$. Observe that 
$\Phi(E_0)$ has a Jordan-H\"older filtration such that $E_0$ is one of its filtration factors (the
other factors are stable spherical objects). Therefore, the induced map
$\Phi_* \colon U \to M_{\sigma_+}(\vv)$ is injective, and the image does not contain any curve
of S-equivalent objects with respect to $\sigma_0$. Also, $\Phi_*(U)$ has complement of codimension 
at least two (see e.g. \cite[Proposition 21.6]{GrossHuybrechtsJoyce}).
Since $\ell_{\sigma_0}$ does not contract any curves in $\Phi_*(U)$, it 
cannot contract any divisors in $M_{\sigma_+}(\vv)$.
\end{proof}

The next step is to construct the divisorial contraction when there exists an orthogonal spherical
class.  To clarify the logic, we first treat the simpler case of a wall that is not totally semistable:
\begin{Lem} \label{lem:minimaldivisorialcontraction}
Assume $\HH$ is non-isotropic, $\WW$ a potential wall associated to $\HH$, and that $\vv$ is a minimal
class of a $G_\HH$-orbit. If there exists a spherical class $\tilde \ss \in \HH$ with
$(\tilde \ss, \vv) = 0$, then $\WW$ induces a divisorial contraction. 

If we assume that $\tilde \ss$ is effective, then
the contracted divisor $D \subset M_{\sigma_+}(\vv)$ has class $\theta(\tilde \ss)$. 
The HN filtration of a  generic element $[E] \in D$ with respect to $\sigma_-$ is of the
form
\[
0 \to \widetilde S \into E \onto F \to 0  \quad \text{or} \quad
0 \to F \into E \onto \widetilde S \to 0,
\]
where $\widetilde S$ and $F$ are $\sigma_0$-stable objects of class $\tilde \ss$ and
$\vv - \tilde \ss$, respectively.
\end{Lem}

\begin{proof}
As before, we only treat the case when $\HH$ admits infinitely many spherical classes. 
In that case, we must have $\tilde \ss = \ss$ or $\tilde \ss = \tt$; we may assume $\tilde \ss = \ss$, and
the other case will follow by dual arguments.

We first prove that $\vv -\ss$ is a minimal class in its $G_\HH$-orbit by a straightforward
computation. If $\vv^2 = 2$, then $(\vv-\ss)^2 =0$ in contradiction to the assumption; therefore $\vv^2 \ge 4$.
If we write $\vv = x \ss + y \tt$, then $(\vv, \ss) = 0$ gives $ y = \frac 2m x$. Plugging in $\vv^2 \ge 4$ gives
$x^2 \left(1 - \frac{4}{m^2}\right) \ge 2$. Since $m \ge 3$, we obtain
\[
x^2\left(1 - \frac 4{m^2}\right)^2 >
x^2\left(1 - \frac 4{m^2}\right) \frac 12 \ge 1, 
\]
and therefore
\[
(\tt, \vv - \ss) = m (x-1) - 2 \frac 2m x
= m x \left(1 - \frac 4{m^2}\right) - m \ge 0.
\]
Also, $(\ss, \vv-\ss) = 2 > 0$, and therefore $\vv-\ss$ has positive pairing with every effective spherical
class.

By Lemma \ref{lem:nottotallysemistable}, the generic element $F \in M_{\sigma_+}(\vv-\ss)$ is also
${\sigma_0}$-\emph{stable}. Since $(\ss, \vv-\ss) = 2$ and $\Hom(F, S) = \Hom(S, F) = 0$, there is
a family of extensions
\[ 0 \to S \into E_p \onto F \to 0 \]
parameterized by $p\in \P^1 \cong \P(\Ext^1(F, S))$. By Lemma \ref{lem:stableextension}, they are
$\sigma_+$-stable. Since all $E_p$ are S-equivalent to each other, the morphism
$\pi^+ \colon M_{\sigma_+}(\vv) \to \overline{M}$ associated to $\WW$
contracts the image of this rational curve. Varying $F \in M_{\sigma_0}^{st}(\vv-\ss)$, these span a
family of dimension $1 + (\vv-\ss)^2 + 2 = \vv^2 + 1$; this is a divisor 
in $M_{\sigma_+}(\vv)$ contracted by $\pi^+$. 

Since $\pi^+$ has relative Picard-rank equal to one, it cannot contract any other component.
\end{proof}

The following lemma treats the general case, for which we will first set up notation. 
As before, we let $\vv_0$ be the minimal class in the $G_\HH$-orbit of $\vv$. By
$\tilde \ss_0$ we denote the effective spherical class with $(\vv_0, \tilde \ss_0) = 0$; 
we have $\tilde \ss_0 = \tt$ or $\tilde \ss_0 = \ss$.
Accordingly, in the list of the $G_\HH$-orbit of $\vv$ given by Proposition 
\ref{prop:orbitlist}, we have either $\vv_{2i} = \vv_{2i+1}$, or $\vv_{2i} = \vv_{2i-1}$ for all
$i$, since $\vv_0$ is fixed under the reflection $\rho_{\tilde \ss_0}$ at $\tilde \ss_0$. 
We choose $l$ such that $\vv = \vv_l$, and such that the corresponding sequence
of reflections sends $\tilde \ss_0$ to $\tilde \ss$:
\[
\tilde \ss = 
\begin{cases}
\rho_{\tt_l} \circ \rho_{\tt_{l-1}} \circ \dots \circ \rho_{\tt_0}(\tilde \ss_0)
& \text{if $l > 0$} \\
\rho_{\ss_l} \circ \rho_{\ss_{l-1}} \circ \dots \circ \rho_{\ss_{-1}}(\tilde \ss_0)
& \text{if $l < 0$}
\end{cases}
\]
Depending on the
ordering of the slopes $\phi^+(\vv), \phi^+(\vv_0)$, we let $\Phi$ be the composition of
spherical twists appearing in Proposition \ref{prop:sphericaltotallysemistable}.
\begin{Lem} \label{lem:sphericaldivisorialwall}
Assume that $\HH$ is non-isotropic, and let $\WW$ be a corresponding potential wall. If there
is an effective spherical $\tilde \ss \in C_\WW$ with $(\vv, \tilde \ss) = 0$, then $\WW$ induces
a divisorial contraction.  

The contracted divisor $D$ has class $\theta(\tilde \ss)$. For $E \in D$ generic, there are $\sigma_+$-stable objects
$F$ and $\widetilde S$ of class $\vv-\tilde \ss$ and $\tilde \ss$, respectively,
and a short exact sequence
\begin{equation} \label{eq:divisorialextension}
0 \to \widetilde S \into E \onto F \to 0  \quad \text{or} \quad
0 \to F \into E \onto \widetilde S \to 0.
\end{equation}
The inclusion $\widetilde S \into E$ or $F \into E$ appears as one of the filtration steps in a Jordan-H\"older
filtration of $E$.

In addition, there exists an open subset $U^+ \subset M_{\sigma_+}(\vv_0)$, with complement of codimension
two, such that $\Phi(E_0)$ is $\sigma_+$-stable for every $\sigma_+$-stable object
$E_0 \in U^+$.
\end{Lem}

\begin{proof}
We rely on the construction in the proof of Proposition \ref{prop:sphericaltotallysemistable}.
Let $\widetilde S_0$ be the stable spherical object of class $\tilde \ss_0$; we have
$\widetilde S_0 = S$ or $\widetilde S_0 = T$.
As in the proof of Lemma \ref{lem:minimaldivisorialcontraction}, 
one shows that $\vv_0 - \tilde \ss_0$ is the minimal class in its $G_\HH$-orbit.

Let $F_0$ be a generic $\sigma_0$-\emph{stable} object of class $\vv_0 - \tilde \ss_0$. 
Applying Proposition \ref{prop:sphericaltotallysemistable} to the class $\vv - \tilde \ss$, 
we see that $F := \Phi(F_0)$ is $\sigma_+$-stable of that class.

We may again assume that $\Phi$ is of the form $\ST_{T_l^+} \circ \dots \circ \ST_{T_1^+}$; the
other case follows by dual arguments.
Inductively, one shows that
$\Phi(S) = T_{l+1}^+$ and
$\Phi(T) = T_l^+[-1]$. These are both simple objects of the category $\AA_l$ defined by tilting
in the proof of Proposition \ref{prop:sphericaltotallysemistable}; therefore,
$\widetilde S := \Phi(\widetilde S_0)$ is simple in $\AA_l$. By the induction claim in the proof of Proposition \ref{prop:sphericaltotallysemistable},
$F = \Phi(F_0)$ is also a simple object in this category. In particular,
$\Hom(\widetilde S, F) = \Hom(F, \widetilde S) = 0$ and $\ext^1(\widetilde S, F) = 2$.
Applying Lemma \ref{lem:stableextension}
again, and using the compatibility of $\AA_l$ with stability, we obtain a stable extension
of the form \eqref{eq:divisorialextension}.

This gives a divisor contracted by $\pi^+$, and we can proceed as in the previous lemma.

Let $D_0 \subset M_{\sigma_+}(\vv_0)$ be the contracted divisor for the class $\vv_0$. The above proof
also shows that for a generic object $E_0 \in D_0$ (whose form is given by Lemma
\ref{lem:minimaldivisorialcontraction}), the object $\Phi(E_0)$ is a $\sigma_+$-stable (contained
in the contracted divisor $D$). Thus we can take $U^+$ to be the union of the set of
$\sigma_0$-\emph{stable} objects in $M_{\sigma_+}(\vv_0)$ with the open subset of 
$D_0$ of objects of the form given in Lemma \ref{lem:minimaldivisorialcontraction}. 
\end{proof}

\begin{proof}[Proof of Proposition \ref{prop:sphericaldivisorialwall}]
The statements follow from Corollary \ref{cor:notdivisorial} and Lemma 
\ref{lem:sphericaldivisorialwall}.
\end{proof}

\section{Isotropic walls are Uhlenbeck walls}\label{sec:iso}

In this section, we study potential walls $\WW$ in the case where $\HH$ admits an isotropic class
$\ww \in \HH, \ww^2 = 0$.
Following an idea of Minamide, Yanagida, and Yoshioka \cite{MYY2}, we study the wall $\WW$ via a
Fourier-Mukai transform after which $\ww$ becomes the class of a point.  Then $\sigma_+$ corresponds
to Gieseker stability and, as proven in \cite{Jason:Uhlenbeck}, the wall corresponds to the
contraction to the Uhlenbeck compactification, as constructed by Jun Li in \cite{JunLi:Uhlenbeck}.

Parts of this section are well-known.
In particular, \cite[Proposition 0.5]{Yoshioka:Irreducibility} deals with the existence of stable locally-free sheaves.
For other general results, see \cite{Yoshioka:Abelian}.

\subsection*{The Uhlenbeck compactification}

We let $(X,\alpha)$ be a twisted K3 surface.
For divisor classes $\beta,\omega \in \mathrm{NS}(X)_{\Q}$, with $\omega$ ample, and for a vector $\vv\in H^*_\alg(X,\alpha,\Z)$, we denote by $M_{\omega}^{\beta}(\vv)$ the moduli space of $(\beta,\omega)$-Gieseker semistable $\alpha$-twisted sheaves on $X$ with Mukai vector $\vv$. Here, $(\beta,\omega)$-Gieseker stability is defined via the Hilbert polynomial formally twisted by $e^{-\beta}$ (see \cite{MatsukiWenthworth:TwistedVariation,Yoshioka:TwistedStability, Lieblich:Twisted}). When $\beta=0$, we obtain the usual notion of $\omega$-Gieseker stability. In such a case, we will omit $\beta$ from the notation.

We start with the following observation:

\begin{Lem}\label{lem:StabilityIsotropic}
Assume that there exists an isotropic class in $\HH$.
Then there are two effective, primitive, isotropic classes $\ww_0$ and $\ww_1$ in $\HH$,
such that, for a generic stability condition $\sigma_0 \in \WW$, we have
\begin{enumerate}
\item \label{enum:sigma0} $M_{\sigma_0}(\ww_0)=M_{\sigma_0}^{st}(\ww_0)$, and
\item either $M_{\sigma_0}(\ww_1)=M_{\sigma_0}^{st}(\ww_1)$, or
there exists a $\sigma_0$-stable spherical object $S$, with Mukai vector $\ss$, such that $(\ss, \ww_1)<0$ and $\WW$ is a totally semistable wall for $\ww_1$.
\end{enumerate}
Any positive class $\vv' \in P_\HH$ satisfies $(\vv', \ww_i) \ge 0$ for $i = 1, 2$.
\end{Lem}

\begin{proof}
Let $\tilde \ww \in \HH$ be primitive isotropic class; up to replacing $\tilde \ww$ by $-\tilde \ww$, we
may assume it to be effective.
We complete $\tilde \ww$ to a basis $\{\tilde \vv,\tilde \ww\}$ of $\HH_\Q$.
Then, for all $(a,b)\in\Q$, we have
\[
(a\tilde \vv + b \tilde \ww)^2 = a \cdot \left( a {\tilde \vv}^2 + 2 b (\tilde \vv,\tilde \ww) \right).
\]
This shows the existence of a second integral isotropic class. If we choose it to be effective,
then the positive cone $P_\HH$ is given by
$\R_{\ge 0}\cdot \ww_0 + \R_{\ge 0}\cdot \ww_1$. The claim
$(\vv', \ww_i) \ge 0$ follows easily.

By Theorem \ref{thm:nonempty}, we have $M_{\sigma_0}(\tilde{\ww})\neq\emptyset$. If $\WW$ does not coincide
with a wall for $\tilde \ww$, then we can take $\ww_0 = \tilde \ww$, and claim
\eqref{enum:sigma0} will be satisfied.

Otherwise, let $\sigma\in\Stab^\dagger(X,\alpha)$ be a generic stability condition nearby $\WW$;
by \cite[Lemma 7.2]{BM:projectivity}, we have
$M_{\sigma}(\tilde \ww)=M^{st}_{\sigma}(\tilde \ww)\neq\emptyset$.

Up to applying a Fourier-Mukai equivalence, we may assume that $\tilde \ww = (0, 0, 1)$ is the 
Mukai vector of a point on a twisted K3 surface; then we can apply
the classification of walls for isotropic classes in \cite[Theorem 12.1]{Bridgeland:K3}, extended to
twisted surfaces in \cite{HMS:generic_K3s}.
If $\WW$ is a totally semistable wall for $\tilde \ww$, then we are in case $(A^+)$ or $(A^-)$ of \cite[Theorem 12.1]{Bridgeland:K3}:
there exists a spherical $\sigma_0$-stable twisted vector bundle $S$ such that $S$ or $S[2]$
is a JH factor of the skyscraper sheaf $k(x)$, for every $x \in M_{\sigma}(\tilde \ww)$;
moreover, the other non-isomorphic JH factor is either $\ST_S(k(x))$, or $\ST^{-1}_S(k(x))$.
In both cases, the Mukai vector $\ww_0$ of the latter JH factor is primitive and isotropic, and $\WW$ is not a wall for $\ww_0$.

Finally, if $\WW$ is a wall for $\tilde \ww$, but not a totally semistable
wall, it must be a wall of type $(C_k)$, still in the notation of \cite[Theorem
12.1]{Bridgeland:K3}: there is a rational curve $C \subset  M_{\sigma}(\tilde \ww)$ such that
$k(x)$ is strictly semistable iff $x \in C$.
But then the rank two lattice associated to the wall is negative semi-definite by \cite[Remark 6.3]{BM:projectivity};
on the other hand, by Proposition \ref{prop:HW}, claim \eqref{enum:JHfactorsinHW},
it must coincide with $\HH$, which has signature $(1, -1)$. This is a contradiction.
\end{proof}

Let $\ww_0,\ww_1 \in C_\WW$ be the effective, primitive, isotropic classes given by the above lemma,
and let $Y:=M_{\sigma_0}(\ww_0)$.
Then $Y$ is a K3 surface and, by \cite{Mukai:BundlesK3,Caldararu:NonFineModuliSpaces,Yoshioka:TwistedStability,HuybrechtsStellari:CaldararuConj}, there exist a class $\alpha'\in\mathrm{Br}(Y)$ and a Fourier-Mukai transform
\[
\Phi \colon \Db(X,\alpha)\xrightarrow{\sim} \Db(Y,\alpha')
\]
such that $\Phi(\ww_0)=(0,0,1)$. By construction, skyscraper sheaves of points on $Y$
are $\Phi_*(\sigma_0)$-stable. By Bridgeland's Theorem \ref{thm:BridgelandK3geometric}, there exist divisor classes
$\omega, \beta \in \NS(Y)_\Q$, with $\omega$ ample, such that up to the $\widetilde{\GL}^+_2(\R)$-action,
$\Phi_*(\sigma_0)$ is given by $\sigma_{\omega, \beta}$.
In particular, the category $\PP_{\omega, \beta}(1)$ is the extension-closure of 
skyscraper sheaves of points, and the shifts $F[1]$ of $\mu_{\omega}$-stable torsion-free sheaves $F$ 
with slope $\mu_{\omega}(F) = \omega\cdot\beta$.
Since $\sigma_0$ by assumption does not lie on any other wall with
respect to $\vv$, the divisor $\omega$ is generic with respect to $\Phi_*(\vv)$.

By abuse of notation, we will from now on write $(X,\alpha)$ instead of $(Y,\alpha')$, $\vv$ instead of $\Phi_*(\vv)$, and 
$\sigma_0$ instead of $\sigma_{\omega, \beta}$. 
Let $\sigma_+ = \sigma_{\omega, \beta - \epsilon}$ and
$\sigma_- = \sigma_{\omega, \beta + \epsilon}$; here $\epsilon$ is a sufficiently small positive
multiple of $\omega$.

\begin{Prop}[\cite{Jason:Uhlenbeck, LoQin:miniwalls}]\label{prop:Jason}
An object of class $\vv$ is $\sigma_+$-stable if and only if it is the shift $F[1]$ of a 
$(\beta,\omega)$-Gieseker stable sheaf $F$ on $(X, \alpha)$; the shift $[1]$ induces
the following identification of moduli spaces:
\[
M_{\sigma_+}(\vv) = M_{\omega}^{\beta}(-\vv).
\]
Moreover, the contraction morphism $\pi^+$ induced via Theorem \ref{thm:contraction} for generic
$\sigma_0 \in \WW$ is the Li-Gieseker-Uhlenbeck morphism to the Uhlenbeck compactification.

Finally, an object $F$ of class $\vv$ is $\sigma_-$-stable if and only if
it is the shift $F^\vee[2]$ of the derived dual of a $(-\beta,\omega)$-Gieseker stable sheaf on
$(X, \alpha^{-1})$.
\end{Prop}

\begin{proof}
The identification of $M_{\sigma_+}(\vv)$ with the Gieseker moduli space is well-known,
and follows with the same arguments as in \cite[Proposition 14.2]{Bridgeland:K3}.
For $\sigma_0$, two torsion-free sheaves $E, F$ become S-equivalent if and only if they have the same image in the
Uhlenbeck space (\cite[Theorem 3.1]{Jason:Uhlenbeck}, \cite[Section 5]{LoQin:miniwalls}): indeed,
if $E_i$ are the Jordan-H\"older factors of $E$ with respect to slope-stability, then
$E$ is S-equivalent to 
\[ \bigoplus E_i^{**} \oplus \left(E_i^{**}/E_i\right), \]
precisely as in \cite[Theorem 8.2.11]{HL:Moduli}.
Thus, Theorem \ref{thm:contraction} identifies $\pi^+$ with the morphism to the Uhlenbeck space.

The claim of $\sigma_-$-stability follows by
Proposition \ref{prop:dualstability} from the case of $\sigma_+$-stability; see also
see \cite[Proposition 2.2.7]{Minamide-Yanagida-Yoshioka:wall-crossing} in the case $\alpha = 1$.
\end{proof}

In other words, the coarse moduli space $M_{\sigma_0}(\vv)$ is isomorphic to the Uhlenbeck
compactification (\cite{JunLi:Uhlenbeck, Yoshioka:TwistedStability}) of the moduli space of
slope-stable vector bundles on $(X,\alpha)$.
Given a $(\beta,\omega)$-Gieseker stable sheaf $F \in M_{\omega}^{\beta}(-\vv)$, the
$\sigma_+$-stable object $F[1]$ becomes strictly semistable with respect to
$\sigma_0$ if and only if $F$ is not locally free, or if $F$ is not slope-\emph{stable}.

In particular, when the rank of $-\vv$ equals one, then the contraction morphism $\pi^+$
is the Hilbert-Chow morphism $\mathrm{Hilb}^n(X) \to \Sym^n(X)$; see also \cite[Example
10.1]{BM:projectivity}.

\subsection*{Totally semistable isotropic walls}

We start with the existence of a unique spherical stable object in the case the wall is totally semistable:

\begin{Lem}\label{lem:UniqueSpherical}
Assume that $\WW$ is a totally semistable wall for $\vv$.
\begin{enumerate}
\item  \label{enum:TotSStIsotropic1} There exists a unique spherical $\sigma_0$-stable object $S\in\PP_{\sigma_0}(1)$.
\item  \label{enum:TotSStIsotropic2} Let $E\in M_{\sigma_+}(\vv)$ be a generic object.
Then its HN filtration with respect to $\sigma_{-}$ has length $2$ and takes the form
\begin{equation}\label{eq:TotSStIsotropic}
S^{\oplus a} \to E \to F, \quad \text{ or } \quad  F \to E \to S^{\oplus a},
\end{equation}
with $a\in\Z_{>0}$.
The $\sigma_-$-semistable object $F$ is generic in $M_{\sigma_-}(\vv')$, for $\vv':=\vv(F)$, and $\dim M_{\sigma_-}(\vv') = \dim M_{\sigma_+}(\vv) = \vv^2 +2$.
\end{enumerate}
\end{Lem}

The idea of the proof is very similar to the one in Lemma \ref{lem:nottotallysemistable}.
The only difference is that we cannot use a completely numerical criterion like Lemma \ref{lem:numericaldimensions} and we will replace it by Mukai's Lemma \ref{lem:Mukai}.

\begin{proof}[Proof of Lemma \ref{lem:UniqueSpherical}]
We first prove \eqref{enum:TotSStIsotropic1}.
We consider again the two maps
\begin{align*}
&\pi^+\colon M_{\sigma_+}(\vv) \to \overline{M},\\
&\mathrm{HN}\colon M_{\sigma_+}(\vv) \dashrightarrow M_{\sigma_-}(\aa_1)\times \dots \times M_{\sigma_-}(\aa_m).
\end{align*}
The first one is induced by $\ell_{\sigma_0}$ and the second by the existence of relative HN filtrations.
By \cite[Section 4.5]{HL:Moduli}, we have, for all $i=1,\dots,m$ and for all $A_i\in M_{\sigma_-}(\aa_i)$,
\[
\dim M_{\sigma_-}(\aa_i) \leq \ext^1(A_i,A_i).
\]
Hence, by Mukai's Lemma \ref{lem:Mukai}, we deduce
\begin{equation}\label{eq:InequalityDimensionHNFactorsIsotropic}
\dim M_{\sigma_+}(\vv) \geq \sum_{i=1}^m \dim M_{\sigma_-}(\aa_i).
\end{equation}
Equation \eqref{eq:InequalityDimensionHNFactorsIsotropic} is the analogue of \eqref{eq:InequalityDimensionHNFactors} in the non-isotropic case.
Since any curve contracted by $\mathrm{HN}$ is also contracted by $\pi^+$, it follows that
\[
\sum_{i = 1}^m \dim M_{\sigma_-}(\aa_i) \ge \dim \overline{M} = \dim M_{\sigma_+}(\vv).
\]
Therefore equality holds, and $\mathrm{HN}$ is a dominant map.

This shows that the projections
\[
M_{\sigma_+}(\vv) \dashrightarrow M_{\sigma_-}(\aa_i)
\]
are dominant.
By Theorem \ref{thm:KLS}, $M_{\sigma_-}(\aa_i)$ has symplectic singularities.
Hence, we deduce that either $M_{\sigma_-}(\aa_i)$ is a point, or $\dim M_{\sigma_-}(\aa_i) = \dim M_{\sigma_+}(\vv) = \vv^2+2$.
Since $m\geq 2$, by Lemma \ref{lem:JHspherical} this shows the existence of a spherical $\sigma_0$-stable object in $\PP_{\sigma_0}(1)$.
By Proposition \ref{prop:stablespherical}, there can only be one such spherical object.

To prove \eqref{enum:TotSStIsotropic2}, we first observe that by uniqueness (and by Lemma \ref{lem:JHspherical} again), all $\sigma_-$-spherical objects appearing in a HN filtration of a generic element $E \in M_{\sigma_+}(\vv)$ must be $\sigma_0$-stable as well.
As a consequence, the length of a HN filtration of $E$ with respect to $\sigma_-$ must be $2$ and have the form \eqref{eq:TotSStIsotropic}.
Since the maps $M_{\sigma_+}(\vv) \dashrightarrow M_{\sigma_-}(\aa_i)$ are dominant, the
$\sigma_-$-semistable object $F$ is generic.
\end{proof}

We can now prove the first implication for the characterization of totally semistable walls in the isotropic case.
We let $\ss:=\vv(S)$, where $S$ is the unique $\sigma_0$-stable object in $\PP_{\sigma_0}(1)$.

\begin{Prop}\label{prop:TotSStWallIsotropicHasNumericalProps}
Let $\WW$ be a totally semistable wall for $\vv$. Then either
there exist an isotropic vector $\ww$ with $(\ww,\vv)=1$, or the effective
spherical class $\ss$ satisfies $(\ss,\vv)<0$.
\end{Prop}

\begin{proof}
We continue to use the notation of Lemma \ref{lem:UniqueSpherical}; in particular,
let $a > 0$ be as in the short exact sequence \eqref{eq:TotSStIsotropic}, and
$\vv' = \vv - a \ss$.

If $(\vv')^2>0$, then by Lemma \ref{lem:UniqueSpherical} and Theorem \ref{thm:nonempty}\eqref{enum:dimandsquare}, we have $(\vv')^2 = \vv^2$.
Since $\vv' = \vv - a \ss, a>0$, this implies $(\ss,\vv)<0$.

So we may assume $\vv'^2 = 0$. Then
$\vv^2 = 0 + 2a (\vv', \ss) - 2a^2$, and it follows that $(\vv', \ss) > 0$.
In the notation of Lemma \ref{lem:StabilityIsotropic}, this means that $\vv'$ is a positive multiple
of $\ww_0$, which we can take to be the class of a point: 
$\vv' = c \ww_0 = c (0, 0, 1)$.

Then the coarse moduli space $M_{\sigma_0}(\vv')$ is the symmetric product
$\Sym^c X$; if we define $n$ by $\vv^2 = 2n-2$, then 
the equality of dimensions in Lemma \ref{lem:UniqueSpherical} becomes $c = n$.
Therefore
\[
2n-2 = \vv^2 = (a \ss + n \ww_0)^2 = -2 a^2 + 2 a n (\ss, \ww_0)
\]
or, equivalently,
\begin{equation}\label{eq:Columbus0902II}
n-1 = a \bigl(n (\ss, \ww_0) - a\bigr)
\end{equation}
Recall that $(\ss, \ww_0) > 0$. If the right-hand side is positive, then it is at least
$n(\ss, \ww_0) - 1$. Thus, \eqref{eq:Columbus0902II} only has solutions if $(\ss, \ww_0) = 1$, in
which case they are $a = 1$ and $a = n-1$.
In the former case, $(\vv,\ww_0)=1$. In the latter case, observe that $\ww_1 = \ww_0 + \ss$,
and $(\vv,\ww_1)=1$ follows directly.
\end{proof}

The converse statement follows from Proposition \ref{prop:sphericaltotallysemistable} above, and Lemma \ref{lem:NumericalPropertiesImplyTotSStIsotropicII} below.

\begin{Lem}\label{lem:NumericalPropertiesImplyTotSStIsotropicII}
Let $\WW$ be a potential wall.
If there exists an isotropic class $\ww \in \HH_\WW$ with $(\ww,\vv)=1$, then $\WW$ is a totally semistable wall.
\end{Lem}

\begin{proof}
Note that by Lemma \ref{lem:StabilityIsotropic}, the primitive class $\ww$ is automatically
effective.
Let $\sigma_0\in\WW$ be a generic stability condition.
If $M_{\sigma_0}^{st}(\ww)\neq\emptyset$, then we can assume $\ww=(0,0,1)$.
In this case $-\vv$ has rank one, $M_{\sigma_+}(\vv)$ is the Hilbert scheme,
and $\WW$ is the Hilbert-Chow wall discussed in
\cite[Example 10.1]{BM:projectivity}; in particular, it is totally 
semistable.

Otherwise, $M_{\sigma_0}^{st}(\ww)=\emptyset$; hence, in the notation
of Lemma \ref{lem:StabilityIsotropic}, we are in the case $\ww = \ww_1$, and there exists a
$\sigma_0$-stable spherical object $S$, with Mukai vector $\ss$, such that $(\ss, \ww_1)<0$.

Write $\ww_1 = \ww_0 + r \ss$, where $r = (\ss, \ww_0) \in\Z_{>0}$.
Then
\[
1 = (\vv,\ww_1) = (\vv,\ww_0) + r (\vv,\ss).
\]
By Lemma \ref{lem:StabilityIsotropic}, $(\vv,\ww_0)$ is strictly positive, and
so $(\vv,\ss)\leq 0$. If the inequality is strict, Proposition \ref{prop:sphericaltotallysemistable}
applies.  Otherwise, $(\vv, \ss) = 0$ and $(\vv,\ww_1)=(\vv,\ww_0)=1$; thus we are again
in the case of the Hilbert-Chow wall, and $\WW$ is a totally semistable wall for $\vv$.
\end{proof}

\subsection*{Divisorial contractions}

We now deal with divisorial contractions for isotropic walls.
The case of a flopping wall, a fake wall, and no wall will be examined in Section \ref{sec:flopping}.

\begin{Prop}\label{prop:DivisorialWallIsotropicHasNumericalProps}
Let $\WW$ be a wall inducing a divisorial contraction.
Assume that $(\vv,\ww)\neq1,2$, for all isotropic vectors $\ww\in \HH$.
Then there exists an effective spherical class $\ss\in \HH$ with $(\ss,\vv)=0$.
\end{Prop}

\begin{proof}
The proof is similar to the one of Lemma \ref{lem:notdivisorial}: in particular, we are going to use Theorem \ref{thm:NamikawaWierzba}.
Let $D\subset M_{\sigma_+}(\vv)$ be an irreducible divisor contracted by $\pi^+\colon M_{\sigma_+}(\vv)\to\overline{M}$.
We know that $\dim \pi^+(D) = \vv^2$.
Consider the rational map
\[
\mathrm{HN}_D\colon D\dashrightarrow M_{\sigma_-}(\aa_1)\times\dots\times M_{\sigma_-}(\aa_l)
\]
induced by the relative HN filtration with respect to $\sigma_-$.
We let $I \subset \{1, \dots, l\}$ be the subset of indices $i$ with $\aa_i^2 > 0$, and
$\aa = \sum_{i \in I} \aa_i$. We can assume $\abs{I} < l$, otherwise the proof is identical
to Lemma \ref{lem:notdivisorial}.

\medskip

\noindent
{\bf Step 1.} We show that there is an $i$ such that $\aa_i$ is a multiple of a 
spherical class $\ss$.

Assume otherwise.
Then we can write $\vv = n_0 \ww_0 + n_1 \ww_1 + \aa$.
By symmetry, we may assume $n_1 \ge n_0$; in particular $n_1 \neq 0$.
Also note that for $i = 0, 1$ we have $(\ww_0,\ww_1)\geq1$,
$(\vv,\ww_i)\geq3$ and $(\ww_i, \aa) \ge 1$ as long as $\aa \neq 0$.

In case $\abs{I} \ge 1$, i.e., $\aa \neq 0$, we obtain a contradiction by
\begin{align}
\vv^2 & = \bigl( (\aa + n_0 \ww_0) + n_1 \ww_1\bigr)^2
 = \aa^2 + 2 n_0 (\aa, \ww_0) + 2 n_1 (\vv, \ww_1) \nonumber\\
& \ge \aa^2 + 2 n_0 + 6 n_1 
> \aa^2 + 2 + 2 n_0 + 2 n_1 \label{dumb2} \\
& \ge \sum_{i \in I} (\aa_i^2 + 2) + 2 n_0 + 2 n_1 = \sum_{i = 1}^l \dim M_{\sigma_-}(\aa_i) \ge
\dim \pi^+(D) = \vv^2,
\label{dumb3}
\end{align}
where we used the numerical observations in \eqref{dumb2},
and Lemma \ref{lem:numericaldimensions} for the case $\abs{I} > 1$ in \eqref{dumb3}.

Otherwise, if $\abs{I} = 0$, then $\vv = n_0 \ww_0 + n_1 \ww_1$ with $n_0, n_1 > 0$ and
$n_i ( \ww_0, \ww_1) \ge 3$ by the assumption $(\vv, \ww_i) \ge 3$. We get
a contradiction from
\begin{align*}
\vv^2 = 2n_0 n_1 (\ww_0, \ww_1)
& = 2n_0 + 2n_1 + 2 (n_0 -1 ) (n_1-1) -2 + 2n_0n_1 \bigl((\ww_0, \ww_1) -1\bigr) \\
& > 2n_0 + 2n_1 = \sum_{i = 1}^l \dim M_{\sigma_-}(\aa_i) \ge \vv^2.
\end{align*}

\medskip

{\bf Step 2.} We show $(\ss,\vv)\leq0$.

Assume for a contradiction that $(\ss,\vv)>0$.
Using $\ww_1=\rho_{\ss}(\ww_0)$ we can write $\vv = a\ss + b\ww_0 + \aa$.
By Step 1, we have $a>0$.
In case $\aa \neq 0$, we use $(\aa, \ww_0) > 0$ to get
\begin{align*}
\aa^2 & = \bigl( \left(\vv -a\ss\right) - b\ww_0\bigr)^2 \\ 
 &= \vv^2 -2a(\vv,\ss) -2a^2 -2b (\aa + b\ww_0 ,\ww_0) \\
 &\leq \vv^2 -2a(\vv,\ss) -2a^2 -2b.
\end{align*}
This leads to a contradiction:
\[
\vv^2 >\vv^2 -2a(\vv,\ss) -2a^2 +2 \ge \aa^2 +2 + 2b 
\ge \sum_{i=1}^l \dim M_{\sigma_-}(\aa_i) \ge \vv^2.
\]
If $\aa = 0$, our assumptions give $a(\ss,\ww_0) = (\vv, \ww_0) > 2$ and
$-2a + b(\ss, \ww_0) = (\ss, \vv) > 0$. This leads to 
\[
\vv^2 = -2a^2 + 2ab (\ss, \ww_0) > 
 ab(\ss, \ww_0) > 2b = 
\sum_{i=1}^l \dim M_{\sigma_-}(\aa_i) \ge \vv^2.
\]

\medskip

\noindent
{\bf Step 3.} We show $(\ss,\vv)=0$.

Assume for a contradiction that $(\ss,\vv)<0$.
By Proposition \ref{prop:sphericaltotallysemistable}, $\WW$ is a totally semistable wall for $\vv$.
We consider $\vv'=\rho_{\ss}(\vv)$ as in Lemma \ref{lem:UniqueSpherical}.
The wall $\WW$ induces a divisorial contraction for $\vv$ if and only if it induces one for $\vv'$.
But, since $(\vv,\ww)\neq1,2$, for all $\ww$ isotropic, then $(\vv',\ww)\neq1,2$ as well.
Moreover, $(\ss,\vv')>0$.
This is a contradiction, by Step 2.
\end{proof}

The converse of Proposition \ref{prop:DivisorialWallIsotropicHasNumericalProps} is a consequence of the following three lemmas:

\begin{Lem}\label{lem:NumericalPropertiesImplyDivisorialWallI}
Assume that $(\vv,\ww_0)=2$.
Then $\WW$ induces a divisorial contraction.
\end{Lem}

\begin{proof}
It suffices to show that $M_{\omega}^{\beta}(-\vv)$ contains a divisor of non-locally free sheaves.
Since $(\vv,\ww_0)=2$, we can write $\vv = -(2,D,s)$, where $D$ an integral divisor which is either primitive or $D=0$.
Consider the vector $\vv'=-(2,D,s+1)$
with $(\vv')^2 = \vv^2 -4 \geq-2$. By Theorem \ref{thm:nonempty}, we get
$M_{\omega}^{\beta}(-\vv')\neq\emptyset$ .
Given a sheaf $F \in M_{\omega}^{\beta}(-\vv')$ and a point $x \in X$, the
surjections $F \onto k(x)$ induce a $\P^1$ of extensions
\[
k(x) \to E[1] \to F[1] \to k(x)[1]
\]
of objects in $M_{\sigma_+}(\vv)$ that are S-equivalent with respect to $\sigma_0$. Dimension counting  shows
that they sweep out a divisor.
\end{proof}

\begin{Lem}\label{lem:NumericalPropertiesImplyDivisorialWallII}
Assume that there exists an effective spherical class $\ss \in \HH$ such that $(\vv,\ss)=0$.
Then $\WW$ induces a divisorial contraction.
\end{Lem}

\begin{proof}
Let $S$ be the unique $\sigma_0$-stable spherical object with Mukai vector $\ss$.
Let $\aa=\vv-\ss$; then
\[
 \aa^2 = (\vv-\ss)^2 = \vv^2 -2 \quad \text{and} \quad
 (\aa,\ss) = -\ss^2 =2.
\]
If $\vv^2>2$, then $\aa^2>0$.
By Lemma \ref{lem:StabilityIsotropic}, $\ww_1 = b\ss +\ww_0$ with $b>0$; hence $(\ww_1,\aa)>(\ww_0,\aa)$.
If $(\ww_0,\aa)\geq2$, then Proposition \ref{prop:TotSStWallIsotropicHasNumericalProps} implies that $\WW$ is not a totally semistable wall for $\aa$, since $(\aa,\ss)=2$.
Hence, given $A\in M_{\sigma_0}(\aa)$, all the extensions
\[
S \to E \to A
\]
give a divisor $D\subset M_{\sigma_+}(\vv)$, which is a $\P^1$-fibration over $M_{\sigma_0}^{st}(\aa)$ and which gets contracted by crossing the wall $\WW$.

If $(\ww_0,\aa)=1$, then there is a spherical class of the form $\aa + k \ww_0$. By the uniqueness
up to sign, $\ss$ must be of this form; hence also $(\ww_0,\ss)=1$.
From this we get $(\ww_0,\vv)=2$, and so $\WW$ induces a divisorial contraction by Lemma \ref{lem:NumericalPropertiesImplyDivisorialWallI}.

Finally, assume that $\vv^2=2$. Then $\aa$ is an isotropic vector with
$(\aa,\vv) = (\aa,\ss) =2$.
But this implies that $(\ww_0,\vv)=1,2$.
Indeed, by Lemma \ref{lem:StabilityIsotropic}, the fact that $\aa$ is an effective class with $(\aa,\ss)>0$ implies that $\aa$ has to be a positive multiple of $\ww_0$.
The case $(\ww_0,\vv)=2$ is again Lemma \ref{lem:NumericalPropertiesImplyDivisorialWallI};
and if $(\ww_0, \vv)=1$, then $-\vv$ has rank $1$, and we are in the case of the Hilbert-Chow wall.
\end{proof}

\begin{Lem}\label{lem:NumericalPropertiesImplyDivisorialWallIII}
Let $\WW$ be a potential wall.
If there exists an isotropic class $\ww$ such that $(\vv,\ww)\in \{1,2\}$,
then $\WW$ induces a divisorial contraction.
\end{Lem}

\begin{proof}
By Lemma \ref{lem:StabilityIsotropic}, the class $\ww$ is automatically effective.
By Lemma \ref{lem:NumericalPropertiesImplyDivisorialWallI}, the only remaining case is
$\ww = \ww_1$, with $\ww_1=b\ss+\ww_0$, $b>0$, where $\ss$ is the class of the unique $\sigma_0$-stable spherical object.
By Lemma \ref{lem:NumericalPropertiesImplyDivisorialWallII}, we can assume that $(\ss,\vv)\neq0$.

If $(\ss,\vv)>0$, then 
\[
(\ww_1, \vv) = b (\ss,\vv) + (\ww_0,\vv) \in \{1, 2\}.
\]
Since $(\ww_0,\vv)>0$ and $b>0$, this is possible only if $(\ww_0,\vv)=1$, which corresponds to the Hilbert-Chow contraction.

Hence, we can assume $(\ss,\vv)<0$.
By Proposition \ref{prop:sphericaltotallysemistable}, $\WW$ is a totally semistable wall for $\vv$,
and $\WW$ induces a divisorial contraction with respect to $\vv$ if and only if it
induces one with respect to $\vv' = \rho_\ss(\vv)$.
But then $(\vv',\ww_0)=(\vv,\ww_1) \in \{1,2\}$.
Again, we can use Lemma \ref{lem:NumericalPropertiesImplyDivisorialWallI} to finish the proof.
\end{proof}

\section{Flopping walls}
\label{sec:flopping}

This section deals with the remaining case of a potential wall $\WW$: assuming that $\WW$ does not
correspond to a divisorial contraction, we describe in which cases it is a flopping wall, a fake
wall, or not a wall. This is the content of Propositions \ref{prop:flops} and \ref{prop:noflops}.

\begin{Prop} \label{prop:flops}
Assume that $\WW$ does not induce a divisorial contraction.  If either
\begin{enumerate}
\item \label{enum:sum2positive}
$\vv$ can be written as the sum 
$\vv = \aa_1 + \aa_2$ of two positive classes $\aa_1, \aa_2 \in P_\HH \cap \HH$, or 
\item \label{enum:sphericalflop}
there exists
a spherical class $\tilde \ss \in \WW$ with $0 < (\tilde \ss, \vv) \le \frac{\vv^2}2$,
\end{enumerate}
then $\WW$ induces a small contraction.
\end{Prop}

\begin{Lem} \label{lem:findparallelogram}
Let $M$ be a lattice of rank two, and $C \subset M \otimes \R^2$ be a convex cone not containing a
line. If a primitive lattice element $\vv \in M \cap C$ can be written as the sum $\vv = \aa + \bb$ of two classes in
$\aa, \bb \in M \cap
C$, then it can be written as a sum $\vv = \aa' + \bb'$ of two classes $\aa', \bb' \in M \cap C$ in such a way that 
the parallelogram with vertices $0, \aa', \vv, \bb'$ does not contain any other lattice point besides
its vertices.
\end{Lem}

\begin{proof}
If the parallelogram $0, \aa, \vv, \bb$ contains an additional lattice point $\aa'$, we may replace
$\aa$ by $\aa'$ and $\bb$ by $\vv-\aa'$.  This procedure terminates.
\end{proof}

\begin{Lem} \label{lem:parallelogram}
Let $\aa, \bb, \vv \in \HH \cap C_\WW$ be effective classes with $\vv = \aa + \bb$.
Assume that the following conditions are satisfied:
\begin{itemize}
\item The phases of $\aa, \bb$ satisfy $\phi^+(\aa) < \phi^+(\bb)$.
\item The objects $A, B$ are $\sigma_+$-stable with $\vv(A) = \aa, \vv(B) = \bb$.
\item The parallelogram in $\HH \otimes \R$ with vertices $0, \aa, \vv, \bb$ does not contain any other
lattice point.
\item The extension $A \into E \onto B$ satisfies $\Hom(B, E) = 0$.
\end{itemize}
Then $E$ is $\sigma_{+}$-stable.
\end{Lem}

\begin{proof}
Since $A$ and $B$ are $\sigma_+$-stable, they are $\sigma_0$-semistable.
Therefore, the extension $E$ is also $\sigma_0$-semistable.
Let $\aa_i$ be the Mukai vector of the $i$-th HN factor of $E$ with respect to $\sigma_+$. 
By Proposition \ref{prop:HW} part \eqref{enum:semistablehasfactorsinHW} and Remark \ref{rem:HW}, we
have $\aa_i \in \HH$. 
We have $E \in \PP_+([\phi^+(\aa), \phi^+(\bb)])$, and hence $\aa_i$ is contained in the
cone generated by $\aa, \bb$. Since the same holds for $\vv - \aa_i = \sum_{j \neq i} \aa_j$, 
$\aa_i$ is in fact contained in the parallelogram with vertices $0, \aa, \vv, \bb$. Since
it is also a lattice point, the assumption on the
parallelogram implies $\aa_i \in \{\aa, \bb, \vv\}$.

Assume that $E$ is not $\sigma_+$-stable, and let $A_1 \subset E$ be the first HN filtration factor. Since 
$\phi^+(\aa_1) > \phi^+(\vv)$, we must have $\aa_1 = \bb$. By the stability of $A, B$ we
have $\Hom(A_1, A) = 0$, and $\Hom(A_1, B) = 0$ unless $A_1 \cong B$. Either of these is
a contradiction, since $\Hom(A_1,E)\neq0$ and $\Hom(B,E)=0$.
\end{proof}


\subsubsection*{Proof of Proposition \ref{prop:flops}}
We first consider case \eqref{enum:sum2positive}, so $\vv = \aa_1 + \aa_2$ with
$\aa_1, \aa_2 \in P_\HH$. Using Lemma
\ref{lem:findparallelogram}, we may assume that the parallelogram with vertices
$0, \aa_1, \vv, \aa_2$ does not contain an interior lattice point. In particular,
$\aa_1, \aa_2$ are primitive. We may also assume that $\phi^+(\aa_1) < \phi^+(\aa_2)$. 
By the signature of $\HH$ (see the proof of Lemma \ref{lem:numericaldimensions}),
we have $(\aa_1, \aa_2) > 2$. 
By Theorem \ref{thm:nonempty}, there exist $\sigma_+$-stable objects $A_i$ of class
$\vv(A_i) = \aa_i$. The inequality for the Mukai pairing implies 
$\ext^1(A_2, A_1) > 2$.  By Lemma \ref{lem:parallelogram}, any extension 
\[ 0 \to A_1 \into E \onto A_2 \to 0 \]
of $A_2$ by $A_1$ is $\sigma_+$-stable of class $\vv$. As all these extensions are S-equivalent to
each other with respect to $\sigma_0$, we obtain a projective space of dimension at least two
that gets contracted by $\pi^+$.

Now consider case \eqref{enum:sphericalflop}.
First assume that $\tilde \ss$ is an effective class. Note that $(\vv - \tilde \ss)^2 \ge -2$
and $(\tilde \ss, \vv - \tilde \ss) = (\tilde \ss, \tilde \vv) - \tilde \ss^2 > 2$.
Consider the parallelogram $\mathbf{P}$ with vertices $0, \tilde \ss, \vv, \vv-\tilde \ss$,
and the function $f(\aa) = \aa^2$ for  $\aa \in \mathbf{P}$. By homogeneity, its minimum is obtained
on one of the boundary segments; thus
\[ \bigl( \tilde \ss + t (\vv - \tilde \ss) \bigr)^2 > -2 + 4t - 2t^2 > -2\]
for $0 < t < 1$, along with a similar computation for the other line segments, shows
$f(\aa) > -2$ unless $\aa \in \{\tilde \ss, \vv-\tilde \ss\}$. In particular, if there is
any lattice point  $\aa \in \mathbf{P}$ other than one of its vertices, then 
$\aa^2 \ge 0$ and $(\vv - \aa)^2 \ge 0$. Thus $\vv = \aa + (\vv-\aa)$ can be
written as the sum of positive classes, and the claim follows from the previous paragraph.
Otherwise, let $\widetilde S$ be the
$\sigma_+$-stable object of class $\tilde \ss$, and $F$ any $\sigma_+$-stable object of class
$\vv - \tilde \ss$; then $\ext^1(\widetilde S, F) = \ext^1(F, \widetilde S) > 2$. Thus, with the same
arguments we obtain a family of $\sigma_+$-stable objects parameterized by a projective space that
gets contracted by $\pi^+$. 

We are left with the case where $\tilde \ss$ is not effective. Set
$\tilde \tt = - \tilde \ss$, which is an effective class.
With the same reasoning as above, we may assume that the parallelogram
with vertices $0, \tilde \tt, \vv, \vv - \tilde \tt$ contains no additional lattice points. 
Set 
\[
\vv' = \rho_{\tilde \tt}(\vv) - \tilde \tt = \vv - \bigl((\tilde \ss,\vv) + 1\bigr)\tilde \tt,,
\]
and consider the parallelogram $\mathbf P$ with vertices $0$, $\bigl((\tilde \ss,\vv) + 1\bigr)
\tilde \tt$, $\vv$, $\vv'$ (see Figure \eqref{fig:snoteffective}).
We have $\vv'^2 \ge -2$ and $(\tilde \tt, \vv') = (\tilde \ss, \vv) + 2 > 2$. The lattice points
of $\mathbf P$
are given by $k \tilde \tt$ and $\vv' + k \tilde \tt$ for $k \in \Z$, $0 \le k \le (\tilde \ss, \vv) + 1$
(otherwise, already the parallelogram with vertices $0, \tilde \tt, \vv, \vv-\tilde \tt$ would contain
additional lattice points).

\begin{wrapfigure}{r}{0.26\textwidth}
\vspace{-1em}
\definecolor{zzttqq}{rgb}{0.6,0.2,0}
\definecolor{tttttt}{rgb}{0.2,0.2,0.2}
\definecolor{uququq}{rgb}{0.25,0.25,0.25}
\begin{tikzpicture}[line cap=round,line join=round,>=triangle 45,x=1.0cm,y=1.0cm]
\clip(-1,-1.2) rectangle (3,5.0);
\fill[line width=0.4pt,dash pattern=on 4pt off 4pt,color=zzttqq,fill=zzttqq,fill opacity=0.1] (0,0)
-- (0,4.2) -- (2.1,4.84) -- (2.1,0.64) -- cycle;
\fill[line width=0.4pt,dotted,pattern color=zzttqq,fill=zzttqq,pattern=north west lines] (0,0) --
(0,0.84) -- (2.1,4.84) -- (2.1,4) -- cycle;
\draw [->] (0,0) -- (0,0.84);
\draw [->] (0,0) -- (2.1,0.64);
\draw [->] (0,0) -- (0,-0.84);
\draw [line width=0.4pt,dash pattern=on 4pt off 4pt,color=zzttqq] (0,0)-- (0,4.2);
\draw [line width=0.4pt,dash pattern=on 4pt off 4pt,color=zzttqq] (0,4.2)-- (2.1,4.84);
\draw [line width=0.4pt,dash pattern=on 4pt off 4pt,color=zzttqq] (2.1,4.84)-- (2.1,0.64);
\draw [line width=0.4pt,dash pattern=on 4pt off 4pt,color=zzttqq] (2.1,0.64)-- (0,0);
\draw [line width=0.4pt,dotted,color=zzttqq] (0,0)-- (0,0.84);
\draw [line width=0.4pt,dotted,color=zzttqq] (0,0.84)-- (2.1,4.84);
\draw [line width=0.4pt,dotted,color=zzttqq] (2.1,4.84)-- (2.1,4);
\draw [line width=0.4pt,dotted,color=zzttqq] (2.1,4)-- (0,0);
\fill [color=uququq] (0,0) circle (1.5pt);
\draw[color=uququq] (-0.28,0.02) node {$0$};
\fill [color=tttttt] (0,0.84) circle (1.5pt);
\draw[color=tttttt] (-0.25,0.84) node {$\tilde \tt$};
\fill [color=black] (2.1,0.64) circle (1.5pt);
\draw[color=black] (2.5,0.64) node {$\vv'$};
\fill [color=uququq] (2.1,1.48) circle (1.5pt);
\draw[color=uququq] (2.6,1.5) node {$\rho_{\tilde \tt}(\vv)$};
\fill [color=uququq] (0,1.68) circle (1.5pt);
\fill [color=uququq] (0,2.52) circle (1.5pt);
\fill [color=uququq] (0,3.36) circle (1.5pt);
\fill [color=uququq] (0,4.2) circle (1.5pt);
\draw[color=uququq] (0.3,4.78) node {$((\tilde \ss, \vv)+1)\tilde \tt$};
\fill [color=uququq] (2.1,4.84) circle (1.5pt);
\draw[color=uququq] (2.3,4.84) node {$\vv$};
\fill [color=uququq] (2.1,4) circle (1.5pt);
\fill [color=uququq] (2.1,3.16) circle (1.5pt);
\fill [color=uququq] (2.1,2.32) circle (1.5pt);
\fill [color=uququq] (0,-0.84) circle (1.5pt);
\draw[color=uququq] (-0.25,-0.84) node {$\tilde \ss$};
\end{tikzpicture}
\vspace{-25pt}
\caption{-$\tilde \ss$ effective.} \label{fig:snoteffective}
\end{wrapfigure}

Let $\widetilde T$ and $F$ be $\sigma_+$-stable objects of class $\tilde \tt$ and $\vv'$, respectively.
Let us assume $\phi^+(\tilde \tt) > \phi^+(\vv)$; the other case follows by dual arguments.
Any subspace $\VVV \subset \Ext^1(\widetilde T, F)$ of dimension
$(\tilde \ss, \vv) + 1$ defines an extension
\[ 0 \to F \into E \onto \widetilde T \otimes \VVV \to 0 \phantom{\hspace{5cm}} \]
such that $E$ is of class $\vv(E) = \vv$, and satisfies $\Hom(\widetilde T, E) = 0$. If $E$ were
not $\sigma_+$-stable, then the class of the maximal destabilizing subobject $A$ would have to be
a lattice point in $\mathbf P$ with $\phi^+(\vv(A)) > \phi^+(\vv)$;
therefore, $\vv(A) = k \tilde \tt$. The only $\sigma_+$-semistable object of this class is
$\widetilde T^{\oplus k}$, and we get a contradiction. Thus, we have constructed a family of
$\sigma_+$-stable objects of class $\vv$ parameterized by the Grassmannian
$\mathrm{Gr}((\tilde \ss, \vv) + 1, \ext^1(\widetilde T, F))$ that become S-equivalent with respect to
$\sigma_0$. 
\hfill$\Box$


It remains to prove the converse of Proposition \ref{prop:flops}:

\begin{Prop} \label{prop:noflops}
Assume that $\WW$ does not induce a divisorial contraction. Assume that $\vv$ cannot be written
as the sum of two positive classes in $P_\HH$, and that there is no spherical class
$\ss \in \HH$ with $0 < (\ss, \vv) \le \frac{\vv^2}2$.
Then $\WW$ is either a fake wall, or not a wall.
\end{Prop}

\subsubsection*{Proof}
First consider the case where $\vv = \vv_0$ is the minimal class in its orbit $G_\HH.\vv$. We will prove that
every $\sigma_+$-stable object $E$ of class $\vv_0$ is also $\sigma_0$-\emph{stable}. 
Assume otherwise, so $E$ is strictly $\sigma_0$-semistable, and therefore $\sigma_-$-unstable.
Let $\aa_1, \dots, \aa_l$ be the Mukai vectors
of the HN filtration factors of $E$ with respect to $\sigma_-$. If all classes $\aa_i$ are positive,
$\aa_i \in P_\HH$, then we have an immediate contradiction to the assumptions.

Otherwise, $E$ must have
a spherical destabilizing subobject, or a spherical destabilizing quotient. 
Let $\tilde \ss$ be the class of this spherical object. If there is only one $\sigma_0$-stable
spherical object, then it is easy to see that $\vv_0 - \tilde \ss$ is in the positive cone;
therefore, $(\tilde \ss, \vv_0) < \frac{\vv_0^2}2$ in contradiction to our assumption.

\begin{wrapfigure}{r}{0.4\textwidth}
\vspace{-1em}
\begin{tikzpicture}[line cap=round,line join=round,>=triangle 45,x=1.5cm,y=1.5cm]
\clip(-0.3,-0.3) rectangle (3.5,2.5);
\draw[->,color=black] (-1,0) -- (3.5,0);
\foreach \x in {-1,1,2,3,4}
\draw[shift={(\x,0)},color=black] (0pt,2pt) -- (0pt,-2pt);
\draw[->,color=black] (0,-1) -- (0,2.5);
\foreach \y in {-1,1,2,3}
\draw[shift={(0,\y)},color=black] (2pt,0pt) -- (-2pt,0pt);
\draw [samples=50,domain=-0.99:0.99,rotate around={-135:(-1,-1)},xshift=-1.5cm,yshift=-1.5cm] plot ({2*(1+(\x)^2)/(1-(\x)^2)},{1.15*2*(\x)/(1-(\x)^2)});
\draw [samples=50,domain=-0.99:0.99,rotate around={-135:(-1,-1)},xshift=-1.5cm,yshift=-1.5cm] plot ({2*(-1-(\x)^2)/(1-(\x)^2)},{1.15*(-2)*(\x)/(1-(\x)^2)});
\draw [samples=50,domain=-0.99:0.99,rotate around={135:(1,0)},xshift=1.5cm,yshift=0cm] plot ({0.58*(1+(\x)^2)/(1-(\x)^2)},{1*2*(\x)/(1-(\x)^2)});
\draw [samples=50,domain=-0.99:0.99,rotate around={135:(0,1)},xshift=0cm,yshift=1.5cm] plot ({0.58*(-1-(\x)^2)/(1-(\x)^2)},{1*(-2)*(\x)/(1-(\x)^2)});
\draw (1.4,0.68) node[anchor=north west] {$\mathbf{v}^2=\mathbf{v}_0^2$};
\draw (1.4,1.3) node[anchor=north west] {$ (\mathbf{v}-\mathbf{t})^2=-2$};
\draw (0.4,2.2) node[anchor=north west] {$(\mathbf{v}-\mathbf{s})^2=-2$};
\draw (0.37,0.31) node[anchor=north west] {$ \mathbf{v}_0 $};
\fill [color=black] (1,0) circle (1.5pt);
\draw[color=black] (1.0,-0.18) node {$S_1$};
\fill [color=black] (0,1) circle (1.5pt);
\draw[color=black] (-0.18,1) node {$S_2$};
\fill [color=black] (0.44,0.39) circle (1.5pt);
\end{tikzpicture}
\caption{Proof of Proposition \ref{prop:noflops}}
\label{fig:NoFlop}
\vspace{-10pt}
\end{wrapfigure}

If there are two $\sigma_0$-stable spherical objects of classes $\ss, \tt$, consider the two vectors
$\vv_0-\ss$ and $\vv_0 -\tt$. The assumptions imply
$(\vv_0-\ss)^2 < -2$ and $(\vv_0-\tt)^2 < -2$; on the other hand, $\vv_0 - \tilde \ss$ is effective; 
using Lemma \ref{lem:JHspherical}, this implies that $\vv_0-\ss$ or $\vv_0-\tt$ must be effective.
We claim that this leads to a simple numerical contradiction.
Indeed, $(\vv_0-\tt)^2 < -2$ constrains $\vv_0$ to lie below a concave down hyperbola, and $(\vv_0-\ss)^2 < -2$ 
to lie above a concave up hyperbola; the two hyperbolas intersect at the points $0$ and $\ss + \tt$. 
Therefore, if we write $\vv_0 = x\ss + y\tt$, we have $x,y < 1$. Thus, neither $\vv_0-\ss$ nor $\vv_0-\tt$
can be effective (see Figure \ref{fig:NoFlop}).

In the case where $\vv$ is not minimal, $\vv \neq \vv_0$,
let $\Phi$ be the sequence of spherical twists given by
Proposition \ref{prop:sphericaltotallysemistable}. Since the assumptions of our proposition are
invariant under the $G_\HH$-action, they are also satisfied by $\vv_0$. By the previous case,
we know that every $\sigma_+$-stable objects $E_0$ of class $\vv_0$ is also $\sigma_0$-\emph{stable}. 
Thus $\Phi$ induces a morphism
$\Phi_* \colon M_{\sigma_+}(\vv_0) \to M_{\sigma_+}(\vv)$; since $\Phi_*$ is injective and 
the two spaces are smooth projective varieties of the same dimension, it is an isomorphism.
The S-equivalence class of $\Phi(E_0)$ is determined by that of $E_0$; since S-equivalence is a
trivial equivalence relation on $M_{\sigma_+}(\vv_0)$, the same holds for
$M_{\sigma_+}(\vv)$, and thus $\pi^+$ is an isomorphism.
\hfill$\Box$

Proposition \ref{prop:noflops} finishes the proof of Theorem \ref{thm:walls}.

\section{Main theorems}\label{sec:MainThms}

We will first complete the proof of Theorem \ref{thm:birational-WC}.

\begin{proof}[Proof of Theorem \ref{thm:birational-WC}, part \eqref{enum:birationalautoequivalence}]
We consider a wall $\WW$ with nearby stability conditions $\sigma_\pm$, and $\sigma_0 \in \WW$. 
Since $M_{\sigma_\pm}(\vv)$ are $K$-trivial varieties, it is sufficient to find
an open subset $U \subset M_{\sigma_\pm}(\vv)$ with complement of codimension two, and an
(anti-)autoequivalence $\Phi_\WW$ of $\Db(X,\alpha)$, such that
$\Phi_\WW(E)$ is $\sigma_-$-stable for all $E \in U$.

We will distinguish cases according to Theorem \ref{thm:walls}.
First consider the case when $\WW$ corresponds to a flopping contraction, or when $\WW$
is a fake wall.
If $\WW$ does not admit an effective spherical class $\ss \in \HH_\WW$ with $(\ss, \vv) < 0$
then we can choose $U$ to be the open subset of $\sigma_0$-\emph{stable}
objects; its complement has codimension two, and there is nothing to prove.
Otherwise, there exists a spherical object
destabilizing every object in $M_{\sigma_+}(\vv)$. Let $\vv_0 \in \HH_\WW$ be the
minimal class of the $G_\HH$-orbit of $\vv$, in the sense of Definition \ref{prop:pellequation}.
The subset $U$
of $\sigma_0$-\emph{stable} objects in $M_{\sigma_\pm}(\vv_0)$ has complement of codimension two.
Then the sequence of spherical twists of Proposition \ref{prop:sphericaltotallysemistable}, applied
for $\sigma_+$ and $\sigma_-$, identifies $U$ with subsets of $M_{\sigma_+}(\vv)$
and $M_{\sigma_-}(\vv)$ via derived equivalences $\Phi^+, \Phi^-$; then the composition
$\Phi^- \circ \left(\Phi^+\right)^{-1}$ has the desired property.

Next assume that $\WW$ induces a divisorial contraction. We have three cases to consider:
\begin{description}[leftmargin=0.5cm]
\item[Brill-Noether] Again, we first assume that $\vv$ is minimal, namely there is no
effective spherical class $\ss$ with $(\ss, \vv) < 0$.
The contracted divisor is described in Proposition
\ref{prop:sphericaldivisorialwall}, and the HN filtration of the destabilized
objects in Lemma \ref{lem:minimaldivisorialcontraction}. We may assume that we are in
the case where the Brill-Noether divisor in $M_{\sigma_+}(\vv)$ is described by
$\Hom(\widetilde S, \blank) \neq 0$. 
Now consider the spherical twist
$\ST_{\widetilde S}$ at $\widetilde S$, applied to objects $E \in M_{\sigma_+}(\vv)$. Note that
by $\sigma_+$-stability, we have $\Ext^2(\widetilde S, E) = \Hom(E, \widetilde S)^\vee = 0$
for any such $E$;
since $(\vv(\widetilde S), \vv(E)) = 0$, it follows that $\hom(\widetilde S, E) = \ext^1(\widetilde S, E)$.

If $E$ does not lie on the Brill-Noether divisor, then $\RHom(\widetilde S, E) = 0$, and so
$\ST_{\widetilde S}(E) = E$. Also, for generic such $E$ (away from a codimension two subset), the 
object $E$ is also $\sigma_-$-stable.

If $E$ is a generic element of the Brill-Noether divisor, then $\Hom(\widetilde S, E)
\cong \C \cong \Ext^1(\widetilde S, E)$, and hence we have an exact triangle
\[
\widetilde S \oplus \widetilde S[-1] \to E \to \ST_{\widetilde S}(E).
\]
Its long exact cohomology sequence with respect to the t-structure of $\sigma_0$ induces
two short exact sequences
\[ \widetilde S \into E \onto F  \quad \text{and} \quad F \into \ST_{\widetilde S}(E) \onto \widetilde S. \]
By Lemma \ref{lem:sphericaldivisorialwall}, the former is the HN filtration of $E$
with respect to $\sigma_-$; the latter is the dual extension, which is a $\sigma_-$-stable 
object by Lemma \ref{lem:stableextension}.

Thus, in both cases, $\ST_{\widetilde S}(E)$ is $\sigma_-$-stable.
This gives a birational map $M_{\sigma_+}(\vv)\dashrightarrow M_{\sigma_-}(\vv)$ defined in codimension two and induced by the autoequivalence $\ST_{\widetilde S}$, which is the claim we wanted to prove.

If instead there is an effective spherical class $\ss$ with $(\ss, \vv) < 0$, we reduce
to the previous case, similarly to the situation of flopping contractions:
Let $\vv_0$ again denote the minimal class
in the orbit $G_\HH.\vv$; note that $\WW$ also induces a divisorial contraction of
Brill-Noether type for $\vv_0$. In this case, Lemma \ref{lem:sphericaldivisorialwall}
states that the sequence $\Phi$ of spherical twists identifies
an open subset $U^+ \subset M_{\sigma_+}(\vv_0)$ (with complement of codimension two)
with an open subset of $M_{\sigma_+}(\vv)$; similarly for $U^- \subset M_{\sigma_-}(\vv_0)$. 
Combined with the single spherical twist identifying a common open subset of
$M_{\sigma_\pm}(\vv_0)$, this implies the claim.

\item[Hilbert-Chow]
Here $\WW$ is an isotropic wall and there exists an isotropic primitive vector $\ww_0$ with $(\ww_0,\vv)=1$.
As shown in Section
\ref{sec:iso}, we may assume that shift by one identifies
$M_{\sigma_+}(\vv)$ with the $(\beta,\omega)$-Gieseker moduli space 
$M_{\omega}^{\beta}(-\vv)$ of stable sheaves of rank one on a twisted K3 surface $(Y, \alpha')$.
After tensoring with a line bundle, we may assume that objects in $M_{\sigma_+}(\vv)$ 
are exactly the shifts $I_Z[1]$ of ideal sheaves of 0-dimensional subschemes $Z \subset Y$.

In the setting of Proposition \ref{prop:Jason}, we have $\beta = 0$. Since there are
line bundles on $(Y, \alpha')$, the Brauer group element $\alpha'$ is trivial. By the last
statement of the same Proposition, the moduli space $M_{\sigma_-}(\vv)$ parameterizes
the shifts of derived duals ideal sheaf. Thus there is a natural isomorphism
$M_{\sigma_-}(\vv) \cong M_{\sigma_+}(\vv)$ induced by the derived anti-autoequivalence
$(\blank)^\vee[2]$.
\item[Li-Gieseker-Uhlenbeck]
Here $\WW$ is again isotropic, but $(\ww_0,\vv)=2$.
We will argue along similar lines as in the previous case; unfortunately, the details are more
involved. The first difference
is that we cannot assume $\beta = 0$. Instead,
first observe that $M_{\sigma_+}(\vv) = M_{\omega}^{\beta}(-\vv)$
is parameterizing $(\beta,\omega)$-Gieseker stable sheaves $F$ of rank $2 = (\vv, \ww)$, and of slope
$\mu_{\omega}(F) = \omega.\beta$. If we assume $\omega$ to be generic, then 
Gieseker stability is independent of the choice of $\beta$; we can consider
$M_{\sigma_+}(\vv) = M_{\omega}(-\vv)$ to be the moduli space of shifts $F[1]$ of
$\omega$-Gieseker stable sheaves $F$. 

Since $(Y, \alpha')$ admits rank two vector bundles, the order of $\alpha'$ in the Brauer group is
one or two; in both cases, we can identify $(Y, \alpha')$ with $(Y, (\alpha')^{-1})$, and thus
the derived dual $E \mapsto E^\vee$ defines an anti-autoequivalence of $\Db(Y, \alpha')$.

Write $-\vv = (2, c, d)$, and let $\LL$ be the line bundle with $c_1(L) = c$.
From the previous discussion it follows
that $\Phi(\blank) = (\blank)^\vee \otimes \LL[2]$ is the desired functor:

Indeed, any object in $M_{\sigma_+}(\vv)$ is of the form $F[1]$ for a $\omega$-Gieseker stable sheaf
$F$ of class $-\vv$. Then $\Phi(F[1]) = F^\vee \otimes \LL[1]$ the derived dual of a Gieseker stable
sheaf, and has class $\vv$. By Proposition \ref{prop:Jason}, this is an object of
$M_{\sigma_-}(\vv)$.
\end{description}
\end{proof}

Consider two adjacent chamber $\CC^+, \CC^-$ separated by a wall $\WW$; as always, 
we pick stability conditions $\sigma_{\pm} \in \CC^{\pm}$, and a stability condition
$\sigma_0 \in \WW$.
By the identification of N\'eron-Severi groups induced
by Theorem \ref{thm:birational-WC}, we can think of the corresponding maps
$\ell_{\pm}$ of equation \eqref{eq:elltoAmp} as maps
\[
\ell_{\pm} \colon \CC^{\pm} \to \NS(M_{\sigma_+}(\vv)).
\]
They can be written as the following composition of maps
\[
\Stab^\dag(X,\alpha) \xrightarrow{\ZZ} H^*_{\alg}(X,\alpha,\Z) \otimes \C \xrightarrow{I} \vv^\perp
\xrightarrow{\theta_{\CC^\pm}} \NS(M_{\sigma_+}(\vv))
\]
where $\ZZ$ is the map defined in Theorem \ref{thm:Bridgeland_coveringmap}, 
$I$ is given by $I(\Omega_Z) = \Im \frac{\Omega_Z}{-(\Omega_Z, \vv)}$, and where
$\theta_{\CC^\pm}$ are the Mukai morphisms, as reviewed in Remark \ref{rmk:comparison}. 

Our next goal is to show that these two maps behave as nicely as one could hope; we will
distinguish two cases according to the behavior of the contraction morphism
\[
\pi^+ \colon M_{\sigma_+}(\vv) \to \overline{M}^+
\]
induced by $\WW$ via Theorem \ref{thm:contraction}:

\begin{Lem} \label{lem:ellandreflections}
The maps $\ell^+, \ell^-$ agree on the wall $\WW$ (when extended by continuity).
\begin{enumerate}
\item \label{enum:isomorsmall} (Fake or flopping walls)
When $\pi^+$ is an isomorphism, or a small contraction, then 
the maps $\ell_+, \ell_-$ are analytic continuations of each other.
\item \label{enum:divisorial} (Bouncing walls)
When $\pi^+$ is a divisorial contraction, then the analytic continuations of
$\ell^+, \ell^-$ differ by the reflection $\rho_D$ in $\NS(M_{\sigma_+}(\vv))$ at the 
divisor $D$ contracted by $\ell_{\sigma_0}$.
\end{enumerate}
\end{Lem}

As a consequence, in case \eqref{enum:isomorsmall} the wall $\WW$ is a fake wall when $\pi^+$ is
an isomorphism, and induces a flop when $\pi^+$ is a small contraction; in case
\eqref{enum:divisorial}, corresponding to a divisorial contraction, the moduli spaces
$M_{\sigma_\pm}(\vv)$ for the two adjacent chambers are isomorphic.

\begin{proof}
We have to prove $\theta_{\CC^-} = \theta_{\CC^+}$ in case \eqref{enum:isomorsmall}, and
$\theta_{\CC^-} = \rho_D \circ \theta_{\CC^+}$ in case \eqref{enum:divisorial}. We 
assume for simplicity that the two moduli spaces admit universal families; the arguments
apply identically to quasi-universal families.

Consider case \eqref{enum:isomorsmall}. If the wall is not totally semistable, then the two
moduli spaces $M_{\CC^\pm}(\vv)$ share a common open subset, with complement of codimension two,
on which the two universal families agree. 
By the projectivity of the moduli spaces, the maps $\theta_{\CC^\pm}$ are determined by their restriction
to curves contained in this subset; this proves the claim. If the wall is instead totally
semistable, we additionally have to use Proposition \ref{prop:sphericaltotallysemistable}.
Let $\Phi^+$ and $\Phi^-$ be the two sequences of spherical twists, sending $\sigma_0$-stable
objects of class $\vv_0$ to $\sigma_+$- and $\sigma_-$-stable objects of class $\vv$, respectively.
The autoequivalence inducing the birational map $M_{\sigma_+}(\vv) \dashrightarrow
M_{\sigma_-}(\vv)$ is given by $\Phi^- \circ (\Phi^+)^{-1}$. 
As the classes of the spherical objects occurring in $\Phi^+$ and $\Phi^-$ are identical, this 
does not change the class of the universal family in the $K$-group; therefore, the Mukai morphisms
$\theta_{\CC^+}, \theta_{\CC^-}$ agree.

Now consider the case of a Brill-Noether divisorial contraction;
we first assume that there is no effective spherical class $\ss' \in \HH_\WW$ with
$(\ss', \vv) < 0$.
The contraction induced by a spherical object $S$ with Mukai vector $\ss := \vv(S) \in \vv^\perp$. 
By Lemma \ref{lem:sphericaldivisorialwall}, the class of the contracted divisor is given by
$\theta_{\CC^\pm}(\ss)$ on either side of the wall.
The universal families differ (up to a subset of codimension two) by the spherical twist
$\ST_S(\blank)$. This induces the reflection at $\ss$ in $H^*_\alg(X,\alpha,\Z)$; thus the Mukai morphisms differ by
reflection at $\theta(\ss)$, as claimed.

If in addition to $\ss \in \vv^\perp$, there does exist
an effective spherical class $\ss' \in \HH_\WW$ with $(\ss', \vv) < 0$, we have to rely
on the constructions of Lemma \ref{lem:sphericaldivisorialwall}, as
in the proof of Theorem \ref{thm:birational-WC}. We have a common open subset
$U \subset M_{\sigma_\pm}(\vv_0)$, such that the two universal families $\EE^{\pm}|_U$
are related by the spherical twist at a spherical object $S_0$ of class $\ss_0$. Let 
$\Phi^{\pm}$ be the sequences of spherical twists obtained from Lemma
\ref{lem:sphericaldivisorialwall}, applied to $\sigma_+$ or $\sigma_-$, respectively.
Their induced maps $\Phi_*^{\pm} \colon H^*_\alg(X,\alpha,\Z) \to H^*_\alg(X,\alpha,\Z)$ on the Mukai lattice are identical,
as they are obtained by twists of spherical objects of the same classes; it sends
$\vv_0$ to $\vv$, and thus $\ss_0$ to $\pm\ss$. Therefore, the composition
$\Phi^- \circ \ST_{S_0} \circ (\Phi^+)^{-1}$ induces the reflection at $\ss$, as claimed.

It remains to consider divisorial contractions of Hilbert-Chow and Li-Gieseker-Uhlenbeck type.  We
may assume $M_{\sigma_+}(\vv)$ is the Hilbert scheme, or a moduli space of Gieseker stable sheaves of
rank two.
By the proof of Theorem \ref{thm:birational-WC}, there is a line bundle
$\LL$ on $X$ such that
\[
\RlHom_{M_{\sigma_\pm(\vv)} \times X}(\EE, (p_X)^*\LL[2])
\]
is a universal family with respect to
$\sigma_-$ on $M_{\sigma_-}(\vv) = M_{\sigma_+}(\vv)$. We use equation \eqref{eq:definetheta}
to compare $\theta_{\CC^\pm}$ by evaluating their degree on a test curve $C \subset M_{\sigma_\pm}(\vv)$.
Let $i$ denote the inclusion $i \colon C \times X \into M_{\sigma_\pm}(\vv) \times X$, and
by $p$ the projection $p \colon C \times X  \to X$.  This yields the following chain of equalities
for $\aa \in \vv^\perp$:

\begin{align}
\theta_{\CC^-}(\aa).C &= 
\Bigl(\aa, \vv\bigl(p_* i^*(\EE^-)\bigr) \Bigr) 
= \Bigl(\aa, \vv\bigl(p_* \RlHom_{C \times X}(i^*\EE, \OO_C \boxtimes \LL[2])\bigr) \Bigr) \label{eqdual} \\
&= \Bigl(\aa, \vv\bigl(p_* \RlHom_{C \times X}(i^*\EE, \omega_C[1] \boxtimes \LL[1])\bigr) \Bigr) \label{eqavperp} \\
&= -\Bigl(\aa, \vv\bigl(\RlHom_X(p_*i^*\EE, \LL)\bigr) \Bigr) \label{Grothendieck} \\
& = -\Bigl(\aa^\vee \cdot \ch(\LL), \vv(p_*i^*\EE)\Bigr)
= - \theta_{\CC^+}\bigl(\aa^\vee \cdot \ch(\LL)\bigr).C \label{eq:Columbus20121222}
\end{align}
Here we used compatibility of duality with base change in \eqref{eqdual}, $\aa \in \vv^\perp$ in
\eqref{eqavperp}, and Grothendieck duality in \eqref{Grothendieck}. In \eqref{eq:Columbus20121222},
we wrote $\aa^\vee$ for the class corresponding to $\aa$ under duality $(\blank)^\vee$.

In the Hilbert-Chow case, with $\vv = -(1, 0, 1-n)$, the class of the contracted divisor
$D$ is proportional to $\theta_{\CC^+}(1, 0, n-1)$, and we have $\LL \cong \OO_X$;
in the Li-Gieseker-Uhlenbeck case, we can write
$\vv = (2, c, d)$, the class of the contracted divisor $D$ is a multiple of
$\theta_{\CC^+}(2, c, \frac{c^2}2-d)$, and $c_1(\LL) = c$. In both cases,
the reflection $\rho_D$ is compatible with the above chain of equalities:
\begin{equation}\label{eq:Columbus20130813}
\rho_D\left(\theta_{\CC^+}(\aa)\right) = -\theta_{\CC^+}\left(\aa^\vee\cdot \ch(\LL)\right).
\end{equation}
Indeed, in the HC case, we can test \eqref{eq:Columbus20130813} for $\aa_1=(1,0,1-n)$ and classes of
the form $\aa_2=(0,c',0)\in \theta_{\CC^+}^{-1}\left(D^\perp\right)$:
since $\aa_1^\vee=\aa_1$ and $\aa_2^\vee = - \aa_2$, and since such classes
span $\vv^\perp$, the equality follows.
Similarly, in the LGU case, we can use $\aa_1=(2,c,\frac{c^2}2-d)$ and
$\aa_2=(0,c',\frac{c'.c}{2})$.
\end{proof}

\begin{proof}[Proof of Theorem \ref{thm:MAP}, \eqref{enum:piecewise}, \eqref{enum:imagemovable}, \eqref{enum:allMMP}]
Lemma \ref{lem:ellandreflections} proves part \eqref{enum:piecewise}. Part \eqref{enum:allMMP}
follows directly from the positivity $\ell_\CC(\CC) \subset \mathrm{Amp} M_\CC(\vv)$ once
we have established part \eqref{enum:imagemovable}.

Consider a big class in the movable cone, given as $\theta_\sigma(\aa)$ for some class
$\aa \in \vv^\perp, \aa^2 > 0$; we have to show that it is in the image of $\ell$.
Recall the definition of $\PP_0^+(X,\alpha)$ given in the
discussion preceding Theorem \ref{thm:Bridgeland_coveringmap}. If we set
\[
\Omega'_\aa := \sqrt{-1}\aa - \frac{\vv}{\vv^2} \in \Halg \otimes \C,
\]
then clearly $\Omega'_\aa \in \PP(X,\alpha)$.
In case there is a spherical class $\ss \in H^*_\alg(X,\alpha,\Z)$ with $(\Omega'_\aa, \ss) = 0$, we modify
$\Omega'_\aa$ by a small real multiple of $\ss$ to obtain $\Omega_\aa \in \PP_0(X,\alpha)$, otherwise we
set $\Omega_\aa = \Omega'_\aa$; in either case, we have
$\Omega_\aa \in \PP_0(X,\alpha)$ with $(\Omega_\aa, \vv) = -1$ and $\Im \Omega_\aa = \aa$.
In addition, the fact that $\theta(\aa)$ is contained in the positive cone gives
$\Omega_\aa \in \PP_0^+(X,\alpha)$. 

Let $\Omega_\sigma \in \PP_0^+(X,\alpha)$ be the central charge for the chosen basepoint
$\sigma \in \Stab^\dag(X,\alpha)$. Then there is a path $\gamma \colon [0,1] \to \PP_0^+(X,\alpha)$ starting at
$\Omega_\sigma$ and ending at $\Omega_\aa$ with the following additional property: for all $t \in
[0,1]$, the class
\[
-\frac{\theta_\sigma(\Im \gamma(t))}{(\gamma(t), \vv)}
\]
is contained in the
movable cone of $M_\sigma(\vv)$.

By Theorem \ref{thm:Bridgeland_coveringmap}, there is a lift $\sigma \colon [0,1] \to
\Stab^\dag(X,\alpha)$ of $\gamma$ starting at $\sigma(0) = \sigma$. By
the above assumption on $\gamma$, this will never hit a wall of the movable cone
corresponding to a divisorial contraction; by Lemma \ref{lem:ellandreflections}, the
map $\ell$ extends analytically, with $\theta_\sigma = \theta_{\sigma(0)}
= \theta_{\sigma(1)}$. Therefore, 
\[
\ell_{\sigma(1)}(\sigma(1)) = \theta_{\sigma(1)}(\aa) = \theta_\sigma(\aa)
\]
as claimed.
\end{proof}

Now recall the Weyl group action of $W_{\mathrm{Exc}}$ of Proposition \ref{prop:Weylmovablecone}. 
The exceptional chamber of a hyperbolic reflection group intersects every orbit
exactly once.  Thus there is a map
\[
W \colon \Pos(M_\sigma(\vv)) \to {\Mov}(M_\sigma(\vv))
\]
sending any class to the intersection of its $W_{\mathrm{Exc}}$-orbit with the fundamental domain.
Lemma \ref{lem:ellandreflections} and Theorem \ref{thm:MAP} immediately give
the following global description of $\ell$:

\begin{Thm} \label{thm:ellandgroupaction}
The map $\ell$ of Theorem \ref{thm:MAP} can be given as the composition of the following
maps:
\[
{\Stab}^\dagger(X,\alpha) \xrightarrow{\ZZ} H^*_\alg(X,\alpha,\Z) \otimes \C
\xrightarrow{I} \vv^\perp \xrightarrow{\theta_{\sigma,\vv}} \Pos(M_\sigma(\vv))
\xrightarrow{W} {\Mov}(M_\sigma(\vv)).
\]
\end{Thm}

To complete the proof of Theorem \ref{thm:MAP}, it remains to prove part \eqref{enum:AmpleCone}:

\begin{Prop}
Let $\CC \subset \Stab^\dag(X,\alpha)$ be a chamber of the chamber decomposition with
respect to $\vv$. Then the image of $\ell_\CC(\CC) \subset \NS(M_\CC(\vv))$ of the
chamber $\CC$ is exactly the ample cone of the corresponding moduli space
$M_\CC(\vv)$.
\end{Prop} 
\begin{proof}
In light of Theorems \ref{thm:ampleness} and \ref{thm:MAP}, \eqref{enum:piecewise}, \eqref{enum:imagemovable}, \eqref{enum:allMMP}, the only potential problem is given
by  walls $\WW \subset \partial \CC$ that do not get mapped to walls of the nef cone of
the moduli space. These are totally semistable fake walls induced by an effective
spherical class $\ss \in \HH_\WW$ with $(\ss, \vv) < 0$. The idea is that there is always a potential
wall $\WW'$, with the same lattice $\HH_{\WW'} = \HH_\WW$, for which all effective spherical classes have 
positive pairing with $\vv$. By Theorem \ref{thm:walls}, $\WW'$ is not a wall, and it will have
the same image in the nef cone of $M_\CC(\vv)$ as the wall $\WW$.

Let $\sigma_0 = (Z_0, \PP_0) \in \WW$ be a very general stability condition on the given wall: 
this means we can assume that $\HH_\WW$ contains all integral classes $\aa \in H^*_\alg(X,\alpha,\Z)$
with $\Im Z_0(\aa) = 0$. If we
write $Z_0(\blank) = (\Omega_0, \blank)$ as in Theorem \ref{thm:Bridgeland_coveringmap}, we may
assume that $\Omega_0$ is normalized by
$(\Omega_0, \vv) = -1$ and $\Omega_0^2 = 0$, i.e., $(\Re \Omega_0, \Im \Omega_0) = 0$
and $(\Re \Omega_0)^2 = (\Im \Omega_0)^2$ (see \cite[Section 10]{Bridgeland:K3}). We will now
replace $\sigma_0$ by a stability condition whose central charge has real part given by
$(-\vv, \blank)$, and identical imaginary part.

To this end, let $\sigma_1 \in \CC$ be a stability condition nearby $\sigma_0$, whose central charge
is defined by $\Omega_1 = \Omega_0 +i\epsilon$, where $\epsilon \in H^*_\alg(X,\alpha,\Z)\otimes \R$
is a sufficiently small vector with $(\epsilon, \vv) = 0$; we may also assume that multiples
of $\vv$ are the only integral classes $\aa \in H^*_\alg(X,\alpha,\Z)$ with $(\Im \Omega_1, \aa) = 0$. Let
$\Omega_2 = -\vv + i \Im \Omega_1$; then a straight-forward computation shows that the straight path 
connecting $\Omega_1$ with $\Omega_2$ lies completely within $\PP_0^+(X,\alpha)$. Finally, 
let $\Omega_3 = -\vv + \Im \Omega_0$; by Theorem \ref{thm:walls}, there are no spherical classes
$\tilde \ss \in \HH_\WW$ with $(\vv, \tilde \ss) = 0$, implying that the straight path from $\Omega_2$ to
$\Omega_3$ is also contained in $\PP_0^+(X,\alpha)$.

By Theorem \ref{thm:Bridgeland_coveringmap}, there is a lift of the path 
$\Omega_0 \mapsto \Omega_1 \mapsto \Omega_2 \mapsto \Omega_3$ to $\Stab^\dag(X,\alpha)$; let $\sigma_2$ and $\sigma_3$
the stability conditions corresponding to $\Omega_2$ and $\Omega_3$, respectively.
By choice of $\epsilon$, we may assume that the paths $\sigma_0 \mapsto \sigma_1$ and
$\sigma_2 \mapsto \sigma_3$ do not cross any walls. Since $(\Omega_1, \vv) = (\Omega_2, \vv) = -1$,
and since the imaginary part on the path $\Omega_1 \mapsto \Omega_2$ is constant, the same holds for
the path $\sigma_1 \mapsto \sigma_2$. Hence $\sigma_3$ is in the closure of the chamber
$\CC$. In particular, $\sigma_3$ lies on a potential wall of $\CC$ with hyperbolic lattice given by
$\HH_\WW$; by construction, any spherical class $\ss \in \HH_\WW$ with $(\vv, \ss) < 0$ satisfies
$(\Omega_3, \ss) > 0$, and thus $\ss$ is not effective.

By Theorem \ref{thm:walls}, $\sigma_3$ does not lie on a wall. Since $\Im \Omega_3 = \Im \Omega_0$,
the images $l_\CC(\sigma_0) = l_\CC(\sigma_3)$ in the N\'eron-Severi group of $M_\CC(\vv)$ agree.
\end{proof}

We conclude this section by proving Corollary \ref{cor:Mukaifull} for moduli spaces of
Bridgeland stable objects on twisted K3 surfaces:

\begin{proof}[Proof of Corollary \ref{cor:Mukaifull}]
$\eqref{enum:Markmancrit} \Rightarrow \eqref{enum:equivnumerical}$:
Let $\phi \colon H^*(X, \alpha, \Z) \to H^*(X', \alpha', \Z)$ be a Hodge isometry sending $\vv^{\perp, \tr} \to
\vv'^{\perp, \tr}$. Up to composing with $[1]$, we may assume $\phi(\vv) = \vv'$. 
If $\phi$ is orientation-preserving, then Theorem \ref{thm:derivedtorelli}, gives an equivalence $\Phi$ with 
$\Phi_* = \phi$.

Otherwise, the composition $(\blank)^\vee \circ \phi$ defines an orientation-preserving Hodge
isometry 
\[ \phi^\vee \colon H^*(X, \alpha, \Z) \to H^*(X', (\alpha')^{-1}, \Z) \quad
\text{with} \quad \phi^\vee(\vv) = (\vv')^\vee. \]
Again, Theorem \ref{thm:derivedtorelli} gives a derived equivalence
$\Psi \colon \Db(X, \alpha) \to \Db(X', (\alpha')^{-1})$; the composition with
$(\blank)^\vee \colon \Db(X', (\alpha')^{-1}) \to \Db(X', \alpha')$ has the desired property.

$\eqref{enum:equivnumerical} \Rightarrow \eqref{enum:equivbirational}$:
Assume that $\Phi\colon \Db(X,\alpha)\xrightarrow{\simeq}\Db(X',\alpha')$ is an (anti-)equivalence
with $\Phi_*(\vv)=\vv'$.
Consider moduli spaces $M_{\sigma}(\vv), M_{\sigma'}(\vv')$ of Bridgeland-stable objects.
We claim that we can assume  the existence of
$\tau \in \Stab^\dag(X',\alpha')$ such that $E \in \Db(X, \alpha)$ of class $\vv$ is
$\sigma$-stable if and only if $\Phi(E)$ is $\tau$-stable:
\begin{itemize*}
\item if $\Phi_*$ is orientation-preserving,
we may replace $\Phi$ by an equivalence satisfying the last claim of Theorem \ref{thm:derivedtorelli}, and set
$\tau = \Phi_*(\sigma) \in \Stab^\dag(X, \alpha')$;
\item
otherwise, we may assume that 
\[ (\blank)^\vee \circ \Phi \colon \Db(X, \alpha) \to \Db(X', (\alpha')^{-1})
\]
satisfies the same claim, and we let
$\tau \in \Stab^\dag(X, \alpha')$ be the stability condition dual
(in the sense of Proposition \ref{prop:dualstability}) to
$\left((\blank)^\vee \circ \Phi\right)_* \sigma \in \Stab^\dag(X',(\alpha')^{-1})$.
\end{itemize*}

By construction, $\Phi$ induces an isomorphism
\[ M_\sigma(\vv) \cong M_{\tau}(\Phi_*(\vv)) = M_{\tau}(\vv'). \]
Due to Theorem \ref{thm:birational-WC}, part \eqref{enum:birationalautoequivalence}, there is
an (anti-)autoequivalence $\Phi'$ of $\Db(X',\alpha')$ inducing a birational map
$M_{\Psi_*(\sigma)}(\vv') \dashrightarrow M_{\sigma'}(\vv')$. The composition 
$\Phi' \circ \Phi$ has the desired properties.
\end{proof}

\section{Application 1: Lagrangian fibrations}\label{sec:Application1}

In this section, we will explain how birationality of wall-crossing implies 
Theorem \ref{thm:SYZ}, verifying the Lagrangian fibration conjecture.

We will prove the theorem for any moduli space $M_\sigma(\vv)$ of Bridgeland-stable objects on
a twisted K3 surface $(X, \alpha)$, under the assumptions that $\vv$ is primitive, and $\sigma$ generic with respect to $\vv$.

One implication in Theorem \ref{thm:SYZ} is immediate: if $f \colon M_\sigma(\vv) \dashrightarrow Z$ is
a rational abelian fibration, then the pull-back $f^*D$ of any ample divisor $D$ on $Z$ has volume
zero; by equation \eqref{eq:BBform}, the self-intersection of $f^*D$ with respect to the
Beauville-Bogomolov form must also equal zero.

To prove the converse, we will first restate carefully the argument establishing part \ref{enum:birational}
of Conjecture \ref{conj:SYZ}, which was already sketched in the introduction; then we will explain
how to extend the argument to also obtain part \ref{enum:nef}.

Assume that there is an integral divisor $D$ on $M_\sigma(\vv)$ with $q(D)=0$.
Applying the inverse of the Mukai morphism $\theta_{\sigma, \vv}$ of Theorem \ref{thm:ModuliSpacesAreHK},
we obtain a primitive vector $\ww = \theta_{\sigma, \vv}^{-1}(D) \in \vv^\perp$ with
$ \ww^2 = 0.$

After a small deformation, we may assume that $\sigma$ is also generic with respect to $\ww$.
As in Section \ref{sec:iso},
we consider the moduli space $Y:=M_\sigma(\ww)$ of $\sigma$-stable objects, which is a
smooth K3 surface. There is a derived equivalence
\begin{equation} \label{eq:edinburgh0903}
\Phi \colon \Db(X,\alpha) \xrightarrow{\sim} \Db(Y,\alpha')
\end{equation}
for the appropriate choice of a Brauer class $\alpha'\in\mathrm{Br}(Y)$; as before, we have $\Phi_*(\ww) = (0, 0, 1)$.

By the arguments recalled in Theorem \ref{thm:derivedtorelli}, we have $\Phi_*(\sigma) \in \Stab^\dag(Y, \alpha')$. 
By definition, $\Phi$ induces an isomorphism
\begin{equation}\label{eq:comparison}
M_{\sigma}(\vv) \cong M_{\Phi_*(\sigma)}(\Phi_*(\vv)),
\end{equation}
where $\Phi_*(\sigma)$ is generic with respect to $\Phi_*(\vv)$.

\begin{Lem}\label{lem:isotropic}
The Mukai vector $\Phi_*(\vv)$ has rank zero.
\end{Lem}
\begin{proof}
This follows directly from $\Phi_*(\ww)=(0,0,1)$ and $(\Phi_*(\ww),\Phi_*(\vv)) = (\ww,\vv)=0$.
\end{proof}

We write $\Phi_*(\vv)=(0,C,s)$, with $C\in\mathrm{Pic}(Y)$ and $s\in\Z$.
Since $\vv^2>0$ we have $C^2>0$.
\begin{Lem} \label{lem:Cisample}
After replacing $\Phi$ by the composition $\Psi \circ \Phi$, where
$\Psi \in \Aut(\Db(Y, \alpha'))$, we may assume that $C$ is ample, and that $s \neq 0$.
\end{Lem}
\begin{proof}
Up to shift $[1]$, we may assume that $H'.C > 0$, for a given ample class $H'$ on $Y$. In
particular, $C$ is an effective class; it is ample unless there is a rational $-2$-curve
$D \subset Y$ with $C. D < 0$. Applying the spherical twist $\ST_{\OO_D}$ at the structure
sheaf\footnote{Note that the restriction of $\alpha$ to any curve vanishes, hence
the structure sheaf $\OO_D$ is a coherent sheaf on $(Y, \alpha')$.} of $D$ replaces
$C$ by its image $C'$ under the reflection at $D$, which satisfies $C'.D > 0$. This procedure
terminates, as the nef cone is a fundamental domain of the Weyl group
action generated by reflections at $-2$-curves.

Since tensoring with an (untwisted) line bundle on $Y$ induces an autoequivalence
of $\Db(Y, \alpha')$, we may also assume $s \neq 0$.
\end{proof}

Let $H'\in\mathrm{Amp}(Y)$ be a generic polarization with respect to $\Phi_*(\vv)$. 
The following is a small (and well-known) generalization of Beauville's integrable system
\cite{Beauville:ACIS}:

\begin{Lem}\label{lem:fibration}
The moduli space $M_{H'}(\Phi_*(\vv))$ admits a structure of Lagrangian fibration
induced by global sections of $\theta_{H', \Phi_*(\vv)}((0,0,-1))$.
\end{Lem}

\begin{proof}
Let $M':=M_{H'}(\Phi_*(\vv))$ and $L' := \theta_{H', \Phi_*(\vv)}((0,0,-1))$.
By an argument of Faltings and Le Potier (see \cite[Section 1.3]{LePotier:StrangeDuality}), we can construct sections of $L'$ as follows:
for all $y\in Y$, we define a section $s_y\in H^0(M',L')$ by its zero-locus
\[
Z(s_y) := \left\{ E\in M'\colon \Hom(E,k(y))\neq0 \right\}.
\]
Whenever $y$ is not in the support of $E$, the section $s_y$ does not vanish at $E$;
hence the sections $\{s_y\}_{y\in Y}$ generate $L'$. Consider the morphism induced by this linear
system. The image of $E$ is determined
by its set-theoretic support; hence the image of $M'$ is the complete
local system of $C$. By Matsushita's theorem \cite{Matsushita:Fibrations, Matsushita:Addendum},
the map must be a Lagrangian fibration.
\end{proof}

By Remark \ref{rem:Giesekerchamber}, there exists a generic stability condition
$\sigma'\in\Stab^\dagger(Y,\alpha')$ with the property that $M_{H'}(\Phi(\vv))=M_{\sigma'}(\Phi(\vv))$.
On the other hand, by the birationality of wall-crossing,
Theorem \ref{thm:birational-WC}, the moduli spaces
$M_{\sigma'}(\Phi_*(\vv))$ and $M_{\Phi_*(\sigma)}(\Phi_*(\vv))$ are birational; combined
with the identification \eqref{eq:comparison}, this shows that $M_{\sigma}(\vv)$ is birational
to a Lagrangian fibration.

It remains to prove part \ref{enum:nef}, so let us assume that $D$ is nef and primitive.
Using the Fourier-Mukai
transform $\Phi$ as above, and after replacing $\sigma$ by $\Phi_*(\sigma)$, we may also assume that
$\vv$ has rank zero, and that $\ww = \theta_{\sigma, \vv}^{-1}(D)$ is the class of skyscraper sheaves of
points.
Now consider the autoequivalence $\Psi \in \Aut \Db(Y, \alpha')$ of Lemma \ref{lem:Cisample}. Except
for the possible shift $[1]$, each autoequivalence used in the construction of $\Psi$ leaves the
class $\ww$ invariant. Thus, in the moduli space $M_{\Psi_*\sigma}(\Psi_*(\vv)) = M_\sigma(\vv)$, the 
divisor class $D$ is still given by $D = \pm \theta_{\Psi_*\sigma, \Psi_*(\vv)}(\ww)$, up to sign.

Let $f \colon M_\sigma(\vv)  \dashrightarrow M_H(\vv)$ be the birational map to the Gieseker moduli
space $M_H(\vv)$ of torsion sheaves induced by a sequence of wall-crossings as above.  The
Lagrangian fibration $M_H(\vv) \to \P^n$ is induced by the divisor $\theta_{H,\vv}(-\ww)$.  By Theorem
\ref{thm:ellandgroupaction}, the classes $f_*D$ and $\theta_{H,\vv}(-\ww)$ are (up to sign) in the
same $W_{\mathrm{Exc}}$-orbit. Since they are both nef on a smooth K-trivial birational model, they
are also in the closure of the movable cone (and in particular, their orbits agree, not just up to sign).

By Proposition \ref{prop:Weylmovablecone}, the closure of the movable cone is the closure of the
fundamental chamber of the action on $W_{\mathrm{Exc}}$ on $\Pos(M)$, which intersects
every orbit exactly once.
Therefore, the classes $f_*D$ and $\theta_{H,\vv}(-\ww)$ have to be equal.

Since $M_\sigma(\vv)$ and $M_H(\vv)$ are isomorphic in codimension two, the section rings of
$D$ and $f_*D$ agree. In particular, $D$ is effective with Iitaka dimension $\frac{\vv^2+2}{2}$.
As explained in \cite[Section 4.1]{Sawon:AbelianFibred}, it follows from
\cite{Verbitsky:Cohomology} that
the numerical Iitaka dimension of $D$ is also equal to $\frac{\vv^2+2}{2}$.
Since $D$ is nef by assumption, $D$ is semi-ample by Kawamata's Theorem (see \cite[Theorem 6.1]{Kawamata:Pluricanonical}
and \cite[Theorem 1.1]{Fujino:KawamataThm}), and thus induces a morphism to $\P^n$.
This completes the proof of Theorem \ref{thm:SYZ}.

\begin{Rem} \label{Rem:SYZcon_movable}
In fact, the above proof shows the following two additional statements:
\begin{enumerate}
\item If $D \in \NS(M_\sigma(\vv))$ with $q(D) = 0$ lies in the closure of the movable cone, 
then there is a birational Lagrangian fibration induced by $D$. (In particular, $D$ is movable.)
\item Any $W_{\mathrm{Exc}}$-orbit of divisors on $M_{\sigma(\vv)}$ satisfying
$q(D) = 0$ contains exactly one movable divisor, which induces a birational Lagrangian fibration.
\end{enumerate}
\end{Rem}

\section{Application 2: Mori cone, nef cone, movable cone, effective cone}
\label{sec:cones}

Let $\vv$ be a primitive vector with $\vv^2 > 0$, let $\sigma$ be a generic stability condition
with respect to $\vv$, and let $M := M_\sigma(\vv)$ be the moduli space of $\sigma$-semistable objects.
In this section, we will completely describe the cones associated to the birational geometry of $M$ in terms of the Mukai lattice of $X$.

Recall that $\overline{\Pos}(M) \subset \NS(M)_\R$ denotes the (closed) cone of positive classes
defined by the Beauville-Bogomolov quadratic form.
Let $\overline{\Pos}(M)_\Q \subset \overline{\Pos}(M)$ be the subcone generated by all rational classes
in $\overline{\Pos(M)}$; it is the union of the interior
$\Pos(M)$ with all rational rays in the boundary $\partial \Pos(M)$. We fix an ample divisor class
$A$ on $M$ (which can be obtained from Theorem \ref{thm:ampleness}).

In the following theorems, we will say that a subcone of $\overline{\Pos}(M)_\Q$ (or of its closure) is ``cut
out'' by a collection of linear subspaces if it is one of the closed chambers of the wall-and-chamber
decomposition of $\Pos(M)_\Q$ whose walls are the given collection of subspaces. This is easily
translated into a more explicit statement as in the formulation of Theorem \ref{thm:nefcone} given
in the introduction.

\begin{Thm} \label{thm:nefcone}
The nef cone of $M$ is cut out in $\overline{\Pos}(M)$ by all linear subspaces of the form
$\theta(\vv^\perp \cap \aa^\perp)$, for all classes $\aa \in H^*_\alg(X,\alpha,\Z)$ satisfying $\aa^2 \ge -2$
and $0 \le (\vv, \aa) \le \frac{\vv^2}2$.
\end{Thm}

Via the Beauville-Bogomolov form we can identify the group $N_1(M)$ of curves up to 
numerical equivalences with a lattice in the N\'eron-Severi group:
$N_1(M)_\Q \cong \left(N^1(M)_\Q\right)^\vee \cong N^1(M)_\Q$. In particular, we get an
induced rational pairing on $N_1(M)$; we then say that the \emph{cone of positive curves}
is the cone of classes $[C] \in N_1(M)_\R$ with $(C, C) > 0$ and $C.A > 0$.
Also, we obtain a dual Mukai isomorphism 
\begin{equation} \label{eq:dualMukai}
\theta^\vee \colon H^*_\alg(X,\alpha,\Z)/\vv \otimes \Q \to N_1(M)_\Q.
\end{equation}
As the dual statement to Theorem \ref{thm:nefcone}, we obtain:
\begin{Thm} \label{thm:Moricone}
The Mori cone of curves in $M$ is generated by the cone of positive curves, and by all curve classes
$\theta^\vee(\aa)$, for all  $\aa \in H^*_\alg(X,\alpha,\Z), \aa^2 \ge -2$ satisfying
$\abs{(\vv, \aa)} \le \frac{\vv^2}2$ and $\theta^\vee(\aa).A > 0$.
\end{Thm}
Some of these classes $\aa$ may not define a wall bordering the nef cone; in this case,
$\theta^\vee(\aa)$ is in the interior of the Mori cone (as it intersects every nef divisor
positively).

\begin{Thm} \label{thm:movable}
The movable cone of $M$ is cut out in $\overline{\Pos}(M)_\Q$ by the following two types of walls:
\begin{enumerate}
\item $\theta(\ss^\perp \cap \vv^\perp)$ for every spherical class $\ss \in \vv^\perp$.
\item $\theta(\ww^\perp \cap \vv^\perp)$ for every isotropic class $\ww \in H^*_\alg(X,\alpha,\Z)$ with $1 \le (\ww, \vv) \le 2$.
\end{enumerate}
\end{Thm}

\begin{Thm} \label{thm:effective}
The effective cone of $M$ is generated by $\overline{\Pos}(M)_\Q$ along with the following 
exceptional divisors:
\begin{enumerate}
\item $D:= \theta(\ss)$ for every spherical class $\ss \in \vv^\perp$ with
$(D, A) > 0$, and
\item $D:= \theta(\vv^2 \cdot \ww - (\vv,\ww) \cdot \vv)$ for every isotropic class $\ww \in H^*_\alg(X,\alpha,\Z)$
with $1 \le (\ww, \vv) \le 2$ and $(D, A) > 0$.
\end{enumerate}
\end{Thm}

Note that only those classes $D$ whose orthogonal complement $D^\perp$ is a wall of the movable cone
will correspond to irreducible exceptional divisors.

The movable cone has essentially been described by Markman for any hyperk\"ahler variety; more
precisely, \cite[Lemma 6.22]{Eyal:survey} gives the intersection of the movable cone with the
strictly positive cone $\Pos(M)$. While our methods give an alternative proof, the only new statement of
Theorem \ref{thm:movable} concerns rational classes $D$ with $D^2 = 0$ in the closure of the movable
cone; such a $D$ is movable due to our proof of the Lagrangian fibration conjecture in Theorem
\ref{thm:SYZ}.

Using the divisorial Zariski decomposition of \cite{Boucksom:Zariskidecomposition}, one can show
for any hyperk\"ahler variety that the pseudo-effective cone is dual to the closure of the movable
cone. In particular, Theorem \ref{thm:effective} could also be deduced from Markman's results and
Theorem \ref{thm:SYZ}.

\begin{proof}[Proof of Theorem \ref{thm:nefcone}]
Let $\CC$ be the chamber of $\Stab^\dag(X,\alpha)$ containing $\sigma$.  By Theorem \ref{thm:MAP}, the boundary
of the ample cone inside the positive cone is equal to the union of the images $\ell(\WW)$, for all
walls $\WW$ in the boundary of $\CC$ that induce a non-trivial contraction morphism. (These are
walls that are not ``fake walls'' in the sense of Definition \ref{def:TypeOfWalls}.) Theorem
\ref{thm:walls} characterizes hyperbolic lattices corresponding to such walls.

For any such hyperbolic lattice $\HH$, we get a class $\aa$ as in Theorem \ref{thm:nefcone} as follows:
\begin{itemize}
\item
in the cases \eqref{enum:niso-divisorial} of divisorial contractions, we let $\aa$ be the corresponding spherical or
isotropic class;
\item in the subcase of \eqref{enum:niso-flop} of a flopping contraction induced by a spherical class
$\ss$, we also set $\aa = \ss$;
\item and in the subcase of \eqref{enum:niso-flop} of a flopping contraction induced by a
sum $\vv = \aa + \bb$, we may assume $(\vv, \aa) \le (\vv, \bb)$, which is equivalent
to $(\vv, \aa) \le \frac{\vv^2}2$.
\end{itemize}
Stability conditions $\sigma = (Z, \AA)$ in the corresponding wall $\WW$ satisfy
$\Im \frac{Z(\aa)}{Z(\vv)} = 0$, or, equivalently, $\ell(\sigma) \in \theta(\vv^\perp \cap \aa^\perp)$.

Conversely, given $\aa$, we obtain a rank two lattice $\HH := \langle \vv, \aa \rangle$. If $\HH$ is
hyperbolic, then it is straightforward to check that it conversely induces one of the walls listed
in Theorem \ref{thm:walls}. Otherwise, $\HH$ is positive-semidefinite. Then the orthogonal complement
$\HH^\perp = \vv^\perp \cap \aa^\perp$ does not contain any positive classes, and thus its image
under $\theta$ in $\NS(M)$ does not intersect the positive cone and can be ignored.
\end{proof}

\begin{proof}[Proof of Theorem \ref{thm:movable}.]
As already discussed in Section \ref{sec:MainThms}, the intersection
$\Mov(M) \cap \Pos(M)$ follows directly from Theorem \ref{thm:MAP}; the statement of 
Theorem \ref{thm:movable} is just an explicit description of the exceptional chamber 
of the Weyl group action. 

A movable class $D$ in the boundary of the positive cone, with $(D, D) = 0$,
automatically has to be rational.  Conversely, by our proof of Theorem \ref{thm:SYZ}, if
we have a rational divisor with $(D, D) = 0$ that is in the closure of the movable cone, then
there is a Lagrangian fibration induced by $D$ on a smooth birational model of $M$;
in particular, $D$ is movable.
\end{proof}

\begin{proof}[Proof of Theorem \ref{thm:effective}.]
We first claim that the class of an irreducible exceptional divisor is (up to a multiple) of the form described in the Theorem. For the
Brill-Noether case, this was proved in \ref{prop:sphericaldivisorialwall}. In the Hilbert-Chow and
Li-Gieseker-Uhlenbeck case, the class of the divisor of non-locally free sheaves can be computed
explicitly; alternatively, it is enough to observe that $\theta^{-1}(D)$ has to be a multiple of the orthogonal
projection of $\ww$ to $\vv^\perp$. 

If $D$ is an arbitrary effective divisor, then $D$ can be written
as $D = A + E$ with $A$ movable and $E$ exceptional, see \cite[Section
4]{Boucksom:Zariskidecomposition}, \cite[Theorem 5.8]{Eyal:survey}.
The class of $A$ is a rational point of $\overline{\Pos(M)}$. Thus the effective cone is
contained in the cone described in the Theorem.

For the converse, first recall $\Pos(M) \subset \Eff(M)$. Now consider a rational divisor with $D^2
= 0$. If $(D, E) < 0$ for some exceptional divisor $E$, then $D$ can be written 
the sum $D = \epsilon E + (D - \epsilon E)$ with $D - \epsilon E \in \Pos(M)$; thus $D$ is in the effective cone. Otherwise
$(D, E) \ge 0$ for every exceptional divisor $E$. By Proposition \ref{prop:Weylmovablecone},
$D$ is in the closure of the movable cone; by Theorem \ref{thm:SYZ} and Remark
\ref{Rem:SYZcon_movable}, a multiple of $D$ induces a birational Lagrangian fibration, so $D$ is effective. 

Finally, when $D$ is one of the classes listed explicitly in the Theorem, consider the orthogonal
complement $D^\perp$. If it does not intersect the movable cone in a face or in the interior, then the inequality
$(D, \blank) \ge 0$ is implied by the inequality $(E, \blank) \ge 0$ for all irreducible exceptional
divisors; hence $D$ is a positive linear combination of such divisors. Since the wall $D^\perp$ is
identical to one of the walls listed in Theorem \ref{thm:movable}, the only other possibility is
that $D^\perp$ defines a wall of the movable cone. The corresponding exceptional divisor is
proportional to $D$.
\end{proof}

\subsection*{Relation to Hassett-Tschinkel's conjecture on the Mori cone}
Hassett and Tschinkel gave a conjectural description of the nef and Mori cones via intersection
numbers of extremal rays in \cite{HassettTschinkel:ExtremalRays}. While their conjecture turned out
to be incorrect (see \cite[Remark 10.4]{BM:projectivity} and \cite[Remark 8.10]{KnutsenCiliberto}),
we will now explain that it is in fact very closely related to Theorem \ref{thm:Moricone}.

We first recall their conjecture. Via the identification
$N_1(M)_\Q \cong N^1(M)_\Q$ explained above, the Beauville-Bogomolov extends to
a $\Q$-valued quadratic form on $N_1(M)$; we will also
denote it by $q(\blank)$.
The following lemma follows immediately from this definition, and the definition
of $\theta^\vee$:

\begin{Lem} \label{lem:thetaveequadratic}
Consider the isomorphism $\vv^\perp_\Q \cong N_1(M)_\Q$ induced by the dual Mukai morphism
$\theta^\vee$ of \eqref{eq:dualMukai}. This isomorphism respects the quadratic form on
either side.
\end{Lem}

Let $2n$ be the dimension of $M$, and as above let $A$ be an ample divisor.
Let $C \subset N_1(M)_\R$ be the cone generated by all integral curve classes
$R \in N_1(M)_\Z$ that satisfy $q(R) \ge -\frac{n+3}2$ and $R.A > 0$.
In \cite[Conjecture 1.2]{HassettTschinkel:ExtremalRays}, the authors conjectured that for
any hyperk\"ahler variety $M$ deformation equivalent to the Hilbert scheme of a K3
surface, the cone $C$ is equal to the Mori cone.

Our first observation shows that the Mori cone is contained in $C$:
\begin{Prop}\label{prop:RelationHT}
Let $R$ be the generator of an extremal ray of the Mori cone of $M$. Then
$(R, R) \ge -\frac{n+3}2$.
\end{Prop}
\begin{proof}
It is enough to prove the inequality for some effective curve on the extremal ray.
Let $\WW$ be a wall inducing the extremal contraction corresponding to the ray generated by $R$, and
$\HH_\WW \subset \Halg$ its associated hyperbolic lattice. Let $\sigma_+$ be a nearby stability
condition in the chamber of $\sigma$, and $\sigma_0 \in \WW$.  Let $\aa \in \HH_\WW$ be a
corresponding class satisfying the assumptions in Theorem \ref{thm:Moricone}: $\aa^2 \ge -2$ and
$\abs{(\vv, \aa)} \le \frac{\vv^2}2$. Replacing $\aa$ by $-\aa$ if necessary, we can also
assume $(\vv, \aa) \ge 0$.

We first claim that there exists a contracted curve whose integral class is given by
$\pm\theta^\vee(\aa)$. We ignore the well-known case of the Hilbert-Chow contraction, and also
assume for simplicity we assume that $\WW$ is not a totally semistable wall for any
class in $\HH_\WW$; the general case can be reduced to this one with the same methods as in the
previous sections. By assumptions, we have both $\aa^2 \ge -2$ and $(\vv-\aa)^2 \ge -2$; therefore,
we can choose $\sigma_0$-\emph{stable} objects $A$ and $B$ of class $\aa$ and $\vv - \aa$,
respectively. We further claim $(\aa, \vv) \ge 2 + \aa^2$:
this claim is trivial when $\aa^2 = -2$, amounts to the exclusion of the Hilbert-Chow case
when $\aa^2 = 0$, and in case $\aa^2 > 0$ it follows from  the signature
of $\HH_\WW$ and the assumption on $(\aa, \vv)$:
\[
(\aa, \vv)^2 > \aa^2 \vv^2 \ge 2\aa^2 (\aa,\vv) \ge (\aa^2 + 2) (\aa, \vv).
\]

Assume that $\phi^+(\aa) < \phi^+(\vv) < \phi^+(\vv-\aa)$; in the opposite case we
swith the roles of $A$ and $B$.  By the above claim, $\ext^1(B, A) = (\aa, \vv-\aa) \ge 2$.
Varying the extension class in $\Ext^1(B, A)$ produces curves of objects in
$M_{\sigma_+}(\vv)$ that are S-equivalent with respect to $\sigma_0$;
in order to compute their class, we have to make the construction explicit.
Let $\P(\Ext^1(B, A))$ be the projective space of one-dimensional subspaces of $\Ext^1(B,A)$.
Choose a parameterized line $\P^1 \into \P(\Ext^1(B, A))$, corresponding to a section $\nu$ of
\[
H^0(\P^1, \OO(1) \otimes \Ext^1(B,A))
= \Ext^1_{\P^1 \times X}(\OO_{\P^1} \boxtimes B, \OO_{\P^1}(1) \boxtimes A).
\]
Let $\EE \in \Db(\P^1 \times X)$ be the extension
$\OO_{\P^1}(1) \boxtimes A \to \EE \to \OO_{\P^1} \boxtimes B $
given by $\nu$. By Lemma \ref{lem:stableextension}, every fiber of $\EE$ is
$\sigma_+$-stable. Thus we have produced a rational curve $R \subset M_{\sigma_+}(\vv)$
of S-equivalent objects. 

To compute its class, it is sufficient to compute the intersection product
$\theta(D).R$ with a divisor $\theta(D)$, for any $D \in \vv^\perp$. We have
\[
\theta(D).R = (D, \vv(\Phi(\OO_R)) = (D, \vv(B) + 2 \vv(A)) = (D, \vv + \aa) = (D, \aa) = 
\theta(D).\theta^\vee(\aa),
\]
where $\Phi \colon \Db(M_{\sigma^+}(\vv)) \to \Db(X)$ denotes the Fourier-Mukai transform, 
and where we used $D \in \vv^\perp$ in the second-to-last equality.

Let $\aa_0 \in \Halg$ denote the projection of $\aa$ to the orthogonal complement of $\vv$.
By Lemma \ref{lem:thetaveequadratic}, we have $(R, R) = \aa_0^2$, and for the latter we
obtain:
\[
(\aa_0, \aa_0) = \left( \aa - \frac{(\vv, \aa)}{\vv^2} \vv, \aa - \frac{(\vv, \aa)}{\vv^2} \vv\right)
= \aa^2 - \frac{(\vv, \aa)^2}{\vv^2} \ge -2 - \frac{\vv^2}4  = - \frac{n+3}2.
\]
\end{proof}

\begin{Rem}\label{rmk:RelationHT}
When $M$ is the Hilbert scheme of points on $X$, we can make the comparison to Hassett-Tschinkel's
conjecture even more precise: in this case, it is easy to see that $\theta^\vee$ induces an
isomorphism 
\[
\Halg/\vv \to N_1(M)
\]
of lattices, respecting the integral structures.  Given a class $R \in N_1(M)$ satisfying the
inequality $(R, R) \ge -\frac{n+3}2$ of \cite{HassettTschinkel:ExtremalRays}, let 
$\aa_0 \in \vv^\perp_\Q$ be the (rational) class with $\theta^\vee(\aa_0) = R$. 
Let $k$ be any integer satisfying $k \le n-1$ and
$k^2 \ge (2n-2)(-2-\aa_0^2)$; by the assumptions, $k = n-1$ is always an example satisfying both
inequalities.
Then $\aa := \aa_0 + \frac{k}{2n-2} \vv$ is a rational class in the algebraic Mukai lattice that
satisfies the assumptions appearing in Theorem \ref{thm:Moricone}. In addition, it has
has integral pairing with both $\vv$, and with every integral class in $\vv^\perp$; thus, it is
potentially an integral class. The Hassett-Tschinkel conjecture holds  if and only if for every extremal
ray of $C$, there is a choice of $k$ such that $\aa$ is an integral class.
\end{Rem}

If we are given a lattice $\vv^\perp$ of small rank, then the algebraic Mukai lattice of $(X,
\alpha)$ can be any lattice in $\vv^\perp_\Q \oplus \Q \cdot \vv$ containing both $\vv^\perp$ and
$\vv$, as long as $\vv$ and $\vv^\perp$ are primitive. In general, the Hassett-Tschinkel conjecture
holds for some of these lattices, but not for others. The question is thus closely related to
the fact that a strong global Torelli statement needs the embedding $H^2(M) \into H^*(X)$, rather
than just $H^2(M)$.

\section{Examples of nef cones and movable cones}
\label{sec:examples}

In this section we examine examples of cones of divisors.

\subsection*{K3 surfaces with Picard number 1...}

Let $X$ be a K3 surface such that $\mathrm{Pic}(X) \cong \Z \cdot H$, with $H^2=2d$.
We let $M:=\mathrm{Hilb}^n(X)$, for $n\geq2$, and $\vv=(1,0,1-n)$.
In this case, everything is determined by certain Pell's equations.
We recall that a basis of $\mathrm{NS}(M)$ is given by 
\begin{equation} \label{eq:thetaexplicit}
 \widetilde H = \theta(0,-H,0) \quad \text{and} \quad B = \theta(-1,0,1-n).
\end{equation}
Geometrically, $\widetilde H$ is the big and nef divisor induced by the symmetric power of $H$ on
$\Sym^n(X)$, and $2B$ is the class of the exceptional divisor of the Hilbert-Chow morphism.

By Theorem \ref{thm:walls}, divisorial contractions can be divided in three cases:
\begin{description}
\item[Brill-Noether] If there exists a spherical class $\ss$ with $(\ss,\vv)=0$.
\item[Hilbert-Chow] If there exists an isotropic class $\ww$ with $(\ww,\vv)=1$.
\item[Li-Gieseker-Uhlenbeck] If there exists an isotropic class $\ww$ with $(\ww,\vv)=2$.
\end{description}

Elementary substitutions show that the case of BN-contraction is governed by 
solution to Pell's equation
\begin{equation}\label{eq:Pell(-2)}
(n-1) X^2 - d Y^2 = 1 \quad \text{via} \quad 
\ss(X, Y) = (X,-YH,(n-1)X). 
\end{equation}
The case of HC-contractions and LGU-contractions are governed solutions of
\begin{equation}\label{eq:Pell(0)}
X^2 - d (n-1) Y^2 = 1 \quad \text{with $X+1$ divisible by $n-1$};
\end{equation}
here we will also allow solutions with $X, Y < 0$, in case the positive solution is $\equiv -1 (n-1)$. 
We get a HC-contraction via
\[
\ww(X,Y)=\left(\frac{X+1}{2(n-1)},-\frac Y2 H,\frac{X-1}2\right)  \]
if both $Y$ and $\frac{X+1}{n-1}$ are even\footnote{The published version incorrectly assumes only that $Y$ is even.}, and otherwise an LGU-conctraction via
\[ 
\ww(X,Y)=\left(\frac{X+1}{n-1},-Y H,X-1\right).
\]
The two equations determine the movable cone:

\begin{Prop}\label{prop:MovHilbSchemePic1}
Assume $\mathrm{Pic}(X)\cong \Z \cdot H$.
The movable cone of the Hilbert scheme $M=\mathrm{Hilb}^n(X)$ has the following form:
\begin{enumerate}
\item\label{enum:Columbus120912_1} If $d=\frac{k^2}{h^2}(n-1)$, with $k,h\geq1$, $(k,h)=1$, then
\[
\mathrm{Mov}(M) = \langle \widetilde H, \widetilde H - \frac kh B\rangle,
\]
where $q(h\widetilde H - kB)=0$, and it induces a (rational) Lagrangian fibration on $M$.
\item\label{enum:Columbus120912_2} If $d(n-1)$ is not a perfect square, and \eqref{eq:Pell(-2)} has a solution, then
\[
\mathrm{Mov}(M) = \langle \widetilde H, \widetilde H - d\frac{y_1}{x_1(n-1)}B\rangle,
\]
where $(x_1,y_1)$ is the solution to \eqref{eq:Pell(-2)} with $x_1, y_1 > 0$, and with smallest
possible $x_1$.
\item\label{enum:Columbus120912_3} If $d(n-1)$ is not a perfect square, and \eqref{eq:Pell(-2)} has no solution, then
\[
\mathrm{Mov}(M) = \langle \widetilde H, \widetilde H - d\frac{y_1'}{x_1'}B\rangle,
\]
where $(x_1',y_1')$ is the solution to \eqref{eq:Pell(0)} with smallest possible
$\frac{y_1'}{x_1'}>0$ (where both $x_1', y_1' > 0$ or $x_1', y_1' < 0$ are allowed).
\end{enumerate}
\end{Prop}

\begin{proof}
Since $\widetilde H$ induces the divisorial HC contraction, it is an extremal ray of the movable
cone; to find the other extremal ray, we need to find $\Gamma > 0$ such that
$\widetilde H - \Gamma B$ lies on one of the walls described by Theorem \ref{thm:movable}, and such
that $\Gamma$ is as small as possible.

Also recall Proposition \ref{prop:Weylmovablecone}: the movable cone is a
fundamental domain for Weyl group action of $W_{\mathrm{Exc}}$ on $\Pos(M)$. 
Any solution to \eqref{enum:Columbus120912_2} or \eqref{enum:Columbus120912_3} determines a wall in
the positive cone via Theorem \ref{thm:movable}; one of its Weyl group translates
thus determines a wall bordering the movable cone.

Part \eqref{enum:Columbus120912_1} follows directly from Theorem \ref{thm:SYZ}.
To prove part \eqref{enum:Columbus120912_2}, it follows immediately from the previous discussion that
if \eqref{eq:Pell(-2)} has a solution, then
one of the solutions determines the second extremal ray. The claim thus follows from the observation
that
\[
D(X,Y) :=\widetilde H - d\frac{Y}{X(n-1)}B  = \theta\left( \left(\frac{dY}{X(n-1)}, -H,
d\frac{Y}{X}\right)  \right)
\]
is obtained as the image under $\theta$ of a class orthogonal to $\ss(X, Y)$, and the fact that
$\frac YX$ is minimal if and only if $X$ is minimal. 

A similar computation shows that given a solution of \eqref{eq:Pell(0)} (which always exists),
the vector 
\[ D'(X, Y) = \widetilde H - d\frac YX B = \theta\left( \left( \frac{dY}{X}, -H,
(n-1)\frac{dY}X\right) \right) \]
 is contained in $\theta(\ww(X, Y)^\perp)$ in both
the HC and the LGU case; this proves part \eqref{enum:Columbus120912_3}.
\end{proof}

\begin{Ex}
If $d = n-2$, then \eqref{eq:Pell(-2)} has $X = 1, Y=1$ as a solution. Therefore
\[
\mathrm{Mov}(M)=\langle \widetilde H, \widetilde H - \frac{n-2}{n-1} B \rangle.
\]
\end{Ex}

For the nef cone, we start with the easy case $n=2$.
Consider the Pell's equation
\begin{equation}\label{eq:(-2)flops_n=2}
X^2 - 4dY^2 = 5.
\end{equation}
The associated spherical class is $\ss(X,Y) = \left(\frac{X+1}2,-Y H,\frac{X-1}2\right)$.

\begin{Lem}\label{lem:n=2}
Let $M=\mathrm{Hilb}^2(X)$.
The nef cone of $M$ has the following form:
\begin{enumerate}
\item\label{enum:Columbus120913_1} If \eqref{eq:(-2)flops_n=2} has no solutions, then
\[
\mathrm{Nef}(M) = \mathrm{Mov}(M).
\]
\item\label{enum:Columbus120913_2} Otherwise, let $(x_1,y_1)$ be the positive solution of \eqref{eq:(-2)flops_n=2} 
with $x_1>0$ minimal.
Then\footnote{In the published version, a factor 2 is missing in the following formula.}
\[
\mathrm{Nef}(M) = \langle \widetilde H, \widetilde H - 2d \frac{y_1}{x_1} B \rangle.
\]
\end{enumerate}
\end{Lem}

\begin{proof}
We apply Theorem \ref{thm:nefcone}.
The movable cone and the nef cone agree unless there is a flopping wall, described
in Theorem \ref{thm:walls}, part \eqref{enum:niso-flop}.
Since $\vv^2 = 2$, the case $\vv = \aa + \bb$ with $\aa, \bb$ positive 
is impossible. This leaves only the case of a spherical class $\ss$
with $(\vv,\ss) = 1$; this exists if and only if \eqref{eq:(-2)flops_n=2} has a solution.
\end{proof}

\begin{Ex}\label{ex:CK}
Let $d=31$.
Then the nef cone for $M=\mathrm{Hilb}^2(X)$ is
\[
\mathrm{Nef}(M) = \langle \widetilde H, \widetilde H - \frac{3658}{657} B\rangle.
\]
In particular, this gives a negative answer to \cite[Question 8.4]{KnutsenCiliberto}.

Indeed, \eqref{eq:(-2)flops_n=2} has a the smallest solution given by $x_1=657$ and $y_1=118$.
This gives a $(-2)$-class $\ss = (329, -59 \cdot H, 328)$, which induces a flop, by Lemma \ref{lem:n=2}.
\end{Ex}

For higher $n>2$ the situation is more complicated, since the number of Pell's equations to consider is higher.
But, in any case, everything is completely determined.

\begin{Ex}\label{ex:n=7}
Consider the case in which $d=1$ and $n=7$, $M=\mathrm{Hilb}^7(X)$.
This example exhibits a flop of ``higher degree'': it is induced by a
decomposition $\vv = \aa + \mathbf{b}$, with $\aa^2, \mathbf{b}^2>0$, and not induced by a
spherical or isotropic class.
Indeed, if $\vv=(1,0,-6)$, $\aa=(1,-H,0)$ and $\bb=(0,H,-6)$, then
the rank two hyperbolic lattice associated to this wall contains no spherical or isotropic classes.
The full list of walls in the movable cone is given the table below.
We can write the nef divisor associated to a wall as $\widetilde H -\Gamma B$, for
$\Gamma\in\Q_{>0}$; as before, the value of
$\Gamma$ is determined from \eqref{eq:thetaexplicit} given an element of $\vv^\perp\cap\aa^\perp$.

\begin{center}
    \begin{tabular}{|c|c|c|c|c|}
    \hline
    &&&&\\[-6pt]
    $\Gamma$ & $\aa$ & $\aa^2$ & $(\vv,\aa)$ & Type \\
    &&&&\\[-6pt]
    \hline
    &&&&\\[-6pt]
    $0$ & $(0,0,-1)$ & 0 & 1 & divisorial contraction \\
    &&&&\\[-6pt]
    \hline
    &&&&\\[-6pt]
    $\frac 14$ & $(1,-H,2)$ & -2 & 4 & flop \\
    &&&&\\[-6pt]
    \hline
    &&&&\\[-6pt]
    $\frac 27$ & $(1,-H,1)$ & 0 & 5 & flop \\
    &&&&\\[-6pt]
    \hline
    &&&&\\[-6pt]
    $\frac 13$ & $(1,-H,0)$ & 2 & 6 & flop \\
    &&&&\\[-6pt]
    \hline
    &&&&\\[-6pt]
    $\frac{6}{17}$ & $(2,-3H,5)$ & -2 & 7 & fake wall \\
    &&&&\\[-6pt]
    \hline
    &&&&\\[-6pt]
    $\frac{4}{11}$ & $(1,-2H,5)$ & -2 & 1 & flop \\
    &&&&\\[-6pt]
    \hline
    &&&&\\[-6pt]
    $\frac 38$ & $-(1,-3H,10)$ & -2 & 4 & flop\\
    &&&&\\[-6pt]
    \hline
    &&&&\\[-6pt]
    $\frac 25$ & $(1,-2H,4)$ & 0 & 2 & divisorial contraction\\[-6pt]
    &&&&\\
    \hline
    \end{tabular}
\end{center}
\end{Ex}

\subsection*{...and higher Picard number}

Let $X$ be a K3 surface such that $\mathrm{Pic}(X) \cong \Z \cdot \xi_1 \oplus \Z \cdot \xi_2$.

\begin{Ex}
We let $M:=\mathrm{Hilb}^2(X)$, and $\vv=(1,0,-1)$.
We assume that the intersection form (with respect to the basis $\xi_1,\xi_2$) is given by
\[
q =
\begin{pmatrix}
28 & 0\\
0 & -4
\end{pmatrix}.
\]
Such a K3 surface exists, see \cite{Morrison:K3,Kovacs:K3}.
We have:
\[
\mathrm{NS}(M) = \Z \cdot \mathbf{s} \oplus \mathrm{NS}(X),
\]
where $\mathbf{s} = (1,0,1)$.
Our first claim is
\begin{equation}\label{enum:HigherPic1}
\mathrm{Nef}(M) = \mathrm{Mov}(M).
\end{equation}
Indeed, by Theorem \ref{thm:walls}, a flopping contraction would have to come from a class $\aa$
with $\aa^2 \ge -2$ and $(\vv, \aa) = 1$; also, the corresponding lattice $\HH = \langle \vv, \aa \rangle$
has to be hyperbolic, which implies $\aa^2 \le 0$. In addition, $\aa^2 = 0$ would correspond to the
Hilbert-Chow divisorial contraction, and thus $\aa^2 = -2$ is the only possibility. 
If we write $\aa = (r, D, r-1)$ with $D = a\xi_1 + b\xi_2$, this gives
\[ -2r(r-1) + 28 a^2 - 4b^2 = -2.
\]
This equation has no solutions modulo 4.

The structure of the nef cone is thus determined by divisorial contractions.
These are controlled by the quadratic equation
\begin{equation}\label{eq:HigherPic2}
X^2 - 2 (7a^2 - b^2) = 1,
\end{equation}
via $\aa = (X,D,X)$.
For example, the Hilbert-Chow contraction corresponds to the solution $a=b=0$ and $X=1$ to \eqref{eq:HigherPic2}.
Other contractions arise, for example, at $a=4$, $b=0$, $X=15$, or $a=2$, $b=2$, $X=7$, etc.
The nef cone is locally polyhedral but not finitely generated. Its walls have an
accumulation point at the boundary, coming from a solution to
\[
X^2 - 2 (7a^2 - b^2) = 0
\]
corresponding to a Lagrangian fibration.
\end{Ex}

We continue to consider the case where $X$ has Picard rank two.
To increase the flexibility of our examples, we now also consider a twist by a Brauer class $\alpha\in\mathrm{Br}(X)$.
We choose $\alpha = e^{\beta_0}$ for some $B$-field class
$\beta_0 \in H^2(X, \Q)$ with 
\[
\beta_0.\mathrm{NS}(X)=0\quad \text{ and } \quad \beta_0^2=0.
\]
(See \cite{HMS:generic_K3s} for more details; in particular, the existence of our examples follows as in \cite[Lemma 3.22]{HMS:generic_K3s}.)

\begin{Ex}\label{ex:RoundNef}
We assume that $2\beta_0$ is integral, and that the intersection form on
\[
H^*_{\alg}(X,\alpha,\Z)=\mathrm{NS}(X) \oplus \Z\cdot (2,2\beta_0,0) \oplus \Z \cdot (0,0,-1)
\]
takes the form
\[
q =
\begin{pmatrix}
4 & 0 & 0 & 0\\
0 & -4 & 0 & 0\\
0 & 0 & 0 & 2\\
0 & 0 & 2 & 0
\end{pmatrix}.
\]
Consider the primitive vector $\vv=(0,\xi_1,0)$, and let $M:=M_H(\vv)$ be the moduli space of $\alpha$-twisted $H$-Gieseker semistable sheaves on $X$, for $H$ a generic polarization on $X$.
Then:
\begin{enumerate}
\item\label{enum:RoundRational1} $\mathrm{Nef}(M)=\mathrm{Mov}(M)$;
\item\label{enum:RoundRational2} $\mathrm{Nef}(M)$ is a rational circular cone.
\end{enumerate}

To prove the above statements, observe that $\vv^2 = 4$ and
$(\vv, \aa) \in 4\Z$ for all $\aa \in H^*_\alg(X, \alpha, \Z)$. According to Theorem
\ref{thm:walls}, the only possible wall in this situation would be given by
a Brill-Noether divisorial contraction, coming from a spherical class $\ss \in \vv^\perp$. But the
above lattice admits no spherical classes, and thus there are no walls.

Thus the nef cone and the closure of the movable cone are both equal to the positive cone.
Since $M$ obviously admits Lagrangian fibrations, the cone is rational.
\end{Ex}

Modifying the previous example slightly, we obtain a moduli space with circular movable cone and locally polyhedral nef cone:

\begin{Ex}
Now assume $3\beta_0$ is integral, and that 
\[
H^*_{\alg}(X,\alpha,\Z)=\mathrm{NS}(X) \oplus \Z\cdot (3,3\beta_0,0) \oplus \Z \cdot (0,0,-1)
\]
has intersection form given by
\[
q =
\begin{pmatrix}
6 & 0 & 0 & 0\\
0 & -6 & 0 & 0\\
0 & 0 & 0 & 3\\
0 & 0 & 3 & 0
\end{pmatrix}.
\]
Consider the primitive vector $\vv=(0,\xi_1,1)$, and let $M:=M_H(\vv)$.
Then:
\begin{enumerate}
\item\label{enum:MovRoundNefPoly1} $\mathrm{Nef}(M)$ is a rational locally-polyhedral cone;
\item\label{enum:MovRoundNefPoly2} $\mathrm{Mov}(M)$ is a rational circular cone.
\end{enumerate}

Indeed, \eqref{enum:MovRoundNefPoly2} follows exactly as in Example \ref{ex:RoundNef}: there are no
spherical classes, and, for all $\aa\in H^*_{\alg}(X,\alpha,\Z)$, $(\aa,\vv)\in3\Z$.
However, flopping contractions are induced by solutions to the quadratic equation
\[
a^2-b^2-2as+s=0,
\]
where we set $D=a\xi_1+b\xi_2$, and $\aa=(3(2a-1),a\xi_1+b\xi_2+ 3(2a-1)\beta_0,s)$.
This has infinitely many solutions.
It is an easy exercise to deduce \eqref{enum:MovRoundNefPoly1} from this.
\end{Ex}

\section{The geometry of flopping contractions}
\label{sec:floppingeometry}

One can also refine the analysis leading to Theorem \ref{thm:walls} to give a precise description of
the geometry of the flopping contraction associated to a flopping wall $\WW$.

As in Section \ref{sec:hyperbolic}, we let $\sigma_0 \in \WW$ be a stability condition on the wall,
and $\sigma_+ \notin \WW$ be sufficiently close to $\sigma_0$.
For simplicity, let us assume throughout this section that the hyperbolic lattice $\HH_\WW$ associated
to $\WW$ via Definition \ref{def:potentialwall} does not admit spherical or isotropic classes; in particular,
$\WW$ is not a totally semistable wall for any class $\aa \in \HH$, and does not induce
a divisorial contraction. 

Let $\PPP$ be the set of unordered partitions
$P = [\aa_i]_i$ of $\vv$ into a sum $\vv = \aa_1 + \dots + \aa_m$ of positive classes
$\aa_i \in \HH$. We say that a partition $P$ is a refinement of another partition
$Q = [\bb_i]_i$ if it can be obtained by choosing partitions of each $\bb_i$. 
This defines a natural partial order on $\PPP$, with $P \prec Q$ if $P$ is a refinement of $Q$.
The trivial partition as the maximal element of $\PPP$.

Given $P  = [\aa_i]_i \in \PPP$, we let $M_P \subset M_{\sigma_+}(\vv)$ be the subset of objects $E$ such
that the Mukai vectors of the Jordan-H\"older factors $E_i$ of $E$ with respect to $\sigma_0$ are
given by $\aa_i$ for all $i$. Using openness of stability and closedness of semistability in
families, one easily proves:

\begin{Lem} The disjoint union $M_{\sigma_+}(\vv) = \coprod_{P \in \PPP} M_P$ defines a
stratification of $M_{\sigma_+}(\vv)$ into locally closed subsets, such that
$M_P$ is contained in the closure of $M_Q$ if and only if $P \prec Q$.
\end{Lem}

In addition, our simplifying assumptions on $\HH_\WW$ give the following:
\begin{Lem} \label{lem:MAnonempty}
Assume that $P = [\aa_1, \aa_2]$ is a two-element partition of $\vv$. Then $M_P \subset
M_{\sigma_+}(\vv)$ is non-empty, and of codimension $(\aa_1, \aa_2) - 1$. 
\end{Lem}
\begin{proof}
Since $\vv$ is primitive, we may assume that
$\aa_1$ has smaller phase than $\aa_2$ with respect to $\sigma_+$. 
By assumption on $\HH_\WW$ and by Theorem \ref{thm:nonempty}, the generic element $A_i \in M_{\sigma_+}(\aa_i)$ is
$\sigma_0$-\emph{stable} for $i= 1, 2$.
In particular, $\Hom(A_1, A_2) = \Hom(A_2, A_1) = 0$, and
therefore $\dim \Ext^1(A_2, A_1) = (\aa_1, \aa_2)$. By Lemma \ref{lem:stableextension}, any non-trivial
extension $A_1 \into E \onto A_2$ is $\sigma_+$-stable. Using Theorem \ref{thm:nonempty} again, one
computes the dimension of the space of such extensions as
\[
\aa_1^2 + 2 + \aa_2^2 + 2 + (\aa_1, \aa_2) - 1 = \vv^2 + 2 - \left((\aa_1, \aa_2) - 1\right).
\]
\end{proof}
For $P$ as above, the flopping contraction $\pi^+$ contracts $M_P$ to the product of moduli spaces
$M_{\sigma_0}^{st}(\aa_1) \times M_{\sigma_0}^{st}(\aa_2)$ of $\sigma_0$-\emph{stable} objects.
The exceptional locus of $\pi^+$ is the union of $M_P$ for all non-trivial partitions $P$.
In particular, when there is more than one two-element partition,
the stratification is only partially ordered, and the exceptional locus has multiple irreducible
components. This leads to a generalization of Markman's
notion of \emph{stratified Mukai flops} introduced in \cite{Markman:BNduality} (where the
stratification is indexed by a totally ordered set). 

Using the above two Lemmas, it is easy to construct examples of flops where the exceptional locus
of the small contraction $\pi^+$ has $m$ intersecting irreducible components:
\begin{Ex} \label{ex:flopirreduciblecomponents}
Choose $M \gg m$ for which $x^2 + Mxy + y^2 = -1$ does not admit an integral solution. 
We define the symmetric pairing on $\HH \cong \Z^2$ via the matrix
$\begin{pmatrix}2 & M \\ M & 2 \end{pmatrix}$, and let $\vv = \begin{pmatrix}1 \\ m-1\end{pmatrix}$.
The positive cone contains the upper right quadrant and is
bordered by lines of slopes approximately $-\frac 1M$ and $-M$. 
Since $M \gg m$ (in fact, $M > 2m$ is enough), any partition of $\vv$ into positive classes is in 
fact a partition in $\Z_{\ge 0}^2$. Therefore, the two-element partitions are given by
$A_k = \left[\begin{pmatrix}1 \\ k\end{pmatrix}, \begin{pmatrix}0 \\ m-1-k\end{pmatrix}\right] $ for
$0 \le k \le m-1$. There is a unique minimal partition
$Q = \left[ \begin{pmatrix}1 \\ 0\end{pmatrix}, \begin{pmatrix}0 \\ 1\end{pmatrix}, \dots, 
\begin{pmatrix}0 \\ 1\end{pmatrix} \right]$, with
$M_Q \subset \overline{M_{A_k}}$ for all $k$; thus, the exceptional locus has $m$ irreducible
components $\overline{M_{A_k}}$ intersecting in $M_Q$.
\end{Ex}

Similarly, one can construct flopping contractions with arbitrarily many
\emph{connected} components:
\begin{Ex}  \label{ex:flopconnecteccomponents}
Let $m$ be an odd positive integer. Choose $M \gg m$ and define the lattice $\HH$ by the matrix
$\begin{pmatrix}-4 & 2M \\ 2M & 4 \end{pmatrix}$. The positive cone lies between the lines of slope
approximately $+ \frac 1M$ and $-M$. We let $\vv = \begin{pmatrix} m \\ 2 \end{pmatrix}$. 
Any summand in a partition of $\vv$ must be of the form $\begin{pmatrix} x \\ y \end{pmatrix}$ 
with $x \ge 0 $ and $y > 0$, and therefore $y = 1$. Besides the trivial element, the only partitions
occurring in $\PPP$ are therefore of the form $A_k = \left[\begin{pmatrix} k \\ 1 \end{pmatrix},
\begin{pmatrix} m-k \\ 1 \end{pmatrix}\right]$, for $0 \le k < \frac m2$. Each corresponding stratum
$M_{A_k}$ is a connected component of the exceptional locus of $\pi^+$, as $A_k$ admits no further
refinement.
\end{Ex}

\begin{Rem}
To show that the lattices $\HH$ as above occur as the lattice $\HH_\WW$ associated to a wall,
we only have to find a K3 surface $X$ such that $\HH$ embeds primitively into its Mukai lattice $\Halg$. 
For example, we can choose $\mathrm{Pic}(X) \cong \HH$  and
$\vv = (0, c, 0)$ for the corresponding curve class $c$. In particular, Example
\ref{ex:flopirreduciblecomponents} occurs in a relative Jacobian of curves on special double covers
$X \to \P^2$, and Example \ref{ex:flopconnecteccomponents} in special quartics $X \subset \P^3$.
This wall crossing already occurs for Gieseker stability with respect to a non-generic polarization
$H$. The morphism $\pi^+$ contracts sheaves supported on reducible curves $C = C_1 \cup C_2$ in the 
corresponding linear system; it forgets the gluing data at the intersection points $C_1 \cap C_2$.
The induced flop preserves the Lagrangian fibration given by the Beauville integrable
system.
\end{Rem}

\section{Le Potier's Strange Duality for isotropic classes}\label{sec:SD}

In this section, we will explain a relation of Theorem \ref{thm:SYZ} to Le Potier's Strange Duality
Conjecture for K3 surfaces.  We thank Dragos Oprea for pointing us to this application.

We first recall the basic construction from \cite{LePotier:StrangeDuality,MarianOprea:StrangeDualitySurvey}.
Let $(X,\alpha)$ be a twisted K3 surface and let $\sigma\in\Stab^\dag(X,\alpha)$ be a generic stability condition.
Let $\vv,\ww\in H^*_{\alg}(X,\alpha,\Z)$ be primitive Mukai vectors with $\vv^2,\ww^2\geq0$.
We denote by $L_\ww$ (resp., $L_\vv$) the line bundle $\OO_{M_\sigma(\vv)}(-\theta_\vv(\ww))$ (resp., $\OO_{M_\sigma(\ww)}(-\theta_\ww(\vv))$).
We assume:
\begin{enumerate}
\item[(I)] $(\vv,\ww)=0$, and
\item[(II)] \label{assumII} for all $E\in M_\sigma(\vv)$ and all $F\in M_\sigma(\ww)$, $\Hom^2(E,F)=0$.
\end{enumerate}
Then the locus
\[
\Theta = \left\{ (E,F)\in M_\sigma(\vv) \times M_\sigma(\ww)\,:\, \Hom(E,F)\neq0 \right\}
\]
gives rise to a section of the line bundle $L_{\vv,\ww} := L_\ww \boxtimes L_\vv$ on $M_\sigma(\vv)
\times M_\sigma(\ww)$ (which may or may not vanish).
We then obtain a morphism, well-defined up to scalars,
\[
\mathrm{SD}\colon H^0(M_\sigma(\vv),L_\ww)^\vee \longrightarrow H^0(M_\sigma(\ww),L_\vv).
\]
The two basic questions are:
\begin{itemize}
\item When is $h^0(M_\sigma(\vv),L_\ww)=h^0(M_\sigma(\ww),L_\vv)$?
\item If equality holds, is the map $\mathrm{SD}$ an isomorphism?
\end{itemize}

We answer the two previous questions in the case where one of the two vectors is isotropic:

\begin{Prop}\label{prop:StrangeDuality}
Let $(X,\alpha)$ be a twisted K3 surface and let $\sigma\in\Stab^\dag(X,\alpha)$ be a generic stability condition.
Let $\vv,\ww\in H^*_{\alg}(X,\alpha,\Z)$ be primitive Mukai vectors with $(\vv, \ww) = 0$,
$\vv^2\geq2$ and $\ww^2=0$.

We assume that
$-\theta_\vv(\ww)\in \mathrm{Mov}(M_\sigma(\vv))$ and $-\theta_\ww(\vv)\in \mathrm{Nef}(M_\sigma(\ww))$.
Then
\begin{enumerate}
\item\label{enum:EqualityDim} $h^0(M_\sigma(\vv),L_\ww)=h^0(M_\sigma(\ww),L_\vv)$, and 
\item\label{enum:SD} the morphism $\mathrm{SD}$ is either zero or an isomorphism.
\end{enumerate}
\end{Prop}
We will see that the case $\mathrm{SD} = 0$ is caused by totally semistable walls.

\begin{proof}
Let $Y:= M_{\sigma}(\ww)$.
By \cite{Mukai:BundlesK3, Caldararu:NonFineModuliSpaces,Yoshioka:TwistedStability}, there exist an
element $\alpha'\in\mathrm{Br}(Y)$ and a derived equivalence $\Phi\colon \Db(X,\alpha)\xrightarrow{\simeq}\Db(Y,\alpha')$.
Replacing $(X, \alpha)$ by $(Y, \alpha')$, we may
assume that $\ww=(0,0,1)$ and $\vv=(0,D,s)$, for some $s\in\Z$ and $D\in\mathrm{NS}(X)$, and that
$X = M_{\sigma}(\ww)$ is the moduli space of skyscraper sheaves.
Moreover, $D=-\theta_\ww(\vv)\in \mathrm{Nef}(X)$ is effective, by assumption.
By stability and Serre duality, for all $E\in M_\sigma(\vv)$ and all $x\in X$,
$\Hom^2(E,k(x))=\Hom(k(x),E)^\vee=0$, verifying the assumption (II); thus the locus $\Theta$ gives a section of $L_\ww \boxtimes L_\vv$.

By Remark \ref{Rem:SYZcon_movable}, there exists a chamber $\LL_{\infty}$ in the interior of the movable cone $\mathrm{Mov}(M_\sigma(\vv))$ whose boundary contains $-\theta_\vv(\ww)$.
Moreover, there exist a polarization $H$ on $X$ and a chamber $\CC_{\infty}\subset \Stab^\dag(X,\alpha)$ such that $\ell(\CC_{\infty})=\LL_{\infty}$, $M_H(\vv)=M_{\CC_{\infty}}(\vv)$, and the Lagrangian fibration induced by $\ww$ is the Beauville integrable system on $M_H(\vv)$.

The argument in \cite[Example 8]{MarianOprea:StrangeDualitySurvey} shows that $h^0(M_{H}(\vv),L_\ww)=h^0(X,\OO(D))$ and the morphism $\mathrm{SD}$ is an isomorphism.
Since $M_H(\vv)$ is connected to $M_{\sigma}(\vv)$ by a sequence of flops, which do not change the dimension of the spaces of sections of $L_\ww$, we obtain immediately \eqref{enum:EqualityDim}.

To prove \eqref{enum:SD}, we need to study the behavior of the morphism $\mathrm{SD}$ under wall-crossing.
We pick a stability condition $\sigma_{\infty}\in\CC_{\infty}$.
Both $\sigma$ and $\sigma_{\infty}$ belong to the open subset $U(X,\alpha)$ of Theorem \ref{thm:BridgelandK3geometric}.
The restriction of the map $\ZZ$ of Theorem \ref{thm:Bridgeland_coveringmap} to $U(X, \alpha)$
is injective up to the $\widetilde \GL_2^+(\R)$-action (i.e., the map separates points that are in different
orbits). Now consider the map $\ell$ in the formulation of Theorem \ref{thm:ellandgroupaction},
restricted to $U(X, \alpha)$. The composition 
\[
\theta_{\sigma,\vv} \circ I \circ \ZZ|_{U(X, \alpha)} \colon U(X, \alpha) \to \Pos(M_{\sigma}(\vv))
\]
generically has connected fibers. Since both $\sigma_\infty$ and $\sigma$ get mapped to a class in
the movable cone, we can find 
 a path $\gamma$ in $U(X,\alpha)$ connecting $\sigma$ and $\sigma_{\infty}$ whose image stays within the
movable cone. Thus $\gamma$ crosses no divisorial walls.
If $\gamma$ also crosses no totally semistable walls, then the morphism $\mathrm{SD}$ is compatible with the wall-crossing; since it induces an isomorphism at $\sigma_{\infty}$, it induces an isomorphism at $\sigma$.

Assume instead that there is a totally semistable wall.
We write $\sigma=\sigma_{\omega,\beta}$.
The straight path from $\sigma_{\infty}$ to $\sigma_{t\omega,\beta}$, for $t\gg0$, corresponds to a change of polarization for Gieseker stability, and thus does not cross any totally semistable wall.
Therefore, we may replace $\sigma_{\infty}$ with $\sigma_{t\omega,\beta}$, for $t\gg0$.

We claim that all objects $E$ in $M_{\sigma}(\vv)$ must be actual complexes.
Indeed, if there exists a sheaf $E$ in $M_{\sigma}(\vv)$, then the generic element is a sheaf.
Moreover, since $D$ is nef and big, it is globally generated, and we can assume that the support of $E$ is a smooth integral curve.
Stability in $U(X,\alpha)$ for torsion sheaves implies, in particular, that the sheaf is actually stable on the curve.
But then $E$ would be stable for $t\to\infty$.
This shows that we crossed no totally semistable wall.

So $E \in \AA_{\omega, \beta}$ is an actual complex. Since $\rk (E) = 0$ and $\rk
\HH^{-1}(E) > 0$, we must have $\rk \HH^0(E) > 0$;  hence
$\Hom(E,k(x))\neq0$ for all $x\in X$.
This shows that $\Theta$ is nothing but the zero-section of $L_{\vv,\ww}$ and the induced map $\mathrm{SD}$ is the zero map.
\end{proof}

In particular, the previous proposition holds for pairs of Gieseker moduli spaces.

\begin{Ex}
Let $X$ be a K3 surface such that $\mathrm{Pic}(X) = \Z \cdot H$, with $H^2=2$.
Let $\vv=(1,0,-1)$ and $\ww=-(1,-H,1)$.
Consider a stability condition $\sigma_{\infty}=\sigma_{tH,-2H}$, for $t\gg0$.
Then, as observed in \cite[Proposition 1.3]{Beauville:RationalCurvesK3}, $\mathrm{Hilb}^2(X)= M_{\sigma_{\infty}}(\vv)$ admits a flop to a Lagrangian fibration induced by the vector $\ww$.
The assumptions of Proposition \ref{prop:StrangeDuality} are satisfied.
In this case, for all $E[1]\in M_{\sigma_{\infty}}(\ww)$, $E\cong I_{pt}(-H)$, and for all $\Gamma\in \mathrm{Hilb}^2(X)$, we have $\Hom(I_\Gamma,E[1])\neq0$.
Hence, the map $\mathrm{SD}$ is the zero map.
\end{Ex}

The following example shows that the assumption in Proposition \ref{prop:StrangeDuality} is
necessary:
\begin{Ex}
Let $X$ be a K3 surface with $\mathrm{NS}(X)=\Z \cdot C_1 \oplus \Z \cdot C_2$ and intersection form
\[
q =
\begin{pmatrix}
-2 & 4\\
4 & -2
\end{pmatrix}.
\]
We assume the two rational curves $C_1$ and $C_2$ generate the cone of effective divisors on $X$.
Let $\vv=(0,3C_1+C_2,1)$ and $\ww=(0,0,1)$.
Then $\vv^2=4$.
Pick a generic ample divisor $H$ on $X$.
We have
\[
H^0(M_H(\vv),\theta_{\vv}(\ww))\cong \C^{\oplus 4}.
\]
For example, consider the totally semistable wall where $\vv$ aligns with the spherical vector
$(0,C_1,0)$, 
Then Proposition \ref{prop:sphericaltotallysemistable} induces a birational map
$M_H(\vv)\dashrightarrow M_H(\vv_0)$ for $\vv_0=(0,C_1+C_2,1)$, and a chain of isomorphisms
\[
H^0(M_H(\vv),\theta_{\vv}(\ww))\cong H^0(M_H(\vv_0),\theta_{\vv_0}(\ww)) \cong H^0(\P^3,\OO_{\P^3}(1))
\cong \C^{\oplus 4},
\]
where the middle isomorphism follows from Proposition \ref{prop:StrangeDuality}.
However,
\[
H^0(M_H(\ww),\theta_\ww(\vv))\cong H^0(X,\OO_X(3C_1+C_2))\cong\C^{\oplus 5}.
\]
The last isomorphism follows from the exact sequence
\[
0\to \OO_X(2C_1+C_2) \to \OO_X(3C_1+C_2) \to \OO_{\P^1}(-2) \to 0,
\]
since $\OO_X(2C_1+C_2)$ is big and nef and thus has no higher cohomology.
\end{Ex}

\bibliography{all}                      
\bibliographystyle{alphaspecial}     

\end{document}